\documentclass[11pt]{article}

\usepackage{amsmath,amssymb,amsthm, amsfonts}
\usepackage[margin = 1in]{geometry}

\usepackage{hyperref}
\hypersetup{colorlinks=true,citecolor=blue, linkcolor=red, urlcolor=blue}

\usepackage[T1]{fontenc}
\usepackage{libertine}

\usepackage{mathrsfs}
\usepackage{centernot}
\usepackage{caption}
\usepackage{subcaption}
\usepackage{tikz}
\usetikzlibrary{decorations.pathmorphing}

\usepackage{verbatim}

\def\L{\Lambda}

\def\b{\partial}

\def\bt{\partial_{\textsc{n}}}
\def\bl{\partial_{\textsc{w}}}
\def\br{\partial_{\textsc{e}}}

\def\PHB{P_\textsc{hb}}
\def\PMHB{P_\textsc{mhb}}
\def\QMHB{Q_\textsc{mhb}}

\def\Z{\mathbb{Z}}
\def\P{\mathbb{P}}

\def\integers{\mathbb{Z}}

\def\R{\mathbb{R}}
\def\N{\mathbb{N}}
\def\E{\mathrm{E}}
\def\B{\mathcal{B}}
\def\Var{\mathrm{Var}}
\def\e{\mathrm{e}}

\def\Tcoup{T_{\rm coup}}

\newcommand{\tightoverset}[2]{%
	\mathop{#2}\limits^{\vbox to -.5ex{\kern-1.37ex\hbox{$#1$}\vss}}}

\def\L{\Lambda}
\def\b{\partial}

\def\PHB{P_\textsc{fk}}
\def\PMHB{P_\textsc{mhb}}
\def\MHB{\textsc{mhb}}
\def\Z{\mathbb{Z}}
\def\R{\mathbb{R}}
\def\E{\mathrm{E}}
\def\Var{\mathrm{Var}}
\def\e{\mathrm{e}}

\newcommand{\llb}{[\![}
\newcommand{\rrb}{]\!]}

\newcommand{\north}{{\textsc{n}}}
\newcommand{\south}{{\textsc{s}}}
\newcommand{\east}{{\textsc{e}}}
\newcommand{\west}{{\textsc{w}}}

\newcommand{\ext}{\textsc{ext}}
\newcommand{\interior}{\textsc{int}}

\newcommand{\tmix}{t_{\textsc{mix}}}
\newcommand{\gap}{\mathrm{gap}}
\newcommand{\tv}{{\textsc{tv}}}

\usepackage{bbm}

\newcommand{\inner}[3]{\langle #1 , #2 \rangle_{#3}}

\newtheorem{lemma}{Lemma}[section]
\newtheorem{theorem}[lemma]{Theorem}

\newtheorem{proposition}[lemma]{Proposition}
\newtheorem{claim}[lemma]{Claim}
\newtheorem{fact}[lemma]{Fact}

\theoremstyle{remark}
\newtheorem{remark}{Remark}
\theoremstyle{definition}
\newtheorem{definition}[lemma]{Definition}

\title{Random-cluster dynamics in~$\Z^2$: rapid mixing with general boundary conditions}

\author{
	Antonio Blanca\thanks{School of Computer Science, Georgia Tech, Atlanta, GA 30332.
		Email: {\textrm \{ablanca,vigoda\}@cc.gatech.edu}.}
	\and
	Reza Gheissari \thanks{Courant Institute of Mathematical Sciences, New York University, New York, NY 10003. Email: {\textrm reza@cims.nyu.edu}}
	\and
	Eric Vigoda$^*$
	}

\begin{document}
	
	\maketitle	
\vspace{-.6cm}
\begin{abstract} 
\noindent The random-cluster model with parameters $(p,q)$ is a random graph model that generalizes bond percolation ($q=1$) and the Ising and Potts models ($q\geq 2$). We study its Glauber dynamics on $n\times n$ boxes $\L_{n}$ of the integer lattice graph $\mathbb Z^2$, where the model exhibits a sharp phase transition at $p=p_c(q)$. Unlike traditional spin systems like the Ising and Potts models, the random-cluster model has non-local interactions. 
	Long-range interactions can be imposed as external connections in the boundary of $\L_n$, known as \textit{boundary conditions}.
	For select boundary conditions that do not carry long-range information (namely, wired and free), Blanca and Sinclair proved that when $q>1$ and $p\neq p_c(q)$, the Glauber dynamics on $\L_n$ mixes in optimal $O(n^2 \log n)$ time. In this paper, we prove that 
	this mixing time is polynomial in $n$
	for every boundary condition that is \emph{realizable} as a configuration on $\mathbb Z^2 \setminus \L_{n}$.
	We then use this to prove near-optimal $\tilde O(n^2)$ mixing time for ``typical'' boundary conditions. As a complementary result, we construct classes of non-realizable (non-planar) boundary conditions inducing slow (stretched-exponential) mixing at~$p\ll p_c(q)$. 
\end{abstract}

\section{Introduction}

The \textit{random-cluster model} \cite{FK} is a well-studied 
generalization of independent bond percolation, with connections to the study of electrical networks and random spanning trees~\cite{Grimmett}.
Perhaps most importantly, the random-cluster model is closely related to the \textit{ferromagnetic Ising/Potts model},
which is the classical model for spin systems in statistical physics. In fact, the random-cluster model is often referred to as the FK-representation of this model
and is 
a key tool in the analysis of its phase transition; see, e.g., the recent breakthroughs on the infinite 2-dimensional
integer lattice graph $\integers^2$~\cite{BD1,DST,DGHMT}. It also plays an indispensable role in
the design of efficient
Markov Chain Monte Carlo (MCMC) algorithms, like the Swendsen-Wang algorithm~\cite{SW}, for studying the Ising/Potts model.

For a graph $G=(V,E)$ and parameters $p\in (0,1)$ and $q>0$, random-cluster configurations are subsets of edges in $\Omega = \{S\subseteq E\}$, with the probability of $S\subset E$ given by
\begin{align}\label{eq:pi}
\pi_{G,p,q}(S) = \frac 1Z p^{|S|}(1-p)^{|E\setminus S|}q^{c(S)}\,,
\end{align}
where $c(S)$ is the number of connected components (including isolated vertices) in the subgraph 
$(V,S)$, and the \textit{partition function} $Z=Z_{p,q}$ is the normalizing 
constant that makes $\pi_{G,p,q}$ a probability measure.  

In addition to being interesting in its own right as a random graph and dependent percolation model, the random-cluster model provides a unifying framework for the study
of several important probabilistic models. 
For example, for \textit{integer} $q \ge 2$ the random-cluster model is dual, in a precise sense, to the $q$-state Ising/Potts model, where
configurations are assignments of spin values $\{1,\dots,q\}$ to the vertices of $G$. Each configuration $\sigma \in [q]^{V}$ has probability $\mu_G(\sigma)\propto \exp(\beta H(\sigma))$,
where $H(\sigma)$ is the number of edges connecting vertices with the same spin values,
and $\beta > 0$ is a model parameter associated with the inverse temperature of the system.
When $\beta=-\ln(1-p) > 0$, 
it is straightforward to check 
via a probabilistic coupling \cite{ES} that 
correlations in the Ising/Potts model correspond to connectivities 
in the random-cluster setting; this has illuminated much of the study of these models
and 
has made 
the random-cluster model
an accepted generalization of the Ising/Potts model to non-integer values of~$q$.

The random-cluster model, however, is not a spin system in the usual sense, as 
the weight of a configuration $S$
is not a function of local interactions between edges in $G$, but instead of global connectivity properties.  
This non-local structure is a crucial feature of the model but significantly complicates its analysis. The present paper studies the influence of this non-locality on the speed of convergence of the Glauber dynamics (i.e., local Markov chains) on subsets of $\mathbb Z^2$, where the model has been of particular interest. 

On the infinite graph $\mathbb Z^2$, both the random-cluster and the Ising/Potts model undergo phase transitions
corresponding to the sudden
emergence of long-range correlations as some parameter of the system is continuously varied  ($p$ and $\beta$ in this case).
A classical result of Onsager \cite{Onsager} established that the Ising model (the $q=2$ case) undergoes
a phase transition at a critical $\beta = \beta_c(2) = \ln(1+\sqrt{2})$. 
Recently it was shown in a celebrated result
of Beffara and Duminil-Copin \cite{BD1} that the Potts model ($q\geq 3$ integer),
and more generally the random-cluster model for any $q>1$, 
also undergo phase transitions at the critical points $\beta_c(q) = \ln(1+\sqrt{q})$ and $p_c(q) = \sqrt{q}/(\sqrt{q}+1)$, respectively. Note that $\beta_c(q) = -\ln(1-p_c(q))$.
These phase transitions can also be understood as transitions in the number of (unique vs.\ multiple) 
infinite-volume Gibbs measures on $\integers^2$.  
The infinite-volume measures on $\Z^2$ are obtained by taking limits of the distributions
on finite boxes with a sequence of ``boundary conditions'', described below.

We begin with the notion of boundary conditions in the Ising/Potts model.
Let $\L_n$ be an $n\times n$ square region of $\Z^2$ with nearest-neighbor edges $E(\L_n)$, and let 
$\b\L_n$ be its (inner) boundary (i.e., 
those vertices in $\L_n$ that are adjacent to vertices in $\Z^2 \setminus \L_n$).
An Ising/Potts boundary condition $\tau$ is a fixed assignment of spins to $\b\L_n$, and 
$\mu^\tau_{\L_n}$ is the Gibbs distribution on $\L_n$ conditional on the  
assignment $\tau$ to $\b\L_n$. (Since the interactions are nearest-neighbor, this is the same as conditioning on a configuration on all of $\mathbb Z^2 \setminus \L_n$.) 

For the random-cluster model on $\L_n$, a boundary condition $\xi$ on $\b \L_n$ is  
a partition $\{\xi_1,\xi_2,...\}$ of the boundary vertices such that all vertices in $\xi_i$ are always in the same connected component of a configuration $S$ via ``ghost'' (or external) wirings; these connections are considered in the counting of $c(S)$ in~(\ref{eq:pi}) and can therefore impose highly non-local interactions. 
Of particular interest are boundary conditions for the random-cluster model corresponding
to configurations on $\Z^2 \setminus \L_n$: i.e., where the boundary partition is induced by the connections of a random-cluster configuration on $E(\mathbb Z^2)\setminus E(\L_n)$. We call such boundary conditions {\em 
	realizable}.  (In fact, many works, including the standard text~\cite{Grimmett}, often restrict attention to 
realizable boundary conditions.) 
We note that boundary conditions are fundamental to the study of static and dynamic properties of the random-cluster and Ising/Potts models, especially in finite subsets of $\Z^2$.

In this paper we consider 
the Glauber dynamics for the random-cluster model on
the finite subgraph $(\L_n,E(\L_n))$ of $\Z^2$, in the presence of boundary conditions. 
This Markov chain
is of significant interest: it provides a simple Markov chain Monte Carlo (MCMC) algorithm for sampling
configurations of the system; and is, in many cases, a plausible model for the evolution
of the underlying system.

Specifically, we consider
the following discrete-time variant of the Glauber dynamics chain, which we refer to as the {\em FK-dynamics}.
For $t\in \mathbb N$, from $S_t\subseteq E(\L_n)$, transition to $S_{t+1}\subseteq E(\L_n)$ as follows:
\begin{enumerate}
	\item Choose an edge $e\in E(\L_n)$ uniformly at random;
	\item let $S_{t+1} = S_t \cup \{e\}$ with probability
	$$
	\frac{\pi_{\L_n,p,q}(S_t \cup \{e\})}{\pi_{\L_n,p,q}(S_t \cup \{e\})+\pi_{\L_n,p,q}(S_t \setminus \{e\})} = \left\{\begin{array}{ll}
	\frac{p}{q(1-p)+p} & \mbox{if $e$ is a ``cut-edge'' in $(\L_n,S_t)$;} \\
	p & \mbox{otherwise;}
	\end{array}\right.
	$$
	\item else let $S_{t+1} = S_t \setminus \{e\}$.
\end{enumerate}
We say $e$ is a {\it cut-edge} in $(\L_n,S_t)$ if the number of connected components in $S_t\cup \{e\}$ and $S_t \setminus \{e\}$ differ.  Under a boundary condition $\xi$, the property of $e$ being a cut-edge is defined with respect to the augmented graph $(\L_n,S^{\xi}_t)$, where $S^\xi_t$ adds external wirings between all pairs of vertices in the same element of $\xi$. 
This Markov chain converges to $\pi_{\L_n,p,q}$ by construction, and we study 
its speed of convergence.

A standard measure for quantifying the speed of convergence of a Markov chain is the \textit{mixing time}, which is defined as the time until the dynamics is close (in total variation 
distance) to its stationary distribution, starting from a worst-case initial state.
We say the dynamics is {\em rapidly mixing} if the mixing time is polynomial in $|V|$, and {\em torpidly mixing} when the mixing time is exponential 
in $|V|^\varepsilon$ for some~$\varepsilon>0$.

The corresponding Glauber dynamics for the Ising/Potts model (which updates spins one at a time according to the spins of their neighbors), is by now quite well understood on finite regions of $\mathbb Z^2$.  In the high-temperature region $\beta<\beta_c$ (corresponding to $p<p_c$)
the Glauber dynamics has 
optimal mixing time $\Theta(n^2\log{n})$ on boxes $\L_n$~\cite{MOS,Cesi,BD1,Alexander}; moreover, this same
asymptotic bound of $\Theta(n^2\log{n})$ holds for every fixed boundary condition.
These bounds follow as a consequence of the exponential decay of correlations
of the model in the high-temperature regime, which 
holds even near the boundary for arbitrary Ising/Potts boundary conditions; this property is known as \emph{strong spatial mixing}. 
In the low-temperature region $\beta>\beta_c$
the mixing time is exponential in $n$ for free and periodic (toroidal) boundaries~\cite{Thomas, BCT, GL1}.
The more general problem of understanding the mixing time of the Glauber dynamics
for other boundary conditions
at low temperatures is a long-standing open problem, e.g., see~\cite{LMST,MaTo}.

The FK-dynamics is quite powerful 
since
the self-duality of the model on $\mathbb Z^2$ implies that 
it is rapidly mixing in the low-temperature regime where
the Ising/Potts Glauber dynamics is torpidly mixing. For the FK-dynamics on $\L_n$,~\cite{BS}
showed that the mixing time is $\Theta(n^2\log{n})$ for all $q>1$ whenever $p\neq p_c(q)$;
see also~\cite{GS} for recent results concerning the cutoff phenomenon in the FK-dynamics. 
(At the critical $p=p_c(q)$ 
the FK-dynamics
may exhibit torpid mixing depending on the
``order'' of the phase transition~\cite{GL1,GL2}.)
Since the proof in~\cite{BS} (as well as in~\cite{GS}) used a strong spatial mixing property for the random-cluster model, the $\Theta(n^2 \log n)$ upper bound 
only holds under boundary conditions that are free (no boundary condition), wired (all boundary vertices are connected to one another) or periodic (the torus). Note, crucially, that these boundary conditions do not carry information about random-cluster connectivities in non-local ways: namely, configurations in different regions of $\L_n$ do not interact through these boundaries. The behavior of the FK-dynamics under arbitrary random-cluster boundary conditions remained unclear.

We prove that the FK-dynamics at $p\neq p_c(q)$ is rapidly mixing for all realizable boundary conditions.

\begin{theorem}\label{thm:planar-mixing:intro}
	For every $q> 1$, $p\neq p_c(q)$, there exists a constant $C>0$ such that the mixing time of the FK-dynamics on the $n\times n$ box $\L_n\subset \integers^2$ with any realizable boundary condition is~$O(n^{C})$.
\end{theorem}
We pause to comment on the proof of Theorem~\ref{thm:planar-mixing:intro}. As mentioned above, proofs of fast mixing when $p\neq p_c(q)$ have relied crucially on a strong spatial mixing property, which in the random-cluster model, would say that correlations between edges (even near the boundary $\partial \L_n$) decay exponentially in the graph distance between them. It is easy to construct examples of realizable boundary conditions where this correlation does not decay at all, even if $p\ll p_c(q)$, as the boundary can enforce long-range interactions. 
Since the exponential decay of correlations does hold for edges in the ``bulk'' of $\L_n$ (i.e, at distance $\Theta(\log n)$ away from its boundary), 
we are able to reduce the proof of Theorem~\ref{thm:planar-mixing:intro} to proving a polynomial upper bound 
for the mixing time of the FK-dynamics
on thin rectangles 
of dimension $n\times \Theta(\log n)$
with realizable boundary conditions. 
This will be the key technical difficulty for us and is established in Theorem~\ref{thm:main-thin-rect}. 

In the setting of spin systems, or boundary conditions that do not encode long-range interactions, a polynomial upper bound on $n \times \Theta(\log n)$ rectangles would follow from standard canonical paths arguments~\cite{Martinelli-SP,Martinelli-notes,JS,Sinclair}. However, even realizable boundary conditions can heavily distort the graph with external wirings, preventing this approach from succeeding.
Instead, to prove Theorem~\ref{thm:main-thin-rect} we devise a novel application of a recursive 
(block dynamics) 
scheme
common in the analysis of spin systems. 
In lieu of splitting rectangles into two overlapping sub-rectangles, e.g., the first two-thirds and the second two-thirds, as is done for spin systems, our choice of smaller subsets at every step of the recursion is delicately dictated by the boundary conditions. As a consequence, the subsets we recurse over are no longer restricted to rectangles, but can be arbitrary ``groups of rectangles'' that are ``compatible'' with the realizable boundary conditions. We point the reader to Section~\ref{sec:thin-rect} for a more detailed proof overview.

Theorem~\ref{thm:planar-mixing:intro} shows a polynomial upper bound on the mixing time, uniformly over all
realizable boundary conditions.  Utilizing this theorem
we prove near-optimal $O(|V|(\log{|V|})^C)$ mixing time for ``typical'' boundaries as we detail now. The notion of typicality should be understood as with high probability under some probability distribution over realizable boundary conditions, with a natural choice being the marginal distribution of the infinite random-cluster measure $\pi_{\Z^2,p,q}$ on $\mathbb Z^2 \setminus \L_n$. 

As mentioned earlier, a key obstacle to proving mixing time upper bounds are long, distinct boundary connections that enforce long-range correlations. 
For any realizable boundary condition $\xi$ corresponding to a partition $\{\xi_1,\xi_2,...\}$ of $\b \L_n$, let $L(\xi_i)$ be the smallest connected subgraph of $\b \L_n$ containing all vertices in $\xi_i$. 
The key class of boundary conditions we consider are those where all the $L(\xi_i)$
are small.

\begin{definition}\label{def:a-localized}
	We say a boundary condition $\xi$ on $\b \L_n$ with corresponding partition $\{\xi_1,\xi_2,...\}$ is in $\mathcal C_\alpha$ if $\max_i |L(\xi_i)|\leq \alpha \log n$.  We say that a realizable boundary condition $\xi$ is in $\mathcal C_\alpha^\star$ if its dual boundary condition $\xi^\star$ is in $\mathcal C_\alpha$; see Section \ref{sec:prelim} for the definition of dual configuration. 
	We refer to the classes $\mathcal C_\alpha$ and $\mathcal C_\alpha^\star$ as {\em $\alpha$-localized} boundaries. \end{definition}

\noindent It is straightforward to see that if one samples a ``random'' boundary condition from the infinite-volume measure $\pi_{\Z^2,p,q}$, then with high probability, the
induced boundary condition on $\partial \L_n$ is $\alpha$-localized for some $\alpha>0$.  Since there
is a unique random-cluster measure on $\integers^2$ when $p\neq p_c(q)$,
this is well-defined.  

\begin{theorem}\label{thm:random-boundary}
	For all $q>1$, $p \neq p_c(q)$, let $\omega$ be a random-cluster configuration sampled from $\pi_{\Z^2,p,q}$. 
	If $\xi_\omega$ is the boundary condition on $\b\L_n$
	induced by the connections of $\omega$ in $E(\Z^2) \setminus E(\L_n)$,
	then,  with probability $1-o(1)$,
	$\xi_\omega$ is $\alpha$-localized for $\alpha>0$ sufficiently large.
\end{theorem}
There are other similar ways of defining typicality of realizable boundary conditions that may also be of interest. For example, if $\omega$ were sampled from the random-cluster measure in a large concentric box containing $\L_n$, with probability $1-o(1)$, $\xi_\omega$ would again be $\alpha$-localized; see Remark~\ref{rem:finite-typicality}.

We are able to prove nearly-optimal mixing time $\tilde{O}(n^2)$ for $\alpha$-localized boundaries, and 
hence, with high probability, for a random boundary sampled from \emph{any} off-critical random-cluster measure.

\begin{theorem}\label{thm:typical-mixing:intro}
	For every $q>1$, $p < p_c(q)$ (resp., $p>p_c(q)$), and every $\alpha>0$, there exists a constant $C>0$ such that for every realizable boundary condition $\xi\in \mathcal C_\alpha$ (resp., $\xi\in \mathcal C_\alpha^\star$) on $\partial \L_n$, the mixing time of the FK-dynamics on the $n\times n$ box 
	$\L_n$ with boundary condition $\xi$ is $O(n^2(\log{n})^C)$.
\end{theorem}

The proof of Theorem~\ref{thm:typical-mixing:intro} uses Theorem~\ref{thm:planar-mixing:intro} in a crucial way. 
Typical boundary conditions do not exhibit the strong spatial mixing property from \cite{BS}; however, for boundary conditions in $C_\alpha$ we are able to prove that correlations between edges near the boundary decay exponentially in their graph distance divided by $\alpha\log n$. 
Using this correlation bound, together with the aforementioned general framework in Section \ref{sec:general} to derive mixing time estimates from spatial mixing properties,
we reduce bounding the mixing time on $\L_n$ with typical boundaries to bounding the mixing time on $\Theta((\log n)^2)\times \Theta((\log n)^2)$ rectangles with arbitrary realizable boundary conditions. Theorem~\ref{thm:planar-mixing:intro} then implies that the mixing time of the FK-dynamics in these smaller rectangles is at most poly-logarithmic in $n$. 
Similar classes of typical boundary conditions were considered in~\cite{GL1,GL2} at $p=p_c(q)$; there, comparison methods were used to disregard the influence of long boundary connections at the expense of super-polynomial factors in the mixing time.

Given that our rapid mixing result for realizable boundaries  relies heavily on the planarity of the boundary connections in $\mathbb Z^2\setminus \L_n$, one may wonder whether rapid mixing holds for all possible FK boundary conditions (including those not realizable as configurations on $\mathbb Z^2 \setminus \L_n$).
We answer this in the negative, showing that there exist (non-realizable) boundaries
for which the FK-dynamics is torpidly mixing even while $p\neq p_c(q)$.
In fact, this torpid mixing holds at $p\ll p_c(q)$,
which may sound especially surprising as correlations in the Gibbs measure $\pi_{\L_n,p,q}$ die off faster as $p$ decreases. 
\begin{theorem}
	\label{thm:lb:intro}
	Let $q>2$. For every $\alpha \in (0,\frac 12]$ and $\lambda > 0$ there exists a boundary condition $\xi$, 
	such that when $p = \lambda n^{-\alpha}$ the mixing time of the FK-dynamics
	on the $n\times n$ box $\L_n$ with boundary condition $\xi$
	is~$
	\exp(\Omega(n^\alpha))
	$.
\end{theorem}

Our proof of this theorem is constructive: we take any graph $G$ 
on $m$ edges for which torpid mixing of the FK-dynamics is known at some value of $p(m)< p_c(q)$, and show how to embed $G$ into the boundary of $\L_n$. 
We then develop a procedure to transfer mixing time bounds from $G$ to $\L_n$.
The high-level idea is that for sufficiently small $p(m)$ 
the effect of the configuration away from the boundary is negligible, and so the mixing time of the FK-dynamics on $G$ completely governs the mixing time of FK-dynamics near the boundary $\b\L_n$.
We can then use known torpid mixing results for the mean-field random-cluster model (the case where $G$ is the complete graph) in its critical window at $q>2$~\cite{GoJe, BS-MF,GSV,GLP}.

We remark about the condition $q>2$ in Theorem~\ref{thm:lb:intro}.  In~\cite{GuoJer} it was shown that the mixing time of FK-dynamics when $q=2$ is at most polynomial in the number of vertices on \emph{any} graph and at every $p\in (0,1)$.
It is believed that this rapid mixing holds for all $q\leq 2$; hence 
the requirement $q>2$ appears to be sharp for Theorem~\ref{thm:lb:intro}.  
We believe that the above torpid mixing result may also extend to small, but $\Omega(1)$ values of~$p<p_c(q)$, though our current proof does not allow for this. In principle, one would want to embed a bounded degree graph into $\partial \L_n$, so that its critical point at which it exhibits slow mixing is $\Omega(1)$. Note that there are several examples of bounded degree graphs where torpid mixing is known~\cite{CF,BCFKVV,BCT,GL1,GL3}.

Finally, we remark that mixing times of the FK-dynamics on a graph $G$ 
are comparable, up to polynomial factors in $|E|$ to mixing times of the Chayes-Machta dynamics~\cite{CM}, 
as well as of the Swendsen-Wang algorithm that walks on the FK configurations of $G$ \cite{Ullrich1,Ullrich2,BS-MF}.
By slight adaptations of the comparison results in \cite{Ullrich1,Ullrich2,BS-MF}, our theorems thus provide upper and lower bounds for these non-local dynamics in the presence of FK boundary conditions.

The rest of the paper is organized as follows.
In Section~\ref{sec:prelim}, we formally define various preliminary notions that are used in our proofs. 
In Section~\ref{sec:general}, we introduce our general framework to deduce mixing time estimates on $\L_n$ from spatial and local mixing properties.
We then present  our key rapid mixing result 
for thin rectangles (Theorem~\ref{thm:main-thin-rect}) in Section~\ref{sec:thin-rect}, before completing the proof of Theorem~\ref{thm:planar-mixing:intro} in Section~\ref{section:planar}. This is boosted to nearly-optimal mixing time for typical boundaries (Theorem~\ref{thm:typical-mixing:intro}) in Section~\ref{section:typical}. Finally, the torpid mixing result (Theorem~\ref{thm:lb:intro})
is proved in Section~\ref{section:lb:intro}.

\medskip \noindent \textbf{Acknowledgements.} The authors thank Insuk Seo for detailed comments. The research of A.B.\ and E.V.\ was supported in part by NSF grants CCF-1617306 and CCF-1563838.

\section{Preliminaries: the random-cluster model in $\mathbb Z^2$}\label{sec:prelim}

In this section we introduce a number of definitions, notation, and background results that we will refer to repeatedly. More details and proofs can be found in the books \cite{Grimmett,LP}.
We will be considering the random-cluster model on rectangular subsets of $\mathbb Z^2$ of the form 
$$
\Lambda_{n,l} = \{0,...,n\} \times\{0,...,l\}= \llb 0,n\rrb \times \llb 0,l \rrb\,.
$$
When $n = l$, we use $\L_n$ for $\L_{n,n}$.
For simplicity,
in this preliminary section we shall focus on the $n=l$ case, but
everything stated here holds more generally for rectangular subsets with $n \neq l$.
Abusing notation, we will also use $\L_n$ for the graph $(\L_n, E(\L_n))$ where $E(\L_n)$ consists of all nearest neighbor pairs of vertices in $\L_n$.
We denote by $\b \L_n$ the (inner) boundary of $\L_n$; that is the vertex set consisting of all vertices in $\L_n$ adjacent to vertices in $\mathbb Z^2 \setminus \L_n$. The north, east, south, west,  boundaries of $\L_n$ will be delineated $\bt \L_n,\b_\east \L_n, \b_\south \L_n$ and $\b_\west \L_n$ respectively. 

A \emph{boundary condition} $\xi$ of $\L_n$ is a partition of the vertices in $\b \L_n$.  
When $u, v \in \b\L_n$ are in the same element of $\xi$, we say that they are \textit{wired} in $\xi$.
If there exists a random-cluster configuration $\omega$ on $E(\mathbb Z^2) \setminus E(\L_n)$ such that,
for all $u,v\in \partial \L_n$, $u$ and $v$ are in the same connected component of $\omega$ if and only if they are wired in $\xi$, then
we say that the boundary condition $\xi$ is \emph{realizable}.

\begin{remark}
	Realizable boundary conditions are the most natural class of boundary conditions (see~\cite{Grimmett}) since they enforce the planarity of $\mathbb Z^2$.
	However, non-realizable boundary conditions are still relevant in some cases; for example,
	when considering the random-cluster model on non-lattice graphs such as trees.
	In Section~\ref{section:lb:intro}, we consider the mixing time of the FK-dynamics under non-realizable boundary conditions.
\end{remark}

For $p\in (0,1)$ and $q>0$, the \emph{random-cluster model} on $\L_n$ with a boundary condition $\xi$ is the probability measure over the subsets $S\subseteq E(\L_n)$ given by
$$
\pi_{\L_n,p,q}^\xi (S) = \frac 1Z p^{|S|} (1-p)^{|E(\L_n)\setminus S|} q^{c(S;\xi)},
$$
where $c(S;\xi)$ corresponds to the number of connected components in the augmented graph $(V, S^{\xi})$ and $S^{\xi}$ adds auxiliary edges between all pairs of vertices in $\partial \L_n$ that are in the same element of $\xi$. 
Every subset $S\subseteq E(\L_n)$, can be naturally identified with some edge configuration $\omega:E(\L_n) \to \{0,1\}$ via $\omega(e)=1$ if $e\in S$ ($e$ is \emph{open}) and $\omega(e)=0$ if $e\notin S$ ($e$ is \emph{closed}). 
We sometimes interchange vertex sets with the subgraph they induce; e.g., the random-cluster configuration on a set $R \subset \Z^2$ corresponds to the random-cluster configuration in the subgraph induced by $R$.
We omit the subscripts $p,q$ when understood from context.   \hfill

\medskip  
\noindent \textbf{Monotonicity.} 
Define a partial order over boundary conditions by $\xi \leq \eta$ if the partition corresponding to $\xi$ is \emph{finer} than that of $\eta$. The extremal boundary conditions then, are the \emph{free} boundary where $\xi = \{\{v\}:v\in \partial \L_n\}$, which we denote by $\xi=0$, and the \emph{wired} boundary where $\xi = \{\partial \L_n\}$, denoted by $\xi =1$. 
When $q>1$, the random-cluster model satisfies the following monotonicity in boundary conditions: if $\xi, \eta$ are two boundary conditions on $\partial \L_n$ with $\xi \leq \eta$, then $\pi_{\L_n}^\xi  \preceq \pi_{\L_n}^{\eta}$, where $\preceq$ denotes stochastic domination.

\medskip\noindent
{\bf Planar duality.} Let $\L_n^*=(\L_n^*,E(\L_n^*))$ denote the planar dual of $\L_n$. That is, $\L_n^*$ corresponds to the set of faces of $\L_n$, and for each $e \in E(\L_n)$, there is a dual edge $e^* \in E(\L_n^*)$ connecting the two faces bordering $e$. 
The random-cluster distribution satisfies
$\pi_{\L_n,p,q}(S) = \pi_{\L_n^*,p^*,q}(S^*)$,
where $S^*$ is the dual configuration to $S \subseteq E$ (i.e., $e^* \in S^*$ iff $e \not\in S$), and
$$p^* = \frac{q(1-p)}{q(1-p)+p}\,.$$
Under a realizable boundary condition $\xi$, this distributional equality becomes $\pi_{\L_n,p,q}^{\xi}(S)=\pi_{\L_n^*,p^*,q^*}^{\xi^*}(S)$, where $\xi^*$ is the boundary condition induced by taking the dual configuration of the configuration on $\mathbb Z^2 \setminus \L_n$ identified with $\xi$. 
Notice that $\mathbb Z^2$ is isomorphic to its dual. 
The unique value of $p$ satisfying $p=p^*$, denoted $p_{sd}(q)$, is called the {\it self-dual point}.

\medskip\noindent
{\bf Infinite-volume measure and phase transition.}  A random-cluster measure $\pi_{\Z^2,p,q}$ can be defined on the infinite lattice $\Z^2$ as the limit as $n\to \infty$ of the sequence of random-cluster measures on $n\times n$ boxes with free boundary conditions. The measure $\pi_{\Z^2,p,q}$ exhibits a phase transition corresponding to the appearance of an infinite connected component. That is, there exists a critical value $p = p_c(q)$ such that if $p < p_c(q)$ (resp., $p > p_c(q)$), then, almost surely, all components are finite (resp., there exists an infinite component).
For $q \ge 1$, the exact value of $p_c(q)$ for $\Z^2$ was recently settled in~\cite{BD1}, proving
$$p_c(q) = p_{sd} (q) = \frac{\sqrt{q}}{\sqrt{q}+1}\,.$$

\medskip\noindent
\textbf{Exponential decay of connectivities (EDC).} 
A consequence of the results in \cite{Alexander,BD1} is that for
every $q>1$ and $p<p_c(q)$, there is a $c = c(p,q)>0$ such that for every boundary condition $\xi$ and all~$u,v\in \L_n$,
\begin{align}\label{eq:EDC}
\pi^{\xi}_{\L_n,p,q}(u\stackrel{\L_n}\longleftrightarrow v)\leq {\e}^{-c d(u,v)}\,,
\end{align}
where $d(u,v)$ is the graph distance between $u,v$ in $\mathbb Z^2$ and $u\stackrel{\L_n}\longleftrightarrow v$ denotes that there is an open path between $u$ and $v$ in the FK configuration on $E(\L_n)$ (not using the connections of $\xi$).

\subsection{Random-cluster dynamics}
\label{subsection:prelim:dynamics}
In this section we overview some preliminaries related to Markov chain mixing times and the FK dynamics.   

\medskip
\noindent \textbf{Mixing and coupling times.} 
Consider an ergodic (i.e., irreducible and aperiodic) Markov chain 
$\mathcal{M}$ with finite state space $\Omega$,
transition matrix $P$ and stationary distribution $\mu$. Let
$$
\tmix(\varepsilon) = \min \{t: \max_{X_0 \in \Omega} \| P^t(X_0,\cdot)  - \mu \|_\tv \leq \varepsilon\} 
$$
where $\|\cdot\|_\tv$ is the total-variation distance. The {\it mixing time} of $\mathcal{M}$ is given by $\tmix := \tmix(1/4)$, and for 
any positive $\varepsilon < 1/2$, by sub-multiplicativity, we have $\tmix(\varepsilon) \le \lceil \log_2 \varepsilon^{-1} \rceil\,\tmix$.  
We use $\tmix(\L_n^\xi)$ to denote the mixing time of the FK-dynamics on $\L_n\subset \integers^2$ with boundary condition $\xi$. 

A {\it (one step) coupling} of the Markov chain $\mathcal{M}$ specifies, for every pair of states $(X_t, Y_t) \in \Omega \times \Omega$, a probability distribution over $(X_{t+1}, Y_{t+1})$ such that the processes $\{X_t\}$ and $\{Y_t\}$, viewed in isolation, are faithful copies of $\mathcal{M}$, and if $X_t=Y_t$ then $X_{t+1}=Y_{t+1}$. The {\it coupling time}, denoted $\Tcoup$, is the minimum $T$ such that $\Pr[X_T \neq Y_T] \le 1/4$, starting from the worst possible pair of configurations $X_0$, $Y_0$. The following inequality is standard :
$\tmix \le \Tcoup$.

\medskip \noindent \textbf{Spectral gap and conductance.} 
If $P$ is irreducible and reversible with respect to $\mu$, then it has real eigenvalues $1 = \lambda_1 > \lambda_2 \ge \dots \ge \lambda_{|\Omega|} \ge -1$. 
The \emph{absolute spectral gap} of $P$
is defined by $\gap(P) = 1 - \lambda_*$ where $\lambda_* = \max\{|\lambda_2|,|\lambda_{|\Omega|}|\}$.
Let $ \mu_{\rm min} = \min_{\omega \in \Omega} \mu(\omega)$; the following is then a standard inequality: 
\begin{equation}
\label{eq:prelim:gap}
\gap(P)^{-1}-1 \leq \, \tmix \, \leq \gap(P)^{-1} \log (2\e \cdot \mu_{\rm min}^{-1})\,.
\end{equation}

For $A \subset \Omega$, the conductance of $A$ is defined as
\begin{equation}
	\label{eq:prelim:conductance-def}
	\Phi(A) = \frac{Q(A,A^c)}{\mu(A)} = \frac{\sum_{\omega \in A, \omega' \in A^c} \mu(\omega)P(\omega,\omega')}{\mu(A)}\,.
\end{equation}
The conductance of the chain is given by $\Phi_\star = \min_{A:\mu(A) \le 1/2} \Phi(A)$, and we have
\begin{equation}\label{eq:prelim:conductance}
\frac{\Phi_\star^2}{2} \le \gap(P) \le 2\Phi_\star\,.
\end{equation}

\medskip \noindent \textbf{FK-dynamics and duality.} 
Each run of the FK-dynamics on $\L_n$, with realizable boundary conditions $\xi$ and parameters $p,q$, 
determines a valid run of the
 FK-dynamics on the dual graph $\L_n^*$ with boundary conditions $\xi^*$ and parameters $p^*, q$.
 (Simply identify the FK configuration in each step with its dual configuration; it can be straightforwardly verified that the transitions of the  FK-dynamics on the dual graph occur with the correct probabilities.) 
Hence, the two dynamics have the same mixing times.

\begin{remark}\label{rem:fk-dynamics-duality}
The edge-set of the dual graph $\L_n^*$ is not exactly in correspondence with the edge-set of a rectangle $\L^*= \{-\frac 12,...,n+\frac 12\} \times \{-\frac 12 ,...,n+\frac 12\}$ as it does not include any edges that are between boundary vertices of $\L^*$. 
All the proofs in the paper carry through, only with the natural minor geometric modifications, to the case of rectangles $\Lambda_{n}$ with modified edge-set that only contains edges edges with at least one endpoint in $\Lambda_{n}\setminus \partial \Lambda_{n}$. The dual of this modified graph is then a $(n-1) \times (n-1)$ rectangle with all nearest-neighbor edges. 

We note that the definition of realizability of the boundary condition is slightly different depending on the above choice of edge-set for a rectangle: while $\xi$ is encoded in a configuration on $E(\mathbb Z^2)\setminus E(\L_n)$, its dual $\xi^*$ would be encoded in a configuration
on $E((\mathbb Z^2)^*)\setminus (E(\L^*)\setminus E(\partial \L^*))$. However, it is straightforward to check that these two notions of realizability are equivalent: a partition of $\partial \L_n$ can be encoded in a configuration on $E(\mathbb Z^2)\setminus E(\L_n)$ if and only if it can be encoded in a configuration on $E(\mathbb Z^2)\setminus (E(\L_n)\setminus E(\partial \L_n))$.
With these considerations, it often suffices for us to prove our theorems for $p<p_c(q)$. 
For example, it is sufficient to prove Theorem~\ref{thm:planar-mixing:intro} for~$p<p_c(q)$. 
\end{remark}

\medskip\noindent\textbf{Boundary condition modification.} Finally, we will often appeal to a comparison inequality between mixing times of the FK-dynamics under different boundary conditions.
This inequality is a consequence of a simple comparison between random-cluster measures.

\begin{definition}
For two partitions $\rho$ and $\rho'$ of the vertex set of a graph $G= (V,E)$ we say that $\rho \leq \rho'$ if $\rho$ is a finer partition than $\rho'$ (so that the partition consisting only of singletons is the minimal element of this partial order). Then for two partitions $\rho \leq \rho'$ we define $D(\rho,\rho') = c(\rho) - c(\rho')$ where $c(\rho)$ is the number of components in $\rho$. For two partitions $\rho, \rho'$ that are not comparable, let $\rho''$ be the smallest partition such that $\rho'' \geq \rho$ and $\rho'' \geq \rho'$ and set $D(\rho,\rho') = c(\rho) - c(\rho'')+ c(\rho')- c(\rho'')$. 
\end{definition}

\begin{lemma}
	\label{lemma:simple-rc-bound}
	Let $G=(V,E)$ be an arbitrary graph, $p \in (0,1)$ and $q > 0$. Let $\rho$ and $\rho'$ be two partitions of $V$ encoding two distinct external wirings on the vertices of $G$.
	Let $\pi_{G}^\rho$, $\pi_{G}^{\rho'}$ be the resulting random-cluster measures. Then, for all FK configurations $\omega\in \{0,1\}^E$, we have
	$$
	q^{-2D(\rho,\rho')} {\pi_{G}^{\rho'}(\omega)} \le \pi_{G}^{\rho}(\omega) \le q ^{2D(\rho,\rho')} \pi_{G}^{\rho'}(\omega)\,.
	$$  
\end{lemma}

\noindent With this in hand, the variational form of the spectral gap implies the following (see e.g.,~\cite{SC}). 

\begin{lemma}\label{lem:comparison-tmix}
	Let $G=(V,E)$ be an arbitrary graph, $p \in (0,1)$ and $q > 0$.
	Consider the FK-dynamics on $G$ with the external wirings $\rho$ and $\rho'$,
	and let $\tmix(G^\rho)$, $\tmix(G^{\rho'})$ denote their mixing times. Then:
	$$
	\tmix(G^\rho)\leq   q^{ O(D(\rho,\rho'))} |E| \tmix(G^{\rho'})\,.
	$$
\end{lemma}

\section{Mixing time upper bounds: a general framework}\label{sec:general}

In this section we introduce a general framework for bounding the mixing time of the FK-dynamics on $\L_n = (\L_n, E(\L_n))$ by its mixing times on certain subsets.
In \cite{BS} it was shown that a strong form of spatial mixing (encoding exponential decay of correlations uniformly over subsets of $\L_n$) 
implies optimal mixing of the FK-dynamics.
However, this notion, known as \textit{strong spatial mixing (SSM)} and described in Remark~\ref{rem:ssm}, 
does not hold for most
boundary conditions
for which
fast mixing of the FK-dynamics is still expected.
To circumvent this, we introduce a weaker notion, which we call \textit{moderate spatial mixing (MSM)}.

\medskip
\noindent \textbf{Notation.} We introduce some notation first.
For a set $R \subseteq \L_n$, let
$E(R) \subseteq E_{n}$ be the set of edges of $E(\L_n)$ with both endpoints in $R$. We will denote by $R^c$ the vertex set $\L_n \setminus R$ and by $E^c(R)$ the edge-complement of $R$; i.e., $E^c(R):=E(\L_n)\setminus E(R)$. For a configuration $\omega:E(\L_n)\to \{0,1\}$, we will use $\omega(R)$, or alternatively $\omega(E(R))$, for the configuration of $\omega$ on $E(R)$. 
With a slight abuse of notation, for an edge set $F\subseteq E(\L_n)$, we use  $\{F=\omega\}$ 
for the event that the configuration on $F$ is given by $\omega$; when $\omega$ is the all free or the all wired configuration, we simply use $\{F=0\}$ and $\{F=1\}$, respectively. 

\begin{definition}
	\label{def:msm}
	Let $\xi$ be a boundary condition for $\L_n =(\L_n, E(\L_n))$ and
	let $\mathcal{B} = \{B_1,B_2,\dots,B_k\}$ be a collection of subsets of $\L_n$.
	We say that \emph{moderate spatial mixing (MSM)} holds on $\L_n$ for $\xi$, $\B$ and $\delta > 0$
	if for all $e \in E(\L_n)$, there exists $B_j \in \mathcal{B}$ such that
	\begin{equation}
	\label{eq:ssm:def}
	\left|\pi^\xi_{\L_n,p,q}(\,e=1\mid E^c(B_j) = 1\,)-\pi^\xi_{\L_n,p,q}(\,e=1\mid E^c(B_j) = 0\,)\right| ~\le~ \delta\,.
	\end{equation}
\end{definition}

\noindent
In words, MSM holds for $\B$ if for every edge $e \in E(\L_n)$ we can find $B_j$
such that $e \in E(B_j)$ and the ``influence'' of the configuration on $E^c(B_j)$
on the state of $e$ is bounded by $\delta$.

\begin{remark}\label{rem:ssm}
	SSM as defined in \cite{BS} holds
	when 
	MSM holds for a specific sequence of collections of subsets: if $\mathcal{B}_r$ is the set of subsets containing all the square boxes of side length $2r$
	centered at each $e \in E(\L_n)$ (intersected with $E(\L_n)$),
	then SSM holds if MSM holds for $\mathcal{B}_r$ for every $r \ge 1$ with $\delta = \exp(-\Omega(r))$. 
\end{remark}

MSM does not capture the fast mixing of the FK-dynamics the way SSM does. Namely, it is easy to find collections of subsets for which MSM holds for \textit{all} boundary conditions, including those boundary conditions for which we later prove slow mixing; see Theorem~\ref{thm:lb:intro}.   
However, if, for a collection $\B = \{B_1,B_2,\dots,B_k\}$,
we also bound the mixing time of the FK-dynamics on every $B_j$, we can deduce a mixing time bound for the FK-dynamics on $\L_n$. 
Let  $\tmix(B^\tau)$ denote the mixing time of the FK-dynamics on the subset $B \subseteq \L_n$ with boundary condition $\tau$.
(Recall that $\tau$ corresponds to a partition of $\b B$ and that $\b B$ consists of those vertices in $B$ that are adjacent to vertices in $\Z^2 \setminus B$.)

\begin{definition}
	\label{def:lm}
	Let $\xi$ be a boundary condition for $\L_n = (\L_n, E(\L_n))$ and
	let $\mathcal{B} = \{B_1,B_2,\dots,B_k\}$ with $B_j \subset \L_n$.
	We say that \emph{local mixing (LM)} holds for $\mathcal{B}$ and $T > 0$, if 
	\begin{align*}
	\tmix\big(B_j^{(1,\xi)}\big) \leq T \quad \mbox{and}\quad \tmix\big(B_j^{(0,\xi)}\big) \leq T \qquad \mbox{ for all $j=1,...,k$}
	\end{align*}
	where $(1,\xi)$ (resp., $(0,\xi)$) denotes the boundary condition on $B_j$
	induced by the event $\{E^c(B_j)=1\}$ (resp. $\{E^c(B_j)=0\}$) and the boundary condition $\xi$.
\end{definition}

\begin{remark}
	Observe that when $B_j \cap \b\L_n = \emptyset$, $(1,\xi)$ and $(0,\xi)$ are simply the wired and free boundary condition on $B_j$, respectively.
	When $B_j \cap \b\L_n \neq \emptyset$, the connectivities from $\xi$ could also induce some connections in $(1,\xi)$ and $(0,\xi)$.
\end{remark}

\noindent
Our next theorem, roughly speaking, establishes the following implication:
$$
\text{MSM} + \text{LM} \implies \text {upper bound for mixing time of FK-dynamics}, 
$$
with the quality of the bound depending on the $T$ for which LM holds.
A similar (and inspiring) implication for the 
Glauber dynamics of the Ising model in graphs of bounded degree
was established by Mossel and Sly in~\cite{MS}; there, the notion of MSM is replaced by a 
form of spatial mixing which is stronger than SSM. 

\begin{theorem}
	\label{thm:general-mixing}
	Let $\xi$ be a boundary condition on $\L_n = (\L_n,E(\L_n))$
	and
	let $\mathcal{B} = \{B_1,B_2,\dots,B_k\}$ with $B_j \subset \L_n$ for all $j=1,\dots,k$.
	If for $\xi$ and $\mathcal{B}$, 
	moderate spatial mixing holds for some $\delta \le 1/(12|E(\L_n)|)$
	and local mixing holds for some $T > 0$, then
	$\tmix(\L_n^\xi) = O(Tn^2 \log n)$.
\end{theorem}

\begin{proof}
	Consider two copies $\{X_t\}$, $\{Y_t\}$ of the FK-dynamics.
	We couple the evolution of $\{X_t\}$ and $\{Y_t\}$ 
	by using the same random $e \in E(\L_n)$ and the same uniform random number $r \in [0,1]$ to decide whether to add or remove $e$ from the configurations. 
	This is a standard coupling of the steps of the FK-dynamics (see, e.g.,~\cite{BS,Grimmett}); we call it the {\it identity coupling}.
	It is straightforward to verify that, when $q \ge 1$, the identity coupling is a {\it monotone coupling}, in the sense that if $X_t \subseteq Y_t$ then $X_{t+1} \subseteq Y_{t+1}$ with probability $1$. 
	
	We bound the coupling time $\Tcoup$ of the identity coupling for the FK-dynamics.
	The result then follows from the fact that $\tmix \le \Tcoup$. 
	Since the identity coupling can be extended to a simultaneous coupling of {\it all} configurations that preserves the partial order $\subseteq$, the coupling time starting from any pair of configurations is bounded by the coupling time for initial configurations $Y_0 = \emptyset$ and $X_0 = E(\L_n)$.
	
	We prove that there exists $\hat{T}=O(T n^2 \log n)$ such that the identity coupling $\mathbb P$ satisfies
	$$\P[X_{\hat{T}}(e) \neq Y_{\hat{T}}(e)] \le \frac{1}{4|E(\L_n)|}$$
	for every $e \in E(\L_n)$. A union bound over $e\in E(\L_n)$ then implies that $\Tcoup \le \hat{T}$.
	
	Fix any edge $e \in E(\L_n)$ and let $B_j \in \mathcal{B}$ be a subset for which~\eqref{eq:ssm:def} is satisfied (such a subset exists since moderate spatial mixing holds).
	To bound $\P[\,X_{\hat{T}}(e) \neq Y_{\hat{T}}(e)\,]$, we introduce two additional instances $\{Z^+_t\}$, $\{Z^-_t\}$ of the FK-dynamics. 
	These two copies update only edges with both endpoints in $B_j$ and reject all other updates. 
	We set $Z^+_0 = E(\L_n)$ and $Z^-_0 = \emptyset$. 
	The four Markov chains $\{X_t\}$, $\{Y_t\}$, $\{Z^+_t\}$ and $\{Z^-_t\}$ are coupled with the identity coupling, with updates outside $B_j$ ignored by both $\{Z^+_t\}$ and $\{Z^-_t\}$. The monotonicity of this coupling, along with a triangle inequality, implies that for all $t \ge 0$, 
	\begin{eqnarray}
	\P[\,X_t(e) \neq Y_t(e)\,] &\le& \left|\P[\,Z^+_t(e)=1\,] - \pi^\xi_{\L_n}(\,e=1\,|\, E^c(B_j) = 1\,)\right|\label{mixing:plus-bound}\\
	&~&+\left|\pi^\xi_{\L_n}(\,e=1\,|\, E^c(B_j) = 1\,) - \pi^\xi_{\L_n}(\,e=1\,|\, E^c(B_j) = 0\,)\right| \label{mixing:smp-bound}\\
	&~&+\left|\pi^\xi_{\L_n}(\,e=1\,|\,E^c(B_j) = 0\,) - \P[\,Z^-_t(e)=1\,] \right|. \label{mixing:minus-bound}
	\end{eqnarray}
	Observe that the restrictions of the chains $\{Z^+_t\}$ and $\{Z^-_t\}$ to $B_j$ are lazy versions of FK-dynamics on $B_j$ with stationary measures $\pi^\xi_{\L_n}(\,\cdot\,|\, E^c(B_j) = 1 \,)$ and $\pi^\xi_{\L_n}(\,\cdot\,|\, E^c(B_j)= 0)\,$, respectively. The laziness comes from the fact that they only accept updates that are in $B_j$, and by the local mixing assumption, once $T$ updates have been done in $B_j$, the chains $\{Z_t^+\}$ and $\{Z_t^-\}$ will be mixed. 
	
	Now, after $\hat{T} = C T |E(\L_n)| \log_2 \lceil 24|E(\L_n)| \rceil $ steps, the expected number of updates in $B_j$ is 
	$$
	C T |E(\L_n)| \log_2 \lceil 24|E(\L_n)| \rceil  \frac{|E(B_j)|}{|E(\L_n)|} \ge C T  \log_2 \lceil 24|E(\L_n)| \rceil\,. 
	$$
	Let ${\mathcal A}_{\hat{T} }$ be the event that 
	the number of updates in $B_j$ after $\hat{T}$ steps is at least 
	$T  \log_2 \lceil 24|E(\L_n)| \rceil$.  
	A Chernoff bound then implies that, for a large enough constant $C > 0$,
	$$
	\Pr[{\mathcal A}^c_{\hat{T} }] \le \frac{1}{24|E(\L_n)|}\,.
	$$
	Therefore,
	\begin{equation}
	\label{eq:sec3:proof-a}
	\left|\Pr[Z^+_{\hat{T}}(e)=1 \mid {\mathcal A}_{\hat{T} }] - \Pr[Z^+_{\hat{T}}(e)=1]\right| \le \Pr[ {\mathcal A}^c_{\hat{T} }] \le \frac{1}{24|E(\L_n)|}\,.
	\end{equation}
	
	By the local mixing property, and the fact that $\tmix(\varepsilon) \le \lceil \log_2 \varepsilon^{-1} \rceil \cdot \tmix$ for any  positive $\varepsilon < 1/2$, we have  
	\begin{equation}
	\label{eq:sec3:proof-b}
	\left|\Pr[Z^+_{\hat{T}}(e)=1 \mid {\mathcal A}_{\hat{T} }] - \pi^\xi_{\L_n}(\,e=1\,|\, E^c(B_j) = 1\,)\right| \le \frac{1}{24|E(\L_n)|}\,.
	\end{equation}
	It then follows from \eqref{eq:sec3:proof-a}, \eqref{eq:sec3:proof-b} and the triangle inequality
	that when $t = \hat{T}$ the right-hand side in~\eqref{mixing:plus-bound} is at most $\frac{1}{12|E(\L_n)|}$.
	The same bound can be deduced for \eqref{mixing:minus-bound} in a similar manner. 
	
	Finally, since moderate spatial mixing holds for some $\delta < \frac{1}{12|E(\L_n)|}$, we have
	$$
	\left|\pi^\xi_{\L_n}(\,e=1\,|\, E^c(B_j) = 1\,) - \pi^\xi_{\L_n}(\,e=1\,|\, E^c(B_j) = 0\,)\right| \le \frac{1}{12|E(\L_n)|}\,;
	$$ 
	see Definition~\ref{def:msm}. Putting these together we see that
	$\P[\,X_{\hat{T}}(e) \neq Y_{\hat{T}}(e)\,] \le \frac{1}{4|E(\L_n)|}$ 
	as desired.
\end{proof}

\section{Fast mixing on thin rectangles}
\label{sec:thin-rect}

The main difficulty in proving Theorem~\ref{thm:planar-mixing:intro} 
using the general framework from Section~\ref{sec:general}
is obtaining mixing time estimates on thin rectangles of dimension $\Theta(n) \times \Theta(\log n)$  with realizable boundary conditions. To motivate this, we notice that since $p<p_c(q)$, the influence of the boundary is lost, with high probability, at a distance $\Theta(\log n)$. Thus the main difficulty will be to bound the mixing time of the FK-dynamics in the annulus of width $\Theta(\log n)$ with realizable boundary conditions on the outside. 
As such, the key ingredient in the proof of Theorem~\ref{thm:planar-mixing:intro} will be the following mixing time bound on thin rectangles.

For an $n\times l$ rectangle $\L_{n,l}=\llb0,n\rrb\times\llb0,l\rrb$, we use
$\b_{\north}\L_{n,l}$, $\b_{\east}\L_{n,l}$, $\b_{\south}\L_{n,l}$ and $\b_{\west}\L_{n,l}$
for its north, east, south and west boundaries respectively.
Recall that for $a,b\in \mathbb Z$, we set $\llb a,b \rrb = \{a,a+1,\dots,b\}$ and $\llb a,b \rrb^c = \llb0,n \rrb \setminus \llb a,b \rrb$.

\begin{theorem}
	\label{thm:main-thin-rect}
	Consider $\Lambda_{n,l}= (\Lambda_{n,l},E(\Lambda_{n,l}))$ 
	for $l\leq n$ with an arbitrary \emph{realizable} boundary condition $\xi$ that is either free or wired on $\partial_\east \L_{n,l} \cup \partial_\west \L_{n,l} \cup\partial_\south \L_{n,l}$. 
	Then, for every $q>1$ and $p \neq  p_c(q)$, the mixing time 
	of the FK-dynamics on $\Lambda_{n,l}$ is at most $\exp(O(l+\log n))$.
\end{theorem} 

\noindent Observe that when $l =O(\log n)$, this implies the mixing time is $n^{O(1)}$, which will be the setting of interest in our proofs. Moreover, we note that it suffices for us to prove Theorem~\ref{thm:main-thin-rect} for the set of realizable boundary conditions $\xi$ that are free on $\partial_\east \L_{n,l} \cup\partial_{\west} \L_{n,l}\cup \partial_\south \L_{n,l}$ and all $p\neq p_c(q)$, as the set of boundary conditions dual to these are exactly the set of realizable boundary conditions that are wired on   $\partial_\east \L_{n,l} \cup\partial_{\west} \L_{n,l}\cup \partial_\south \L_{n,l}$; see Remark~\ref{rem:fk-dynamics-duality}.

In Section~\ref{subsec:thin:sketch}, we give a overview of the main ideas in the proof of this theorem. In Sections~\ref{subsec:disconnecting-intervals}--\ref{subsec:compatible-bc}, we introduce some crucial notions regarding groups of rectangular subsets of $\L_{n,l}$ and their relations with the boundary conditions $\xi$. Sections~\ref{subsec:sub-blocks}--\ref{subsec:block-dynamics} bound the mixing time of the block dynamics with respect to a suitably-chosen set of subsets of $\L_{n,l}$. The recursive proof of Theorem~\ref{thm:main-thin-rect} is then completed in Section~\ref{subsec:complete-proof}. 

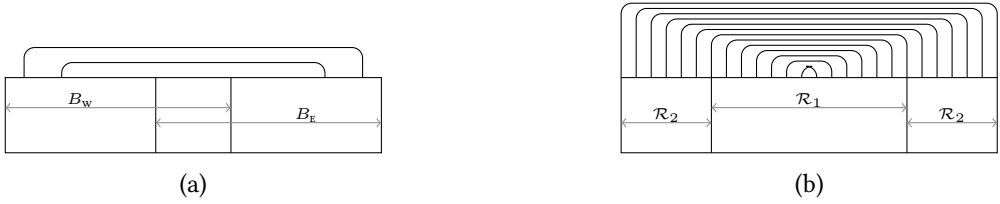
\begin{figure}[t]
	\begin{subfigure}[b]{0.49\textwidth}
		
		\begin{center}
			\begin{tikzpicture}
			\draw[draw=black] (0,0) rectangle (5,1);
			\draw[draw=black] (2,0)--(2,1);
			\draw[draw=black] (3,0)--(3,1);
			\draw[draw=gray!90,<->] (0,0.6) -- (3,0.6);
			\draw[draw=gray!90,<->] (2,0.4) -- (5,0.4);
			\draw [rounded corners]  (0.25,1)--(0.25,1.4)--(4.75,1.4)--(4.75,1);
			\draw [rounded corners]  (0.75,1)--(0.75,1.2)--(4.25,1.2)--(4.25,1);
			\node at (1,0.7) {{\tiny $B_\west$}};
			\node at (4,0.5) {{\tiny $B_\east$}};
			
			\end{tikzpicture}
			\caption{}
			
		\end{center}
	\end{subfigure}
	\begin{subfigure}[b]{0.49\textwidth}
		
		\begin{center}
			\begin{tikzpicture}
			\draw[draw=black] (0,0) rectangle (5,1);
			
			\draw [rounded corners]  (0,1)--(0,1.99)--(5,1.99)--(5,1);
			\draw [rounded corners]  (0.2,1)--(0.2,1.92)--(4.8,1.92)--(4.8,1);
			\draw [rounded corners]  (0.4,1)--(0.4,1.85)--(4.6,1.85)--(4.6,1);
			\draw [rounded corners]  (0.6,1)--(0.6,1.78)--(4.4,1.78)--(4.4,1);
			\draw [rounded corners]  (0.8,1)--(0.8,1.71)--(4.2,1.71)--(4.2,1);
			\draw [rounded corners]  (1,1)--(1,1.64)--(4,1.64)--(4,1);
			\draw [rounded corners]  (1.2,1)--(1.2,1.57)--(3.8,1.57)--(3.8,1);
			\draw [rounded corners]  (1.4,1)--(1.4,1.5)--(3.6,1.5)--(3.6,1);
			\draw [rounded corners]  (1.6,1)--(1.6,1.43)--(3.4,1.43)--(3.4,1);
			\draw [rounded corners]  (1.8,1)--(1.8,1.36)--(3.2,1.36)--(3.2,1);
			\draw [rounded corners]  (2,1)--(2.0,1.29)--(3,1.29)--(3,1);
			\draw [rounded corners]  (2.2,1)--(2.2,1.22)--(2.8,1.22)--(2.8,1);
			\draw [rounded corners]  (2.4,1)--(2.4,1.15)--(2.6,1.15)--(2.6,1);

			\draw[draw=black] (1.2,0)--(1.2,1);
			\draw[draw=black] (3.8,0)--(3.8,1);
			\draw[draw=gray!90,<->] (1.2,0.6) -- (3.8,0.6);
			\draw[draw=gray!90,<->] (0,0.4) -- (1.2,0.4);
			\draw[draw=gray!90,<->] (3.8,0.4) -- (5,0.4);
			
			\node at (2.5,0.7) {{\tiny $\mathcal R_1$}};
			\node at (0.6,0.5) {{\tiny $\mathcal R_2$}};	
			\node at (4.4,0.5) {{\tiny $\mathcal R_2$}};	
			\end{tikzpicture}
			\caption{}
		\end{center}
	\end{subfigure}
	
	\caption{{{\rm (a)} A boundary condition for which no configuration in $B_\west \cap B_\east$ isolates $B_\west\setminus B_\east$ from $B_\east\setminus B_\west$.} 
		{{\rm (b)} A boundary condition $\xi$ where every pair of overlapping rectangles (as in \eqref{eq:potential-blocks}) must interact through $\xi$; however, the two \emph{groups of rectangles} $\mathcal R_1$, $\mathcal R_2$ do not interact through $\xi$.}}     
	\label{figure:sketch:blocks}
\end{figure}

\subsection{Outline of proof of Theorem~\ref{thm:main-thin-rect}}
\label{subsec:thin:sketch}

We first mention some obstructions that boundary conditions present to proving Theorem~\ref{thm:main-thin-rect} using approaches that are common in analogous problems for spin systems.
A traditional approach to proving mixing time bounds for thin rectangles is the canonical paths method (\cite{Martinelli-SP,Martinelli-notes,JS,Sinclair}), which gives an upper bound that is exponential in the shorter side length; however, boundary conditions can significantly distort the augmented graph with external wirings, preventing this approach from succeeding. 

A sharper approach would be to use an inductive scheme~\cite{Cesi,Martinelli-notes,GL2}, whereby, we bound the mixing time of the FK-dynamics on $n\times l$ rectangles by the mixing times in smaller rectangular blocks, e.g., 
\begin{align}\label{eq:potential-blocks}
B_\west = \llb 0,\tfrac 23 n\rrb \times \llb 0,l\rrb \qquad \mbox{and}\qquad B_\east = \llb \tfrac 13 n ,n\rrb \times \llb 0,l\rrb\,.
\end{align}
This method requires bounding by the mixing time of the so-called \textit{block dynamics}.
\begin{definition}
	The \textit{block dynamics} $\{X_t\}$ with blocks $B_\west,B_\east\subset \L_{n, l}$ such that $E(B_\west) \cup E(B_\east)= E(\L_{n, l})$ is the discrete-time Markov chain that, at each $t$, picks $i$ uniformly at random from $\{\west,\east\}$ and updates the configuration in $E(B_i)$ with a sample from the stationary distribution of the chain conditional on the configuration of $X_t$ on $E^c(B_i)$.
	\end{definition}
\noindent
The spectral gap of the FK-dynamics on $\L_{n , l}$ is bounded from below
by the spectral of the block dynamics times the worst gap of the FK-dynamics in any $B_i$ with worst-case configuration on $E^c(B_i)$;
see Theorem~\ref{thm:block-dynamics} for a precise statement. With the choice of blocks in~\eqref{eq:potential-blocks}, applying this recursively, one would bound the spectral gap of the FK-dynamics on $\L_{n,l}$ by the gap of the block dynamics raised to a $\Theta(\log n)$ power. Hence, establishing Theorem~\ref{thm:main-thin-rect} would require an $\Omega(1)$ lower bound on the spectral gap of the block dynamics.

The mixing time and spectral gap of the block dynamics is typically bounded by showing that after the first block update in either $B_\west$ or $B_\east$, 
the configuration in $B_\west \cap B_\east$ will be such that it disconnects the influence of the configuration on $B_\west \setminus B_\east$ from $B_\east \setminus B_\west$ with probability $\Omega(1)$;
this then allows a standard coupling argument to be used to bound the mixing time. 
In the presence of long-range boundary connections, however,  it could be that no configuration on $B_\west \cap B_\east$ would disconnect the two sides from one another and facilitate coupling; see Figure \ref{figure:sketch:blocks}(a) for such an example. As such, our choices of blocks will depend on the boundary conditions and will be chosen to allow for the block dynamics to couple in $O(1)$ time, while ensuring that the blocks are still at most a fraction of the size of the original rectangle, so that after $O(\log n)$ recursive steps we arrive at a sufficiently small base scale.

In particular, we will show that for every realizable boundary condition $\xi$ on $\L_{n,l}$, there exists a choice of two blocks, whose widths are at most $\frac 45 n$, such that they are sufficiently isolated from one another in~$\xi$. As Figure~\ref{figure:sketch:blocks}(b) demonstrates, there are realizable boundary conditions that would force these blocks to not be single rectangles, as in~\eqref{eq:potential-blocks}, but rather collections of rectangular subsets of $\L_{n,l}$. 
Thus, our recursive argument will proceed instead on \emph{groups of rectangles}, 
$$\mathcal R = \bigcup R_i,\, \qquad \mbox{where} \qquad R_i=\llb a_i,b_i \rrb \times \llb 0,l\rrb \subset \L_{n,l}$$ 
are disjoint, with boundary conditions induced by $\xi$ and the configuration of the chain on $\L_{n,l}\setminus \mathcal R$.

We formally define groups of rectangles (Definition~\ref{def:group-of-rectangles}), and their boundary conditions in Section~\ref{subsec:groups-of-rectangles}. We consider next the notion of \textit{compatibility} of a group of rectangles $\mathcal R \subset \L_{n,l}$ with the boundary condition $\xi$. Roughly speaking, we say a group of rectangles $\mathcal R$ is compatible with $\xi$ if $\xi$ limits the boundary interactions between the rectangles of $\mathcal R$ for every possible configuration on $\L_{n,l}\setminus\mathcal R$; see Definition~\ref{def:compatible-bc}. 
In Section~\ref{subsec:sub-blocks}, we provide an algorithm (see Lemma~\ref{lem:splitting-algorithm}) that, for a group of rectangles $\mathcal R$ compatible with $\xi$, finds two suitable subsets for the block dynamics: these subsets will each be group of rectangles 
compatible with~$\xi$, of width between $1/5$ and $4/5$ of the width of $\mathcal R$.
This will allow us to induct on groups of rectangles compatible with $\xi$, while ensuring that number of recursive steps is $O(\log n)$. 
Specifically, our splitting algorithm will find interior and exterior subsets $\mathcal A_{\interior}$ and $\mathcal A_{\ext}$ with no boundary connections between the two; see Figure~\ref{figure:blocks}(a). These will be the cores of the two blocks for $\mathcal R$, but in order to bound the coupling time of the block dynamics, we want the two blocks to overlap, and therefore we enlarge $\mathcal A_{\interior}$ and $\mathcal A_{\ext}$ by $m=\Theta(\log l)$ to form the blocks $\mathcal R_{\interior}$ and $\mathcal R_\ext$ for the block dynamics; see Figure~\ref{figure:blocks}(b). 

Finally, in Section~\ref{subsec:block-dynamics}, we bound the coupling time of this block dynamics 
by some sufficiently large constant to conclude the proof. This follows by leveraging the fact that $p< p_c$ ($p>p_c$) to condition on the existence of certain disconnecting dual (primal) paths in $\mathcal R_\interior \cap \mathcal R_\ext$ that have $\Omega(1)$ probability. 

We emphasize that in order to push through this recursion, is will be crucial that our notion of compatibility with $\xi$ is strong enough to yield a uniform bound on this block dynamics coupling time, while being broad enough that the splitting algorithm always succeeds in finding sub-blocks that are themselves groups of rectangles compatible with $\xi$.

\subsection{Disconnecting intervals}
\label{subsec:disconnecting-intervals}

In this section, we introduce the notion of disconnecting intervals, one of the building blocks of our recursive proof of polynomial mixing. 
Recall that we use
$\b_{\north}\L_{n,l}$, $\b_{\east}\L_{n,l}$, $\b_{\south}\L_{n,l}$ and $\b_{\west}\L_{n,l}$
for the north, east, south and west boundaries of the 
rectangle $\L_{n,l}$, respectively.

\begin{definition}
	\label{def:disconnecting-interval}
	For a \emph{realizable} boundary condition $\xi$ on $\L_{n,l}$ that is free on $\b_{\east}\L_{n,l}\cup\b_{\south}\L_{n,l}\cup\b_{\west}\L_{n,l}$, an interval $\llb a,b \rrb\subset\llb0,n \rrb$
	is called \emph{disconnecting} of
	\begin{enumerate}
		\setlength{\itemsep}{5pt}
		\item \emph{free-type}: if there are no boundary connections in $\xi$ between $\llb a,b\rrb\times \{l\}$ and $\llb a,b\rrb^c \times \{l\}$;
		\item \emph{wired-type}: if there is a boundary component in $\xi$ that contains the vertices $(a,l)$ and $(b,l)$.
	\end{enumerate}
\end{definition}

\begin{figure}[t]
	\begin{center}
			\begin{tikzpicture}		
			\draw[draw=black] (1,0) rectangle (10.8,1);
			\draw [rounded corners]  (1.5,1)--(1.5,1.4)--(7.5,1.4)--(7.5,1);
			\draw [rounded corners]  (2.5,1)--(2.5,1.2)--(4,1.2)--(4,1);
			\draw [rounded corners]  (6,1.4)--(6,1);
			\draw [rounded corners]  (8,1)--(8,1.2)--(9.2,1.2)--(9.2,1);
			\draw [rounded corners]  (9.8,1)--(9.8,1.2)--(10.5,1.2)--(10.5,1);
			
			\draw[color=black,fill=black] (1.5,1) circle (.03);
			\draw[color=black,fill=black] (2.5,1) circle (.03);
			\draw[color=black,fill=black] (4,1) circle (.03);
			\draw[color=black,fill=black] (4.75,1) circle (.03);
			\draw[color=black,fill=black] (5.25,1) circle (.03);
			\draw[color=black,fill=black] (6,1) circle (.03);
			\draw[color=black,fill=black] (7.5,1) circle (.03);
			\draw[color=black,fill=black] (8,1) circle (.03);
			\draw[color=black,fill=black] (9.2,1) circle (.03);
			\draw[color=black,fill=black] (9.8,1) circle (.03);
			\draw[color=black,fill=black] (10.5,1) circle (.03);
			
			\node at (1.5,0.8) {{\tiny $a_0$}};			
			\node at (2.5,0.8) {{\tiny $a_1$}};
			\node at (4,0.8) {{\tiny $a_2$}};
			\node at (4.75,0.8) {{\tiny $a_3$}};
			\node at (5.25,0.8) {{\tiny $a_4$}};
			\node at (6,0.8) {{\tiny $a_5$}};
			\node at (7.5,0.8) {{\tiny $a_6$}};
			\node at (8,0.8) {{\tiny $a_7$}};
			\node at (9.2,0.8) {{\tiny $a_8$}};
			\node at (9.8,0.8) {{\tiny $a_9$}};
			\node at (10.5,0.8) {{\tiny $a_{10}$}};
			\end{tikzpicture}
		\end{center}

	\caption{The rectangle $\L_{n , l}$ with a boundary condition $\xi$ inducing disconnecting intervals. For example, $\llb a_1,a_4\rrb$, $\llb a_3, a_4\rrb$ and $\llb a_7,a_{10}\rrb$ are disconnecting intervals of free-type; $\llb a_1,a_2\rrb$, $\llb a_0,a_6\rrb$, $\llb a_7,a_{8}\rrb$ and $\llb a_9,a_{10}\rrb$ are of free-wired-type; $\llb a_0,a_5\rrb$ and $\llb a_5,a_6\rrb$ are of wired-type.}
	\label{figure:di}
	\label{fig:disc-int}
\end{figure}
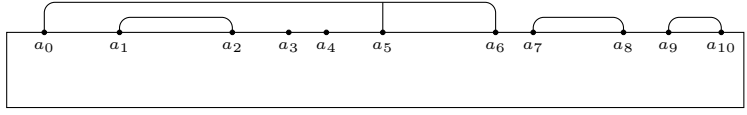

Observe that an interval can be both of free-type and of wired-type if $(a,l)$ and $(b,l)$ are connected through $\xi$ but are not connected to any boundary vertex in $\llb a,b\rrb^c \times \llb 0, l \rrb$; in this case, we may refer to the interval as being of \emph{free-wired-type}; see Figure~\ref{fig:disc-int} for several examples. 

The following properties concerning the union and intersection of disconnecting intervals will be crucial to our proofs. 

\begin{lemma}\label{lem:disconnecting-unions}
	Let $\xi$ be a realizable boundary condition on $\Lambda_{n,l}$ that is free on $\b_{\south}\L_{n,l}\cup\b_{\east}\L_{n,l}\cup\b_{\west}\L_{n,l}$ and 
	let $a < b < c$. If both $\llb a,b\rrb$ and $\llb b,c\rrb$ are disconnecting intervals of wired-type, then so is $\llb a,c\rrb$.
	If both $\llb a,b\rrb$ and $\llb b+1,c\rrb$ are disconnecting intervals of free-type, then so is $\llb a,c\rrb$. 
\end{lemma}

\begin{proof}
	 If $\llb a,b\rrb$ and $\llb b,c\rrb$ are disconnecting intervals of wired-type, then by definition the vertices
	 $(a,l)$, $(b,l)$  and $(c,l)$ are all in the same component of $\xi$; hence
	 $\llb a,c\rrb$ is a disconnecting interval of wired-type.
	 If $\llb a,b\rrb$ and $\llb b+1,c\rrb$ are disconnecting intervals of free-type, 
	 then by definition there are no connections in $\xi$
	 between $\llb a,b\rrb$ and $\llb a,b\rrb^c \supset \llb a,c\rrb^c$
	 or between $\llb b+1,c\rrb$ and $\llb b+1,c\rrb^c \supset \llb a,c\rrb^c$.
	 Consequently, there are not connections in $\xi$ between $\llb a,c\rrb = \llb a,b\rrb \cup \llb b+1,c\rrb$ and $\llb a,c\rrb^c$, and $\llb a,c\rrb$ is thus a disconnecting interval of free-type.
\end{proof}

\begin{lemma}\label{lem:disconnecting-intersections-type}
	Let $\xi$ be a realizable boundary condition on $\Lambda_{n,l}$ that is free on $\b_{\south}\Lambda_{n,l} \cup \b_{\east}\Lambda_{n,l} \cup \b_{\west}\Lambda_{n,l}$.
	Suppose there exist $a<b\leq c<d$ such that $\llb a,c\rrb$ and $\llb b,d\rrb$ are disconnecting intervals. Then either both $\llb a,c\rrb$ and $\llb b,d\rrb$ are of free-type or both are of wired-type.
\end{lemma}

\begin{proof}
	Suppose by way of contradiction 
	and without loss of generality 
	that $\llb a,c\rrb$  is only of free-type (i.e., free-type but not free-wired-type) and $\llb b,d\rrb$ is of wired-type. By definition, there exists a component of $\xi$ that contains both $(b,l)$ and $(d,l)$. 
	Since $a<b \leq c<d$,  there is therefore a connection in $\xi$ between $\llb a,c\rrb \times \{l\}$ and $\llb a,c\rrb^c \times \{l\}$. 
	Hence, $\llb a,c\rrb$ cannot be a disconnecting interval of free-type yielding the desired contradiction.\end{proof}

\begin{lemma}
	\label{lem:disconnecting-intersections}
	Let $\xi$ be a realizable boundary condition on $\Lambda_{n,l}$ that is free on $\b_{\south}\Lambda_{n,l} \cup \b_{\east}\Lambda_{n,l} \cup \b_{\west}\Lambda_{n,l}$. Suppose there exist $a<b<c<d$ such that $\llb a,c\rrb$ and $\llb b,d\rrb$ are disconnecting intervals. 
	\begin{enumerate}
		\setlength{\itemsep}{5pt}
		\item If $\llb a,c\rrb$ and $\llb b,d\rrb$ are both of \emph{wired-type}, then $\llb a,b\rrb, \llb b,c\rrb, \llb c,d\rrb$ and $\llb a,d\rrb$ are all disconnecting intervals of wired-type.
		\item If $\llb a,c\rrb$ and $\llb b,d\rrb$ are both of \emph{free-type}, then $\llb a,b-1\rrb, \llb b,c\rrb,\llb c+1,d\rrb$ and $\llb a,d\rrb$ are all disconnecting intervals of free-type.
	\end{enumerate}
\end{lemma}
\begin{proof}
	For the first part, suppose that both $\llb a,c \rrb$ and $\llb b,d \rrb$ are disconnecting intervals of wired-type. 
	By definition, the vertices 
	$(a,l)$ and $(c,l)$ are in the same component of $\xi$, as are $(b,l)$ and $(d,l)$.
	 By the planarity of realizable boundary conditions, it must be the case that these two components of $\xi$ are indeed the same, and therefore $(a,l),(b,l),(c,l),(d,l)$ are all in the same boundary component of $\xi$. As such, $\llb a,b \rrb$, $\llb b,c \rrb$, $\llb c,d \rrb$ and $\llb a,d \rrb$ are all disconnecting intervals of wired-type. 
	
	For the second part, suppose first
	by way of contradiction that $\llb a,b-1 \rrb$ is not a disconnecting
	interval of free-type. Then there must be a boundary component in $\xi$
	with vertices in
	$\llb a,b-1 \rrb\times\{l\}$ and $\llb a,b-1 \rrb^{c}\times\{l\}$.
	If it is a component containing a vertex in $\llb a,c \rrb^{c}\times\{l\}$,
	it would violate the fact that $\llb a,c \rrb$ is disconnecting of free-type, while
	if it is a component containing a vertex in $\llb b,c \rrb\times\{l\}$,
	it would violate that $\llb b,d \rrb$ is disconnecting of free-type as $\llb a,b-1\rrb \subset \llb b,d\rrb^c$. By analogous
	reasoning $\llb c+1,d \rrb$, $\llb b,c\rrb$ and $\llb a,d\rrb$ are all disconnecting intervals
	of free-type. 
\end{proof}

\subsection{Groups of rectangles}
\label{subsec:groups-of-rectangles}

In this section, we define \emph{groups of rectangles} and their boundary conditions, which constitute the other building blocks of our recursive proof of polynomial mixing for the FK-dynamics on thin rectangles. As in the previous section, 
we consider an $n\times l$ rectangle $\L_{n,l}=\llb0,n\rrb\times\llb0,l\rrb$
with a realizable boundary condition $\xi$ that is free on $\b_{\east}\L_{n,l}\cup\b_{\south}\L_{n,l}\cup\b_{\west}\L_{n,l}$.

Henceforth, we take 
$$
m=m(l)= C_{\star}\log l\,,
$$
where $C_{\star}$ is a large constant such that $C_{\star}>c^{-1}$ which we choose later,
with $c$ being the constant from~\eqref{eq:EDC}. 

A \emph{rectangular subset} $R\subset \L_{n,l}$ is a rectangle of the form $\llb a,b \rrb \times \llb0,l \rrb$ for some $0\leq a <b\leq n$. For such a rectangular subset, we denote by $W(R)$ its width; i.e., $W(R)=b-a$. 
For the union of distinct and disjoint rectangular subsets $\mathcal R = \bigcup_{i=1}^{N(\mathcal R)} R_i$, where $R_i = \llb a_i, b_i \rrb \times \llb 0,l\rrb$ and $a_1 <b_1 <a_{2}<\dots<a_{N(\mathcal R)} < b_{N(\mathcal R)} $, its width is defined by $W(\mathcal R) = \sum_{i=1}^{N(\mathcal R)} W(R_i)$ and its boundary is given by $\partial \mathcal R = \bigcup_{i=1}^{N(\mathcal R)} \partial R_i$. Moreover, we let $\partial_\north \mathcal R = \bigcup_{i=1}^{N(\mathcal R)} \partial_\north R_i$ and similarly define $\partial_\south \mathcal R$; on the other hand, $\partial_\east \mathcal R$ will be the eastern boundary of the right-most rectangle $R_i$ and $\partial_\west \mathcal R$, the western boundary of the left-most. 

\begin{definition}
\label{def:group-of-rectangles}A \emph{group of rectangles} $\mathcal{R}=\bigcup_{i=1}^{N(\mathcal{R})}R_{i}$
is the union of $N(\mathcal{R})$ disjoint rectangular subsets $R_{i}$ of
$\Lambda_{n,l}$ such that 
$W(R_{i})\geq2m$ for every $i=1,...,N(\mathcal{R})$.
\end{definition}

\begin{remark}
	The requirement that $W(R_{i})\geq2m$ for every $i$, which may seem arbitrary at the moment, is because in our recursive argument, we want our groups of rectangles $\mathcal R$ to have interiors that are not influenced by the configuration on $E(\L_{n,l})\setminus E(\mathcal R)$. When a group of rectangles has a thin constituent rectangle $R_i$, the influence of the outside configuration can permeate through all of $R_i$.     
\end{remark}

We will be considering blocks dynamics with blocks consisting of groups of rectangles. 
If we want to update the configuration on a group of rectangles $\mathcal R \subset \L_{n,l}$, the boundary condition induced on $\partial \mathcal R$ will consist of the boundary condition on $\L_{n,l}$ which will be fixed to be $\xi$, together with a fixed random-cluster configuration $\omega_{\mathcal R^c}$ on $E^c(\mathcal R) =  E(\L_{n,l})\setminus E(\mathcal R)$. 
Hence, our boundary conditions on $\mathcal R$ will be of this form, and we denote them by the pair $\zeta = (\xi,\omega_{\mathcal R^c})$.

\subsection{Compatible boundary conditions}\label{subsec:compatible-bc}

We now define the notion of \textit{compatibility} of groups of rectangles with boundary conditions $\xi$. This is crucial in our inductive argument; our algorithm for finding suitable blocks for the block dynamics will guarantee that if the starting group rectangles is compatible with respect to a given boundary condition, so will each of the blocks, enabling an inductive procedure. 

Let $\xi$ be a realizable boundary condition 
on $\b\Lambda_{n,l}$ that is free on $\b_{\south}\Lambda_{n,l} \cup \b_{\east}\Lambda_{n,l} \cup \b_{\west}\Lambda_{n,l}$,
and free in all vertices in $\partial_{\north}\Lambda_{n,l}$ at distance at most $m$ from $\partial_{\east}\Lambda_{n,l}\cup\partial_{\west}\Lambda_{n,l}$ (i.e., they appear as singletons in the corresponding boundary partition). 
This latter requirement for vertices near the corners of $\L_{n,l}$ allows us to always choose blocks in our recursion whose boundaries are at least distance $m$ from $\b _\east \L_{n,l}\cup \b_\west \L_{n,l}$; this simplifies our analysis.

The following will be the distinguishing property of our choice of blocks for the block dynamics.  

\begin{definition}
	\label{def:compatible-bc}
	Let $\mathcal R= \bigcup_{i=1}^{N(\mathcal R)} R_i$ 
	be a group of rectangles with
	$R_{i}=\llb a_{i},b_{i}\rrb\times\llb0,l\rrb$
	and $a_1 < b_1 < a_2 < b_2 < \ldots < a_{N(\mathcal R)} < b_{N(\mathcal R)}$ (with $b_i - a_i \geq 2m$ for all $i$). We say
	$\mathcal R$
	is \emph{compatible} with $\xi$, if 
	\begin{enumerate}
		\setlength{\itemsep}{5pt}
		\item Between every two consecutive rectangles $R_{i}=\llb a_{i},b_{i}\rrb\times\llb0,l\rrb$
		and $R_{i+1}=\llb a_{i+1},b_{i+1}\rrb\times\llb0,l\rrb$
		the interval $\llb b_{i}-m,a_{i+1}+m \rrb$ is a disconnecting interval; and
		\item The interval $\llb a_1+m,b_{N(\mathcal R)}-m \rrb$ is also a disconnecting interval. 
	\end{enumerate}
\end{definition}
\begin{remark}
	It is clear from the definition that $\L_{n, l}$ is compatible with $\xi$: the first condition is vacuous, while the second is satisfied since all vertices a distance
	at most $m$ from $\b_{\east}\L_{n,l}\cup\b_{\west}\L_{n , l}$
	are free. Observe also that $b_i - a_i \ge 2m$ for every $i$ and so 
	$b_{N(\mathcal R)}-m \ge a_1+m$; see Definition~\ref{def:group-of-rectangles}.
\end{remark}

\subsection{Defining the blocks for the block dynamics}
\label{subsec:sub-blocks}

Now that we have introduced disconnecting intervals, group of rectangles and the notion of compatibility, we describe our algorithm for picking two blocks $\mathcal R_\interior,\mathcal R_\ext$ for the block dynamics, based on the boundary condition $\xi$.
Recall the definitions of width $W(\cdot)$ of a rectangular subset and width of a collection of rectangles $W(\mathcal R)= \sum _{i=1}^{N(\mathcal R )} W(R_i)$. It will also be convenient to have the following notation $\partial_\parallel \mathcal R = \bigcup_{i=1}^{N(\mathcal R)} \partial_\west R_i \cup \partial_\east R_i$ for the vertical sides of the group of rectangles $\mathcal R = \bigcup_{i=1}^{N(\mathcal R)} R_i$.  The following lemma provides the 
basis
for our splitting algorithm.
\begin{lemma}
\label{lem:splitting-algorithm}
Let $\xi$ be a realizable boundary condition 
on $\b\Lambda_{n,l}$ that is free on $\b_{\south}\Lambda_{n,l} \cup \b_{\east}\Lambda_{n,l} \cup \b_{\west}\Lambda_{n,l}$
and free in all vertices in $\b_{\north}\Lambda_{n,l}$ at distance at most $m$ from $\partial_{\east}\Lambda_{n,l}\cup\partial_{\west}\Lambda_{n,l}$.
For every group of rectangles $\mathcal{R}$
compatible with $\xi$, with $W(\mathcal R)\geq 100m$, there exists a disconnecting interval $\llb c_\star,d_\star \rrb$ 
such that both $(c_\star,l)$ and $(d_\star,l)$ are in $\partial_\north \mathcal R$, are distance at least $m$
from $\partial_\parallel \mathcal R$,
and 
\begin{align*}
\frac{1}{4}W(\mathcal{R})\leq W(\mathcal{R}\cap(\llb c_\star,d_\star \rrb\times\llb 0,l \rrb))\leq\frac{3}{4}W(\mathcal{R})\,.
\end{align*}
\end{lemma}

We pause to comment on why a disconnecting interval with such properties provides the desired blocks for the block dynamics.  
The interval $\llb c_\star,d_\star\rrb$ from the lemma will be used to define $\mathcal A_\interior=\mathcal R \cap (\llb c_\star,d_\star\rrb \times \llb 0,l\rrb)$ and $\mathcal A_\ext = \mathcal R \setminus \mathcal A_\interior$; their enlargements by $m$ will form the blocks $\mathcal R_\interior$ and $\mathcal R_\ext$ (see Figure~\ref{figure:blocks}). The requirement that $W(\mathcal A_{\interior})$ be a fraction of $W(\mathcal R)$ bounded away from $0$ and $1$ is so that we only recurse $O(\log n)$ times before reaching small enough widths. The requirement that the corners of $\llb c_\star,d_\star\rrb\times \llb 0,l\rrb$ are a distance at least $m$ from $\partial_\| \mathcal R$ is so that when we enlarge the sets $\mathcal A_{\interior},\mathcal A_\ext$ by $m$, we do not overflow beyond the rectangles containing $(c_\star,l)$ and $(d_\star,l)$. Crucially, our ability to pick disconnecting segments that satisfy this latter requirement will be guaranteed by the compatibility of $\mathcal R$ with $\xi$. 

\begin{proof}[Proof of Lemma~\ref{lem:splitting-algorithm}]
We begin by finding a candidate disconnecting interval $\llb c,d \rrb$ with  $(c,l)$, $(d,l)\in \partial_\north \mathcal R$
satisfying: 
\begin{align}
\label{eq:slit}
\frac{1}{3}W(\mathcal{R}) \leq W(\mathcal{R}\cap(\llb c,d \rrb\times\llb 0,l \rrb)) \leq \frac{2}{3}W(\mathcal{R})\,. 
\end{align}
In the second part of the proof we show how to modify 
the interval $\llb c,d \rrb$ to obtain a 
disconnecting interval $\llb {c_\star},{d_\star} \rrb$ with the added property that 
both $(c_\star,l)$ and $(d_\star,l)$ are distance at least $m$
from $\partial_\parallel\mathcal R$.

If there exist a pair of vertices $(x,l),(y,l)\in\partial_{\north}\mathcal{R}$
such that 
$
\frac{1}{3}W(\mathcal{R}) \leq W(\mathcal{R}\cap(\llb x,y \rrb\times\llb 0,l \rrb)) \leq \frac{2}{3}W(\mathcal{R})
$
with
$(x,l)$ connected to $(y,l)$ through
$\xi$,
then we take $c = x$, $d = y$; that is, we use $\llb c,d \rrb = \llb x,y \rrb$ as our candidate disconnecting interval. 
Suppose otherwise that there does not exist any such boundary connection: then every
pair $(x,l),(y,l)\in \partial_\north \mathcal R$ connected through $\xi$ is such that
\begin{align}
\label{eq:assump-alg}
W(\mathcal{R}\cap(\llb x,y \rrb\times\llb 0,l \rrb))<\frac{1}{3}W(\mathcal{R})\,,\qquad\mbox{or}\qquad & W(\mathcal{R}\cap(\llb x,y \rrb\times\llb 0,l \rrb))>\frac{2}{3}W(\mathcal{R})\,.
\end{align}
If the latter holds, then there is a pair, say $(x_0,l),(y_0,l)\in \partial_\north \mathcal R$, for which the latter holds with a minimal width.
The interval
$\llb x_0, y_0\rrb$ would be a disconnecting interval of wired-type and
there is no other vertex $(z,l) \in \partial_\north \mathcal R$ with $z \in \llb x_0+1, y_0-1\rrb$
connected to $(x_0,l)$ and $(y_0,l)$ in $\xi$ since if there were such a $z$ it would violate the assumption that $\llb x_0, y_0\rrb$ is of minimal width with $W(\mathcal{R}\cap(\llb x_0,y_0 \rrb\times\llb 0,l \rrb))>\frac{2}{3}W(\mathcal{R})$
or that every pair of vertices  $(x,l),(y,l)\in\partial_{\north}\mathcal{R}$
connected in $\xi$ satisfy \eqref{eq:assump-alg}.
Consequently, 
all other connections through $\xi$ between vertices 
$(x_1,l),(y_1,l) \in \b_\north \mathcal{R}\cap(\llb x_{0}+1,y_{0}-1 \rrb\times\llb 0,l \rrb)$ will be such that
$$
W(\mathcal{R}\cap(\llb x_1,y_1 \rrb\times\llb 0,l \rrb))<\frac{1}{3}W(\mathcal{R})\,.
$$
We can then partition the vertices of $\partial_\north \mathcal R \cap (\llb x_0+1,y_0-1\rrb\times \{l\})$ into disjoint disconnecting intervals of  free-wired-type using the following procedure: 
\begin{enumerate}
	\setlength{\itemsep}{5pt}
\item Let $\rho=\{C_1,\dots,C_k\}$ be the partition of the vertices of $\partial_\north \mathcal R \cap (\llb x_0+1,y_0-1\rrb\times \{l\})$ induced by the boundary condition $\xi$; every $C_i$ corresponds to a distinct connected component of $\xi$;
 \item For each $C_i$, consider the disconnecting interval $L_i$ of free-wired-type determined by the left-most and right-most vertices of $C_i$ in $\partial_\north \mathcal R \cap (\llb x_0+1,y_0-1\rrb\times \{l\})$. Notice that some of the $C_i$'s may be singletons, which we view as  disconnecting intervals of the free-wired-type; 
\item Let $\{L_{i_1}, L_{i_2},\dots, L_{i_\ell}\}$ be those disconnecting intervals which are maximal, in the sense that there does not exist $j$ and $k$ such that $L_{i_j} \subset L_k$.
\end{enumerate} 
The set of disconnecting intervals $\{L_{i_1}, L_{i_2},\dots, L_{i_\ell}\}$ 
partitions $\llb x_0 + 1,y_0 - 1\rrb$ into disjoint disconnecting intervals of free-wired-type with the property that $W(\mathcal R \cap (L_{i_j}\times \llb 0,l\rrb)) \leq \frac{1}{3}W(\mathcal{R})$ for every $j \in \{1,\dots,\ell\}$. 
We can then use Lemma~\ref{lem:disconnecting-unions} to merge adjacent disconnecting intervals until we obtain a candidate disconnecting interval $\llb c,d \rrb\subset \llb x_0,y_0\rrb$ (of free-type), having width $W(\mathcal R \cap (\llb c,d \rrb\times \llb 0,l \rrb) \in [\frac{1}{3} W(\mathcal R), \frac 23 W(\mathcal R)]$. 

Now that we have found a candidate disconnecting interval $\llb c,d \rrb$ satisfying~\eqref{eq:slit}, we modify it to obtain a disconnecting interval $\llb c_\star,d_\star\rrb$ with the property that both $(c_\star,l)$ and $(d_\star,l)$ are distance at least $m$ from $\partial_\parallel \mathcal R$.

If $(c,l)$ is at distance at least $m$ from $\partial_\parallel \mathcal R$, set $c_\star= c$, and similarly if $(d,l)$ is at distance at least $m$ from $\partial_\parallel \mathcal R$, then set $d_\star= d$. 
Otherwise, suppose $(c,l)$ is at distance less than
$m$ from $\partial_{\west}R_{i}$
for some constituent rectangular subset $R_i= \llb a_i,b_i\rrb \times \llb 0,l\rrb$ of $\mathcal R$.
Since $\mathcal{R}$ is compatible with $\xi$,
the interval $\mathcal{I}_c = \llb b_{i-1}-m,a_i+m\rrb$ is a disconnecting interval,
and we set 
\begin{align*}
c_\star= \Bigg\{\begin{array}{ll}
        a_{i}+m\,, & \text{ if } \mathcal{I}_c\, \mbox{ is of wired-type, or $i=1$, or $W(R_i)=2m$;} \\
        a_i+m+1\,, & \text{ if } \mathcal{I}_c\, \mbox{ is only of free-type, and $W(R_i) > 2m$;}
        \end{array}
\end{align*}
note that the first case includes when $\mathcal I_c$ is of free-wired-type, whereas the second case only applies when the interval is of free-type and not of wired-type. 
When $(c, l)$ is instead at distance less than $m$ from $\partial_\east R_i$ for some $i$, then we simply set $c_\star= b_i -m$. 

Symmetrically, if $(d,l)$ is at distance less than $m$ from $\partial_\east R_i$ for some $R_i= \llb a_i,b_i\rrb \times \llb 0,l\rrb$, let $\mathcal{I}_d = \llb b_{i}-m,a_{i+1}+m\rrb$
\begin{align*}
d_\star= \Bigg\{\begin{array}{ll}
        b_{i}-m\,, & \text{ if } \mathcal{I}_d\, \mbox{ is of wired-type, or $i=N(\mathcal R)$, or $W(R_i)=2m$;} \\ 
        b_i-m-1 \,, & \text{ if } \mathcal{I}_d\, \mbox{ is only of free-type, and $W(R_i) > 2m$.}
        \end{array}
\end{align*} 
When $(d,l)$ is at distance less than $m$ from $\partial_\west R_i$, let $d_\star= a_i +m$. To see that this process is well-defined, notice that since $W(R_i)\geq 2m$ for every $i$, the points $(c,l), (d,l)$ cannot be both less than $m$ away from $\partial_\east R_i$ and less than $m$ away from $\partial_\west R_i$.

We claim that in all of these cases the interval $\llb c_\star,d_\star \rrb$ is a disconnecting interval. The fact that $(c_\star,l), (d_\star,l)\in \partial_\north \mathcal R$ are a distance at least $m$ away from $\partial_\parallel \mathcal R$ follows directly from the construction. 

First, observe that when $(c,l), (d,l)$ are both a distance at least $m$ from $\partial_\parallel \mathcal R$, then we set $c_\star= c$ and $d_\star= d$; in this case $\llb c_\star, d_\star\rrb$ is disconnecting since $\llb c, d\rrb$ was chosen to be disconnecting.

 Next suppose that $d$ was at least a distance $m$ from $\partial_\parallel \mathcal R$ while $c$ was at distance less than $m$ from $\partial_\west R_i$ for some $i$. In this setting, we establish that $\llb c_\star,d_\star\rrb$ is a disconnecting interval by considering the following four cases:
 
 \smallskip 
\emph{Case 1: $i=1$.} \ \ Note that $\llb a_1 + m, b_{N(\mathcal R)}-m\rrb$ and $\llb c, d \rrb$ are disconnecting intervals by compatibility and construction, respectively. 
Also, 
$W(\mathcal R\cap (\llb c, d\rrb)) \geq \frac{100 m}{3}$ since $W(\mathcal R)\geq 100m$. Hence, 
$ c < a_1+m < d<b_{N(\mathcal R)} - m$ and Lemma~\ref{lem:disconnecting-intersections} implies that $\llb c_\star,d_\star \rrb= \llb a_1 +m , d\rrb$ is disconnecting. 

\smallskip 
\emph{Case 2: $i>1$ and $\llb b_{i-1} - m, a_i +m\rrb$ is of wired-type.} \ \  Since $W(\mathcal R\cap (\llb c, d\rrb)) \geq \frac{100 m}{3}$, we have $b_{i-1}-m<c <a_i+m < d$. Then, by Lemma~\ref{lem:disconnecting-intersections-type}, $\llb c,d\rrb$ is a disconnecting interval of wired-type, and Lemma~\ref{lem:disconnecting-intersections} implies that $\llb c_\star,d_\star \rrb  = \llb a_{i}+m, d\rrb$ is a disconnecting interval of wired-type. 
 
 \smallskip 
 \emph{Case 3: $i>1$, $\llb b_{i-1} - m, a_i +m\rrb$ is of free-type and $W(R_i) > 2m$.} \ \
We again have $b_{i-1}-m <c <a_i +m <d$, and by Lemma~\ref{lem:disconnecting-intersections-type} $\llb c, d \rrb$ is a disconnecting interval of free-type. Therefore, by Lemma~\ref{lem:disconnecting-intersections}, $\llb c_\star , d_\star\rrb = \llb a_i + m+1 , d\rrb$ is a disconnecting interval of free-type.  

\smallskip \emph{Case 4: $i>1$, $\llb b_{i-1} - m, a_i +m\rrb$ is only of free-type and $W(R_i) = 2m$.} \ \
 In this case, it must be that $i <N(\mathcal R)$. Therefore, by the compatibility of $\mathcal R$ and $\xi$, $\llb b_i - m, a_{i+1}+m \rrb = \llb a_i+m, a_{i+1}+m\rrb$ is also a disconnecting interval.  In fact, by Lemma~\ref{lem:disconnecting-intersections-type},
 $\llb b_i - m, a_{i+1}+m \rrb$ is a disconnecting interval of free-type. Applying Lemma~\ref{lem:disconnecting-intersections} with respect to $\llb a_i + m+1 , d\rrb$ and $\llb a_i + m, a_{i+1}+m\rrb$ implies $\llb c_\star,d_\star\rrb = \llb a_i+m, d \rrb$ is a disconnecting interval.

 \smallskip 
Suppose otherwise that $c$ was at distance less than $m$ from $\partial_\east R_i$ for some $i$, and $d$ was still at least a distance at least $m$ from $\partial_\parallel \mathcal R$. In this case, $W(\mathcal R\cap (\llb c, d\rrb)) \geq \frac{100 m}{3}$ implies $N(\mathcal R)>1$ as well as $i<N(\mathcal R)$. 
Moreover, $\llb b_{i} - m, a_{i+1} +m\rrb$ is a disconnecting interval by the compatibility of $\mathcal R$ and $\xi$.
Since $b_{i}-m<c <a_{i+1}+m < d$,
Lemma~\ref{lem:disconnecting-intersections-type} then implies that 
when $\llb b_{i} - m, a_{i+1} +m\rrb$ is of wired-type (resp., of free-type)
then $\llb c,d\rrb$ is also of wired-type (resp., of free-type), and therefore by Lemma~\ref{lem:disconnecting-intersections}, $\llb c_\star,d_\star \rrb  = \llb b_{i}-m, d\rrb$ is a disconnecting interval of wired-type (resp., of free-type).

 The symmetric cases where $(c,l)$ is a distance at least $m$ from $\partial_\parallel \mathcal R$, and $(d,l)$ is a distance less than $m$ from $\partial_\parallel \mathcal R$ can be checked analogously. The remaining case in which both $( c,l)$ and $(d,l)$ are a distance less than $m$ from $\partial_\parallel \mathcal R$ can also be checked similarly, by first modifying the candidate interval on one side in order to obtain a disconnecting interval $\llb c_\star, d \rrb$ having $(c_\star,l)$ a distance at least $m$ from $\partial_\parallel\mathcal R$, and then performing the modification on $d$ to obtain the desired disconnecting interval $\llb c_\star,d_\star\rrb$.

Finally, we claim that in all such situations, $\llb c_\star,d_\star \rrb$ satisfies
\begin{align*}
\frac{1}{4}W(\mathcal{R})\leq W(\mathcal{R}\cap(\llb c_\star,d_\star \rrb\times \llb 0,l \rrb))\leq\frac{3}{4}W(\mathcal{R})\,.
\end{align*}
This follows from the facts that $W(\mathcal{R})\geq100m$, $|c-c_\star| \le m$ and $|d-d_\star|\leq m$. 
\end{proof}

We will now use the disconnecting interval given by Lemma \ref{lem:splitting-algorithm}
to define two subsets $\mathcal{A}_{\interior}$ and $\mathcal{A}_{\ext}$ of
$\mathcal{R}$, and set $\mathcal{R}_{\interior}$ and $\mathcal{R}_{\ext}$
to be their enlargments by $m$. 
\begin{definition}
\label{def:subblocks}For a group of rectangles $\mathcal{R}$ compatible
with $\xi$, let $\llb c_\star,d_\star \rrb$ be the disconnecting interval given by Lemma
\ref{lem:splitting-algorithm}. Then define the \emph{interior and exterior cores} as 
$$\mathcal{A}_{\interior}=\mathcal{R}\cap(\llb c_\star,d_\star \rrb\times\llb 0,l \rrb) \qquad \mbox{and} \qquad \mathcal{A}_{\ext}=\mathcal{R}\cap(\llb c_\star,d_\star \rrb^{c}\times\llb 0,l \rrb)\,;$$ 
see Figure~\ref{figure:blocks}(a). Using the disconnecting interval $\llb c_\star,d_\star \rrb$ define the \emph{interior
and exterior blocks} as follows: 
\begin{enumerate}
\item Let $c_{-}=c_\star-m$ and $c_{+}=c_\star+m$. Let $d_{-}=d_\star-m$ and $d_{+}=d_\star+m$.
\item Define $\mathcal{R}_{\interior}=\mathcal{R}\cap(\llb c_{-},d_{+} \rrb\times\llb 0,l \rrb)$
and $\mathcal{R}_{\ext}=\mathcal{R}\cap((\llb 0,c_{+} \rrb\cup\llb d_{-},n \rrb)\times\llb 0,l \rrb)$; see Figure \ref{figure:blocks}(b).
\end{enumerate}
\end{definition}

\begin{figure}[t]
	\begin{subfigure}[b]{0.49\textwidth}
		
		\begin{center}
			\begin{tikzpicture}
				\draw[draw=black] (0,0) rectangle (3,1);
				\draw[draw=black] (4,0) rectangle (7,1);
				
				\draw[draw=black,densely dashed] (1.5,0)--(1.5,1);
				\draw[draw=black,densely dashed] (5.5,0)--(5.5,1);
				\draw[draw=gray!90,<->] (1.5,0.6) -- (3,0.6);
				\draw[draw=gray!90,<->] (4,0.6) -- (5.5,0.6);
				\draw[draw=gray!90,<->] (0,0.4) -- (1.5,0.4);
				\draw[draw=gray!90,<->] (5.5,0.4) -- (7,0.4);
				\draw [rounded corners]  (1.5,1)--(1.5,1.4)--(5.5,1.4)--(5.5,1);
				\draw [rounded corners]  (2.5,1)--(2.5,1.2)--(4.5,1.2)--(4.5,1);
				\draw [rounded corners]  (0.2,1)--(0.2,1.6)--(6.8,1.6)--(6.8,1);
				\node at (2.25,0.75) {{\tiny $\mathcal{A}_\interior$}};
				\node at (4.75,0.75) {{\tiny $\mathcal{A}_\interior$}};
				\node at (0.75,0.55) {{\tiny $\mathcal{A}_\ext$}};
				\node at (6.25,0.55) {{\tiny $\mathcal{A}_\ext$}};
			\end{tikzpicture}
			\caption{}
			
		\end{center}
	\end{subfigure}
	\begin{subfigure}[b]{0.49\textwidth}
		
		\begin{center}

			\begin{tikzpicture}
			\draw[draw=black] (0,0) rectangle (3,1);
			\draw[draw=black] (4,0) rectangle (7,1);
			
			\draw[draw=black,densely dashed] (1.25,0)--(1.25,1);
			\draw[draw=black,densely dashed] (1.75,0)--(1.75,1);
			
			\draw[draw=black,densely dashed] (5.25,0)--(5.25,1);
			\draw[draw=black,densely dashed] (5.75,0)--(5.75,1);
	
			\draw[draw=gray!90,<->] (1.25,0.6) -- (3,0.6);
			\draw[draw=gray!90,<->] (4,0.6) -- (5.75,0.6);
			\draw[draw=gray!90,<->] (0,0.2) -- (1.75,0.2);
			\draw[draw=gray!90,<->] (5.25,0.2) -- (7,0.2);
			\draw [rounded corners]  (1.5,1)--(1.5,1.4)--(5.5,1.4)--(5.5,1);
			\draw [rounded corners]  (2.5,1)--(2.5,1.2)--(4.5,1.2)--(4.5,1);
			\draw [rounded corners]  (0.2,1)--(0.2,1.6)--(6.8,1.6)--(6.8,1);
			\node at (2.125,0.75) {{\tiny $\mathcal{R}_\interior$}};
			\node at (4.875,0.75) {{\tiny $\mathcal{R}_\interior$}};
			\node at (0.875,0.35) {{\tiny $\mathcal{R}_\ext$}};
			\node at (6.125,0.35) {{\tiny $\mathcal{R}_\ext$}};
			\end{tikzpicture}
			\caption{}
			
		\end{center}
	\end{subfigure}
	\caption{{{\rm (a)} The cores $\mathcal{A}_\interior$ and $\mathcal{A}_\ext$} and {{\rm (b)} the blocks $\mathcal{R}_\interior$ and $\mathcal{R}_\ext$. The blocks $\mathcal R_{\interior}$ and $\mathcal R_{\ext}$ are the enlargements of $\mathcal A_{\interior}$ and $\mathcal A_\ext$ by exactly $m$, and are thus, themselves, groups of rectangles.}}     
	\label{figure:blocks}
\end{figure}
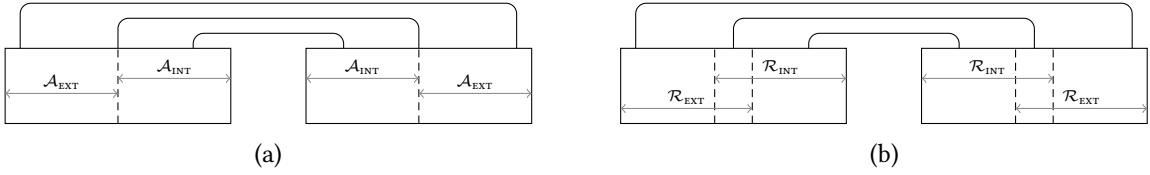

We will demonstrate that this choice of $\mathcal R_{\interior}$ and $\mathcal R_\ext$ has certain fundamental properties that will facilitate the recursive argument in the proof of Theorem~\ref{thm:main-thin-rect}.  
\begin{proposition}
\label{prop:subblock-properties}
If $\mathcal{R}$
is a group of rectangles compatible with $\xi$, and moreover, $W(\mathcal{R})\geq 100m$,
then the sets $\mathcal{R}_{\interior}$ and $\mathcal{R}_{\ext}$
are groups of rectangles satisfying the following properties: 
\begin{enumerate}
\item $\frac{1}{5}W(\mathcal{R})\leq W(\mathcal{R}_{\interior})\leq\frac{4}{5}W(\mathcal{R})$
and likewise $\frac{1}{5}W(\mathcal{R})\leq W(\mathcal{R}_{\ext})\leq\frac{4}{5}W(\mathcal{R})$; 
\item Both $\mathcal{R}_{\interior}$
and $\mathcal{R}_{\ext}$ are compatible with $\xi$.
\end{enumerate}
\end{proposition}
\begin{proof}
We show first that $\mathcal R_{\interior}$ and $\mathcal R_\ext$ are groups of rectangles (see Definition~\ref{def:group-of-rectangles}). 
By construction $\mathcal{R}_{\interior}$ and $\mathcal{R}_{\ext}$ are
unions of rectangular subsets of $\L_{n , l}$. Since $\mathcal{R}$
is a group of rectangles, every constituent
rectangular subset $R_i$ of $\mathcal R$ has $W(R_{i})\geq 2m$ by definition. By Lemma~\ref{lem:splitting-algorithm}, $(c_\star,l),(d_\star,l)$ were such that they are a distance at least $m$ from $\partial_\parallel \mathcal R$; as a consequence $\llb c_-, c_+\rrb \times \llb 0,l\rrb$ and $\llb d_-, d_+\rrb \times \llb 0,l\rrb$ are subsets of $\mathcal R$. Moreover, every constituent rectangular subset of $\mathcal R_{\interior}$ and $\mathcal R_{\ext}$ is either a constituent rectangle of $\mathcal R$, or contains one of $\llb c_-, c_+\rrb \times \llb 0,l\rrb$ and $\llb d_-, d_+\rrb \times \llb 0,l\rrb$, implying that every rectangular subset of $\mathcal R_{\interior}$ and $\mathcal R_\ext$ has width at least $2m$. 
Together, these verify that $\mathcal R_{\interior}$ and $\mathcal R_\ext$ are indeed groups of rectangles. 

The first property follows from the facts that, by Lemma~\ref{lem:splitting-algorithm}, $\frac{1}{4}W(\mathcal{R})\leq W(\mathcal{A}_{\interior})\leq\frac{3}{4}W(\mathcal{R})$
and $\frac{1}{4}W(\mathcal{R})\leq W(\mathcal{A}_{\ext})\leq\frac{3}{4}W(\mathcal{R})$,
while
$W(\mathcal{R})\geq100m$, $W(\mathcal{R}_{\interior})-W(\mathcal{A}_{\interior})=2m$ and $W(\mathcal{R}_{\ext})-W(\mathcal{A}_{\ext})=2m$.

To verify the second property, consider $\mathcal{R}_{\ext}$ first and label its rectangular subsets $R^{\ext}_1,...,R^\ext_{N(R_\ext)}$ (from left to right) with $R^\ext_j = \llb a^\ext_j, b^\ext_j\rrb$. Let $i,i+1$ be the indices of the two distinct rectangular subsets containing $\mathcal{R}_{\ext}\setminus\mathcal{A}_{\ext}$. As before, $R_1,..., R_{N(\mathcal R)}$ are the constituent rectangular subsets of $\mathcal R$. Then, for every $j\in \{1,..., N(\mathcal R_\ext)-1\}\setminus \{i\}$, there is a $k\in \{1,..,N(\mathcal R)-1\}$ such that 
$
b^\ext_j =  b_k
$
and
$ a^\ext_{j+1}= a_{k+1}$.
Hence, the compatibility of $\mathcal R$ with $\xi$ guarantees that the interval $\llb b_j^\ext - m, a_{j+1}^\ext+m\rrb$ is disconnecting for every $j\neq i$. 

To see that $\llb b_i^\ext - m, a_{i+1}^\ext+m \rrb$ is disconnecting, notice that by construction of $\mathcal R_\ext$, $\llb b_i^\ext - m, a_{i+1}^\ext+m \rrb = \llb c_\star,d_\star\rrb$.  Finally, since $(c_\star,l), (d_\star,l)$ were at distance at least $m$ from $\partial_\parallel \mathcal R$, the interval $\llb a^\ext_1 +m, b^\ext_{N(\mathcal R_\ext)} -m\rrb$ matches the interval $\llb a_1+m, b_{N(\mathcal R)} - m\rrb$, and thus by compatibility of $\mathcal R$ with $\xi$ implies the former is a disconnecting interval. Altogether these imply the compatibility of $\mathcal R_\ext$ with~$\xi$.

Similarly, label the constituent rectangles of $\mathcal R_\interior$ as $R_1^\interior, ..., R_{N(\mathcal R_\interior)}^\interior$ and notice that $1$ and $N(\mathcal R_\interior)$ are the indices of the two rectangles containing $\mathcal R_\interior \setminus \mathcal A_\interior$. Every interval of the form $\llb b^\interior_i -m, a^\interior_{i+1} +m\rrb$ corresponds (up to change of index) to such an interval for $\mathcal R$, so that by compatibility of $\mathcal R$ with respect to $\xi$, these are all disconnecting intervals. The interval $\llb a^\interior_1 + m, b^\interior_{N(\mathcal R_\interior)}- m\rrb$ is exactly the interval $\llb c_\star,d_\star\rrb$ given by Lemma~\ref{lem:splitting-algorithm}, and therefore this is disconnecting by construction. Together, these imply the compatibility of $\mathcal R_\interior$ with $\xi$. 
\end{proof}

\subsection{Block dynamics coupling time}
\label{subsec:block-dynamics}

Here we consider the block dynamics on $\mathcal{R}$ with blocks $\mathcal{R}_{\interior}$
and $\mathcal{R}_{\ext}$ as defined in Section~\ref{subsec:sub-blocks}; see also Figure~\ref{figure:blocks}(b). We begin by defining the block dynamics for a group of rectangles. 

\begin{definition}\label{def:block-dynamics-2}
Let $\mathcal R$ be a group of rectangles with boundary condition $\zeta$.
The block dynamics $\{X_t\}$ 
with blocks $\mathscr{B} = \{\mathcal B_1,\ldots,\mathcal B_k\}$ such that $\mathcal B_i \subset \mathcal R$ and $\bigcup_{i=1}^k  E(\mathcal B_i)=E(\mathcal R)$
 is the discrete-time Markov chain that at each $t$, selects $i$ uniformly at random from $\{1,...,k\}$ and updates the configuration on $E(\mathcal B_i)$ from the stationary distribution conditional on the configuration $X_t(E^c(\mathcal B_i))$. 
 \end{definition}

\noindent Fix $\mathscr{B} = \{\mathcal{R}_{\interior},\mathcal{R}_{\ext}\}$ and let $\gap(\mathcal{R^\zeta};\mathscr{B})$ be the spectral gap of
this block dynamics on $\mathcal R$ with boundary condition $\zeta=(\xi,\omega_{\mathcal{R}^c})$,
where $\xi$ is a realizable boundary condition on $\L_{n, l}$ and $\omega_{\mathcal{R}^c}$ is a configuration on $E^c(\mathcal{R})$.

\begin{lemma}
\label{lem:block-dynamics-gap}
Let $\xi$ be a realizable boundary condition 
on $\b\Lambda_{n,l}$ that is free on $\b_{\south}\Lambda_{n,l} \cup \b_{\east}\Lambda_{n,l} \cup \b_{\west}\Lambda_{n,l}$
and free on  vertices in $\b_{\north}\Lambda_{n,l}$ at distance at most $m$ from $\partial_{\east}\Lambda_{n,l}\cup\partial_{\west}\Lambda_{n,l}$.
For every $q>1$ and $p\neq p_c(q)$, there exists $K= K(p,q)\ge1$ such that for every group of rectangles $\mathcal R$ compatible with $\xi$, and every configuration $\omega_{\mathcal R^c}$, 
\begin{align*}
\gap(\mathcal R^{(\xi,\omega(\mathcal R ^c))};\mathscr{B}) & \geq K^{-1}\,.
\end{align*}
\end{lemma}

\begin{proof}
	We consider the $p < p_c(q)$ case first.
	Let $\{X_{t}\},\{Y_{t}\}$ be two instances of the block dynamics on $\mathcal{R}$
	with boundary condition $\zeta = (\xi,\omega_{\mathcal{R}^{c}})$ started from initial configurations $X_0,Y_0$. 
	We design a coupling $\P$ for the steps of $\{X_{t}\}$ and $\{Y_{t}\}$ and bound its coupling time.
	This yields upper bounds for both the mixing time and the inverse spectral gap of the block dynamics; see Section \ref{sec:prelim} for a brief overview of the coupling method.
	For this, we will show that for any two initial configurations $X_0,Y_0$ 
	\begin{align}
	\label{eq:main:coupling}
	\mathbb{P}(X_{2}=Y_{2}) = \Omega(1) & \,.
	\end{align}
	Since this bound will be uniform over $X_{0},Y_{0}$, we can make independent
	attempts at coupling the two chains every two steps.
	Hence,
	there would exist 
	$T = O(1)$ such that 
	$$\max_{X_{0},Y_{0}}\mathbb{P}(X_{T}\neq Y_{T}) \leq  \frac 14\,,$$
	bounding the coupling time by $T=O(1)$ concluding the proof.
	
	First observe that with probability ${1}/{4}$ the first block to
	be updated is $\mathcal{R}_{\interior}$ and the second is $\mathcal{R}_{\ext}$.
	Suppose this is the case and let us consider the update on block $\mathcal{R}_{\interior}$.
	Observe that $X_1(\mathcal{R}_{\interior})$ 
	is distributed according to the random-cluster measure $\pi^{\theta_X}$ 
	on $\mathcal{R}_{\interior}$,
	where $\theta_X$ 
	is the boundary condition induced on $\b \mathcal{R}_{\interior}$
	by the boundary condition $\zeta$ on $\mathcal{R}$
	and the configuration of $X_0$ in $E(\mathcal{R})\setminus E(\mathcal{R}_{\interior})$. 
	Likewise, $Y_1(\mathcal{R}_{\interior})$ has law $\pi^{\theta_Y}$, with the boundary condition
	$\theta_Y$
	defined analogously but considering the configuration of $Y_0$ in $E(\mathcal{R})\setminus E(\mathcal{R}_{\interior})$ instead.
	Therefore, any coupling for the random-cluster measures $\pi^{\theta_X}$, $\pi^{\theta_Y}$
	yields a coupling for the first steps of $\{X_t\}$ and $\{Y_t\}$. 
	  
	Let $\theta_1$ be the boundary condition on $\b\mathcal{R}_{\interior}$ induced by $\zeta$ and the configuration that is all wired on $E(\mathcal{R})\setminus E(\mathcal{R}_{\interior})$;
	let $\pi^{\theta_1}$ be the corresponding random-cluster measure on $\mathcal{R}_{\interior}$.
	Let $Q_{\west}, Q_{\east} \subset \mathcal{R}_{\interior}$
	be the two rectangles of width $m$ that contain all the vertices in $\mathcal{R}_{\interior}\setminus\mathcal{A}_{\interior}$; i.e.,
	$Q_{\west} \cup \mathcal A_\interior \cup Q_{\east} = \mathcal R_\interior$, $Q_{\west} \cap \mathcal A_\interior = \emptyset$ and $Q_{\east} \cap \mathcal A_\interior = \emptyset$ (see Figure \ref{figure:coupling}(a)).
	Let $\b E(Q_{\west})$ be the set edges with one endpoint in $Q_{\west}$ and the other in $\mathcal A_\interior$, and similarly define~$\b E(Q_{\east})$.

	Let $\Gamma_{\west}$ be the set of configurations in $\mathcal{R}_\interior$ that have a \textit{dual-path} in $E(Q_{\west}) \cup \b E(Q_{\west})$ 
	connecting the top-most edge in $\b E(Q_{\west})$ to an edge in $\b_{\south}Q_{\west}$, and similarly define $\Gamma_{\east}$ as the set of configurations in $\mathcal{R}_\interior$ that have a dual-path in $E(Q_{\east}) \cup \b E(Q_{\east})$ 
	from the top-most edge in $\b E(Q_{\east})$ to an edge in $\b_{\south}Q_{\east}$.
	(A dual-path is an open path in the dual configuration.)
	Let $\Gamma=\Gamma_{\east}\cap\Gamma_{\west}$; see Figure \ref{figure:coupling}(b).
	The following lemma supplies the desired coupling.
			
	\begin{figure}[t]
		\begin{subfigure}[b]{0.32\textwidth}
			
			\begin{center}
				\begin{tikzpicture}
				\draw[draw=black] (0,0) rectangle (5,1.5);
				\draw[draw=black,densely dashed] (1,0)--(1,1.5);
				\draw[draw=black,densely dashed] (4,0)--(4,1.5);
				\draw[draw=gray!90,<->] (1,1) -- (4,1);
				\draw[draw=gray!90,<->] (4,0.5) -- (5,0.5);
				\draw[draw=gray!90,<->] (0,0.5) -- (1,0.5);
				\draw [rounded corners]  (1,1.5)--(1,2)--(4,2)--(4,1.5);
				\node at (2.5,1.15) {{\tiny $\mathcal{A}_\interior$}};
				\node at (0.5,0.65) {{\tiny $Q_\west$}};
				\node at (4.5,0.65) {{\tiny $Q_\east$}};
				\end{tikzpicture}
				\caption{}
				
			\end{center}
		\end{subfigure}
		\begin{subfigure}[b]{0.32\textwidth}
			
			\begin{center}
				\begin{tikzpicture}
				\draw[draw=black] (0,0) rectangle (5,1.5);
				\draw [rounded corners]  (1,1.5)--(1,2)--(4,2)--(4,1.5);
				\draw[color=black,densely dotted,decoration={snake, segment length=3mm, amplitude=0.5mm}] (0.5,0)decorate{--(0.7,0.5)--(0.5,1)--(0.8,1.2)--(0.95,1.5)};
				\draw[color=black,densely dotted,decoration={snake, segment length=3mm, amplitude=0.5mm}] (0.5,0)decorate{--(0.7,0.5)--(0.5,1)--(0.8,1.2)--(0.95,1.5)};
				\draw[color=black,densely dotted,decoration={snake, segment length=3mm, amplitude=0.5mm}] (4.2,0) decorate{--(4.5,0.6)--(4.5,1)--(4.05,1.5)};
				\draw[draw=black,densely dotted,decorate,decoration={snake, segment length=3mm, amplitude=0.5mm}] (4.2,0)--(4.5,0.6)--(4.5,1)--(4.05,1.5);
				\node at (2.5,0.9) {{\tiny $\mathcal{A}_\interior$}};
				\draw[draw=gray!90,<->] (1,0.75) -- (4,0.75);
				\draw[draw=black,densely dashed] (1,0)--(1,1.5);
				\draw[draw=black,densely dashed] (4,0)--(4,1.5);
				\end{tikzpicture}	
				
				\caption{}
			\end{center}
		\end{subfigure}
		\begin{subfigure}[b]{0.32\textwidth}
			
			\begin{center}
				\begin{tikzpicture}
				\draw[draw=black] (0,0) rectangle (5,1.5);

				\draw [rounded corners]  (1,1.5)--(1,2)--(4,2)--(4,1.5);
				
				\draw[draw=black,densely dotted,decorate,decoration={snake, segment length=3mm, amplitude=0.5mm}] (0.5,1)--(0.46,1.5);
				\draw[draw=black,densely dotted,decorate,decoration={snake, segment length=3mm, amplitude=0.5mm}] (0.5,1)--(0.05,1.5);
				\draw[color=black,densely dotted,decoration={snake, segment length=3mm, amplitude=0.5mm}] (0.5,0)decorate{--(0.7,0.5)--(0.5,1)--(0.8,1.2)--(0.95,1.5)};
				\draw[color=black,densely dotted,decoration={snake, segment length=3mm, amplitude=0.5mm}] (0.5,0)decorate{--(0.7,0.5)--(0.5,1)--(0.8,1.2)--(0.95,1.5)};
				
				\draw[draw=black,densely dotted,decorate,decoration={snake, segment length=2.8mm, amplitude=0.5mm}] (4.5,1)--(4.95,1.5);
				\draw[draw=black,densely dotted,decorate,decoration={snake, segment length=2.8mm, amplitude=0.5mm}] (4.5,1)--(4.5,1.5);
				\draw[color=black,densely dotted,decoration={snake, segment length=3mm, amplitude=0.5mm}] (4.2,0) decorate{--(4.5,0.6)--(4.5,1)--(4.05,1.5)};
				\draw[draw=black,densely dotted,decorate,decoration={snake, segment length=3mm, amplitude=0.5mm}] (4.2,0)--(4.5,0.6)--(4.5,1)--(4.05,1.5);
				\node at (2.5,0.9) {{\tiny $\mathcal{A}_\interior$}};
				\draw[draw=gray!90,<->] (1,0.75) -- (4,0.75);
				\draw[draw=black,densely dashed] (1,0)--(1,1.5);
				\draw[draw=black,densely dashed] (4,0)--(4,1.5);
				\end{tikzpicture}
				\caption{}
			\end{center}
		\end{subfigure}
		\caption{{{\rm (a)} The block $\mathcal{R}_{\interior}$ with its subsets $\mathcal{A}_{\interior}$, $Q_\west$ and $Q_\east$. } {{\rm (b)}  The block $\mathcal{R}_{\interior}$ with the dual-paths (dotted) of a configuration in $\Gamma$. }{{\rm (c)} The block $\mathcal{R}_{\interior}$ with the dual-paths (dotted) of a configuration in $\Gamma_{\west}^{\north\south} \cap \Gamma_{\west}^{\west\east} \cap \Gamma_{\east}^{\north\south} \cap \Gamma_{\east}^{\west\east} \subset \Gamma$.}}     
		\label{figure:coupling}
	\end{figure}

	\begin{lemma}
		\label{lemma:main:coupling}
		Let $q>1$ and $p<p_c(q)$. There exists a coupling $\P_1$ of the distributions
		$\pi^{\theta_X}$, $\pi^{\theta_Y}$, $\pi^{\theta_1}$
		such that if $(\omega^{\theta_X},\omega^{\theta_Y},\omega^{\theta_1})$ is sampled from $\P_1$, then all of the following hold:
		\begin{enumerate}
			\item $\P_1(\omega^{\theta_X},\omega^{\theta_Y},\omega^{\theta_1}) > 0$
			only if $\omega^{\theta_X} \le \omega^{\theta_1}$  and $\omega^{\theta_Y} \le \omega^{\theta_1}$;
			\item 	$\P_1(\omega^{\theta_X}(\mathcal{A}_\interior) = \omega^{\theta_Y}(\mathcal{A}_\interior) \mid \omega^{\theta_1} \in \Gamma) =1$;
			\item There exists a constant $\rho = \rho(p,q) > 0$ such that $\P_1(\omega^{\theta_1} \in \Gamma) \ge \rho$.
		\end{enumerate}

	\end{lemma}
	
	Hence, if we use the coupling $\P_1$ from Lemma \ref{lemma:main:coupling} to couple the first step of the chains, then $X_1$ and $Y_1$ will agree on $E(\mathcal{A}_{\interior})$
	with probability at least $\rho > 0$. If this occurs,
	then we can easily couple the update on $\mathcal{R}_{\ext}$ in the second step so 
	that $X_2 = Y_2$,
	since $X_1(E(\mathcal{A}_{\interior})) = Y_1(E(\mathcal{A}_{\interior}))$ implies $X_1(E(\mathcal{R})\setminus E(\mathcal{R}_{\ext})) = Y_1(E(\mathcal{R})\setminus E(\mathcal{R}_{\ext}))$,
	and thus the boundary conditions induced by the two instances of the chain on $\mathcal{R}_{\ext}$
	are identical. As a consequence, we obtain that for any $X_{0},Y_{0}$,
	\begin{align*}
	\mathbb{P}(X_{2}=Y_{2})\geq & \frac{\rho}{4}\,.
	\end{align*}
	which gives \eqref{eq:main:coupling} and thus concludes the proof for $p < p_c(q)$. 
	
	The case when $p>p_c(q)$ follows by an analogous dual argument.
	In this case,  the set $\Gamma$ has $o(1)$ probability and we therefore replace it by the set $\Gamma^*=\Gamma_\east^* \cap \Gamma_\west^*$, where $\Gamma_\west^*$ is the set of configurations in $\mathcal{R}_\interior$ that have an open path in $E(Q_\west)\cup \partial E(Q_\west)$ connecting the top-left corner of $\mathcal{A}_{\interior}$
	$(c_\star,l)$ to $\partial_\south Q_\west$; similarly,  $\Gamma_\east^*$ is the set of configurations in $\mathcal{R}_\interior$ that have an open path in $E(Q_\east) \cup \partial E(Q_\east)$ connecting the top-right corner $(d_\star,l)$ of $\mathcal{A}_{\interior}$ to $\partial_\south Q_\east$. 
	Let $\theta_0$ be the boundary condition on $\b\mathcal{R}_{\interior}$ induced by $\zeta$ and the all-free configuration on $E(\mathcal{R})\setminus E(\mathcal{R}_{\interior})$;
	let $\pi^{\theta_0}$ be the resulting random-cluster distribution on $\mathcal{R}_{\interior}$.
	The constant bound on the coupling time of the block dynamics would then follow as above from the following dual analogue of Lemma~\ref{lemma:main:coupling}. 
		
			\begin{lemma}
			\label{lemma:main:coupling:supercritical}
			Let $q>1$ and $p>p_c(q)$. There exists a coupling $\P_0$ of the distributions
			$\pi^{\theta_X}$, $\pi^{\theta_Y}$, $\pi^{\theta_0}$
			such that if $(\omega^{\theta_X},\omega^{\theta_Y},\omega^{\theta_0})$ is sampled from $\P_0$, then all of the following hold:
			\begin{enumerate}
				\item $\P_0(\omega^{\theta_X},\omega^{\theta_Y},\omega^{\theta_0}) > 0$
				only if $\omega^{\theta_X} \ge \omega^{\theta_0}$  and $\omega^{\theta_Y} \ge \omega^{\theta_0}$;
				\item 	$\P_0(\omega^{\theta_X}(\mathcal{A}_\interior) = \omega^{\theta_Y}(\mathcal{A}_\interior) \mid \omega^{\theta_0} \in \Gamma^*) =1$;
				\item There exists a constant $\rho = \rho(p,q) > 0$ such that $\P_0(\omega^{\theta_0} \in \Gamma^*) \ge \rho$.\qedhere
			\end{enumerate}
		\end{lemma}
\end{proof}

We proceed with the proof of the key Lemma~\ref{lemma:main:coupling}.
Parts 1 and 2 of the lemma follow from a fairly standard coupling technique; see, e.g., \cite{Alexander,BS}. We shall also use this approach later in the proof of Lemma~\ref{lemma:ssm:reduction}. Part 3 will be a consequence of the EDC property~\eqref{eq:EDC}; see proof of Claim~\ref{claim:coupling-proof}.

\begin{proof}[Proof of Lemma~\ref{lemma:main:coupling}]
	
	Let 
	$$
	L = \b_\west Q_\west \cup  \b_\north Q_\west \cup \b_\east Q_\east \cup \b_\north Q_\east;
	$$
	note that $L \cap \mathcal A_\interior = \emptyset$. For an FK configuration $\omega$ on $\mathcal{R}_{\interior}$ let
	\begin{equation}\label{def:F}
	F(\omega) := \mathcal{R}_{\interior} \setminus  \bigcup\nolimits_{v \in L}\, C(v,\omega)\,,\notag
	\end{equation}
	where $C(v,\omega)$ is the vertex set of the connected component of $v$ in $\omega$, ignoring the boundary connections.
	
	Clearly, 
	$\pi^{\theta_1}  \succeq \pi^{\theta_X}$ and $\pi^{\theta_1} \succeq \pi^{\theta_Y}$ and thus there exist monotone couplings $\P_X$ (resp., $\P_Y$) for $\pi^{\theta_X}$ and $\pi^{\theta_1}$ (resp., $\pi^{\theta_Y}$ and $\pi^{\theta_1}$). We use $\P_X$ and $\P_Y$ to construct the coupling $\P_1$ as follows: 
	\begin{enumerate}
		\item sample $(\omega^{\theta_X},\omega^{\theta_1})$ from $\P_X$ and sample $\omega^{\theta_Y}$ from $\P_Y(\,\cdot\mid \omega^{\theta_1}\,)$;
		\item if $\mathcal{A}_\interior \subseteq F(\omega^{\theta_1})$, 
		sample $\omega_\Delta$ from $\pi^{\eta_1}_{\Delta}$---where $\Delta=\Delta(\omega^{\theta_1})$ is the subgraph induced by
		$F(\omega^{\theta_1})$ and $\eta_1$ is the boundary condition on $\b F(\omega^{\theta_1})$ induced by $\theta_1$ and the configuration of $\omega^{\theta_1}$ on $E(\mathcal{R}_\interior) \setminus E(F(\omega^{\theta_1}))$---and 
		make $\omega^{\theta_1}(F(\omega^{\theta_1})) = \omega^{\theta_X}(F(\omega^{\theta_1})) = \omega^{\theta_Y}(F(\omega^{\theta_1})) = \omega_\Delta$. 
	\end{enumerate}
	Let $\P_1$ be the resulting distribution
	of $(\omega^{\theta_X},\omega^{\theta_Y},\omega^{\theta_1})$. After step (i), $\P_1$ has the desired marginals.
	Moreover,
	we claim that replacing the configuration in $F(\omega^{\theta_1})$ with $\omega_\Delta$ in step (ii) has no effect on the distribution, provided $\mathcal{A}_\interior \subseteq F(\omega^{\theta_1})$. 
	For this,
	we show that the three boundary conditions $\eta_1$, $\eta_X$, $\eta_Y$ induced on $\b F(\omega^{\theta_1})$
	by the configurations of $\omega^{\theta_X}$, $\omega^{\theta_Y}$, $\omega^{\theta_1}$ on $E(\mathcal{R}_\interior) \setminus E(F(\omega^{\theta_1}))$, respectively, and 
	the corresponding boundary conditions
	$\theta_X$, $\theta_Y$, $\theta_1$
	are all the same.

	First, observe that the boundary condition on $\b_{\south}\mathcal{A}_{\interior}$
	is, in all three cases, free by assumption.
	Also, 
	from the definition of $F(\omega^{\theta_1})$ it follows that 
	every edge of $E(\mathcal{R}_\interior)\setminus E(F(\omega^{\theta_1}))$ incident to $\b F(\omega^{\theta_1})$ is closed in $\omega^{\theta_1}$; hence the same holds for $\omega^{\theta_X}$ and $\omega^{\theta_Y}$.
	The remaining portion of $\partial F(\omega^{\theta_1})$ is precisely the set of vertices $(\partial \mathcal A_{\interior} \cap \partial \mathcal R) \setminus \partial_\south \mathcal R$. 
	To show that $\eta_1, \eta_X,\eta_Y$ also agree on $(\partial \mathcal A_{\interior} \cap \partial \mathcal R) \setminus \partial_\south \mathcal R$ we use the fact that top-left and top-right corners of $\mathcal A_{\interior}$ correspond to the endpoints of the disconnecting interval $\llb c_\star, d_\star \rrb$.
	Indeed, for the boundary conditions $\eta_1, \eta_X,\eta_Y$ to disagree on $(\partial \mathcal A_{\interior} \cap \partial \mathcal R) \setminus \partial_\south \mathcal R$ it must be the case that there are at least
	 two distinct connected components of $\zeta = (\xi, \omega_{\mathcal R^c})$ connecting $(\partial \mathcal A_{\interior} \cap \partial \mathcal R)\setminus \partial_\south \mathcal R$  and $\partial \mathcal R\setminus  \partial \mathcal A_{\interior}$. Since $(c_\star,l),(d_\star,l)\notin \partial_\parallel \mathcal R$ and $\xi$ is free on $\partial_\south \mathcal R$, this would require at least two distinct connected components of $\xi$ connecting vertices in $\llb c_\star, d_\star \rrb \times \{l\}$ to vertices in $\llb c_\star, d_\star\rrb^c \times \{l\}$. However, when $\llb c_\star, d_\star \rrb $ is a disconnecting interval of free-type, there are no such connected components, and when it is disconnecting of wired-type, the planarity of realizable boundary conditions implies that there can be at most one such connected component, which is exactly the component of $\xi$ containing both $(c_\star,l)$ and $(d_\star,l)$. 
	 			
	Altogether, these together imply that when $\mathcal{A}_\interior \subseteq F(\omega^{\theta_1})$, the three boundary conditions $\eta_1, \eta_X, \eta_Y$ induced on $\partial F(\omega^{\theta_1})$ are the same. 
	The {\it domain Markov property} of random-cluster measures (see, e.g., \cite{Grimmett}) then guarantees that replacing the configuration in $F(\omega^{\theta_1})$ with $\omega_\Delta$ had no effect on the distributions.

	Finally, note that if $\omega^{\theta_1} \in \Gamma$, then 
	the vertices in the boundary components of $L$, i.e., $\bigcup_{v \in L} C(v,\omega^{\theta_1})$,
	will be confined to $Q_\west \cup Q_\east$, in which case $\mathcal{A}_\interior \subseteq F(\omega^{\theta_1})$. This establishes parts 1 and 2, and part 3 follows directly from the following claim, which will conclude the proof.
	\end{proof}

	\begin{claim}
		\label{claim:coupling-proof}
		Let $q>1$ and $p<p_c(q)$. There exists $\rho = \rho(p,q) >0$ such that $\pi^{\theta_1}(\Gamma)\geq\rho$. 
	\end{claim}

	\begin{proof}
		
		Let	$\Gamma_{\west}^{\north\south}$ (resp., $\Gamma_{\east}^{\north\south}$) be the set of configurations on $\mathcal R_\interior$
		such that there is dual-path from $\b_\north Q_\west$ to $\b_\south Q_\west$ (resp., from $\b_\north Q_\east$ to $\b_\south Q_\east$) in $E(Q_\west) \cup \b E(Q_\west)$ (resp., $E(Q_\east) \cup \b E(Q_\east)$).
		Also, let $\Gamma_{\west}^{\west\east}$ (resp., $\Gamma_{\east}^{\west\east}$) be the set of configurations on $\mathcal R_\interior$
		such that there is dual-path in $E(Q_\west) \cup \b E(Q_\west)$ (resp., in $E(Q_\east) \cup \b E(Q_\east)$) between the left-most edge in $\b_\north Q_\west$  (resp., the right-most edge in $\b_\north Q_\east$)
		and the top-most edge of $\b E(Q_{\west})$ (resp., $\b E(Q_{\east})$); see Figure \ref{figure:coupling}(c).
		Note that $\Gamma_{\west}^{\north\south} \cap \Gamma_{\west}^{\west\east} \cap \Gamma_{\east}^{\north\south} \cap \Gamma_{\east}^{\west\east} \subset \Gamma$.
		
		The width of $Q_\west$ is $m = C_{\star} \log l$.
		Hence, the EDC property~\eqref{eq:EDC} implies that when $C_{\star}$ is large enough there exists a constant $\rho_0 = \rho_0(p,q)> 0$
		such that
		$\pi^{\theta_1}(\Gamma_{\west}^{\north\south})\geq\rho_0$ and similarly for $\Gamma_{\east}^{\north\south}$.
		The EDC property~\eqref{eq:EDC} also implies that $\pi^{\theta_1}(\Gamma_{\west}^{\west\east})\geq\rho_1$ and $\pi^{\theta_1}(\Gamma_{\east}^{\west\east})\geq\rho_1$, for a suitable $\rho_1 = \rho_1(p,q)> 0$. We justify this as follows. By the EDC property~\eqref{eq:EDC}, there exists some constant $D$ such that with probability $\Omega(1)$, no pair of vertices whose distance is at least $D$, one of which is in $\b_\north Q_\west \setminus \b_\west Q_\west$ and the other in $\b _\east Q_\west \cup \partial_\south Q_\west \cup \partial_\east Q_\west$, will be connected in $Q_\west$. At the same time, with probability $\Omega(1)$, we can force $O(D)$ edges bordering those pairs of vertices that are distance less than $D$ to be closed so that no pair of vertices 
		that are closer than $D$ are connected in $Q_\west$ either. The analogous reasoning holds for $\Gamma_{\east}^{\west\east}$ and $Q_\east$.
		Since each of the events $\Gamma_{\west}^{\north\south}, \Gamma_{\west}^{\west\east}, \Gamma_{\east}^{\north\south}, \Gamma_{\east}^{\west\east}$ are
		decreasing events, by the FKG inequality (see, e.g., \cite{Grimmett}) we get
		$$
		\pi^{\theta_1}(\Gamma) \ge \pi^{\theta_1}(\Gamma_{\west}^{\north\south} \cap \Gamma_{\west}^{\west\east} \cap \Gamma_{\east}^{\north\south} \cap \Gamma_{\east}^{\west\east})
		\geq 
		\pi^{\theta_1}(\Gamma_{\west}^{\north\south})
		\pi^{\theta_1}(\Gamma_{\west}^{\west\east})
		\pi^{\theta_1}(\Gamma_{\east}^{\north\south})
		\pi^{\theta_1}(\Gamma_{\east}^{\west\east}) 
		\ge \rho_0^2\rho_1^2,
		$$ 
		and thus we can take $\rho=\rho_0^2\rho_1^2$ to have the desired estimate. 
	\end{proof}

	\begin{proof}[Proof of Lemma~\ref{lemma:main:coupling:supercritical}]
	The proof of Lemma~\ref{lemma:main:coupling} carries to Lemma~\ref{lemma:main:coupling:supercritical} with certain natural modifications we describe next. As before, 
	we let
	$
	L = \b_\west Q_\west \cup  \b_\north Q_\west \cup \b_\east Q_\east \cup \b_\north Q_\east
	$
	and let $(L^*,E(L^*))$ be the dual-graph induced by the set of dual-edges intersecting $E(L)$ in $(\mathbb Z^2)^*$; its vertex set consists of exactly
	$2E(L)-1$ vertices, and we refer to those outside of $\mathcal{R}_{\interior}$ as $\partial_\ext L^*$. 
	Similarly, let $(\mathcal{R}_{\interior}^*,E(\mathcal{R}_{\interior}^*))$ and $(\mathcal{A}_{\interior}^*,E(\mathcal{A}_{\interior}^*))$ be the dual-graphs induced by the dual-edges intersecting $E(\mathcal{R}_{\interior})$ and $E(\mathcal{A}_{\interior})$ in $(\mathbb Z^2)^*$ respectively.
	
	For an FK configuration $\omega$ on $E(\mathcal{R}_{\interior})$ we define a dual version of 
	$F(\omega)$ as 
	\begin{equation}\label{def:F:supercritical}
	F^*(\omega) := \mathcal{R}_{\interior}^* \setminus  \bigcup\nolimits_{v^* \in \partial_\ext L^*}\, C^*(v^*,\omega)\,,\notag
	\end{equation}
	where $C^*(v^*,\omega)$ 
	is the dual-vertex set of the connected component of $v^*$ in the dual-configuration $\omega^*$  (ignoring the boundary connections).
	
	Using monotone couplings for $\pi^{\theta_0}$ and $\pi^{\theta_X}$, and for $\pi^{\theta_0}$ and $\pi^{\theta_Y}$, we can define $\mathbb{P}_0$ analogously to $\mathbb{P}_1$ with the configuration on $E(F^*(\omega^{\theta_0}))$ being resampled   
	whenever $\mathcal{A}_\interior^* \subseteq F^*(\omega^{\theta_0})$.
	Observe that updating the dual configuration 
	on $E(F^*(\omega^{\theta_0}))$ 
	is equivalent to updating the primal edges intersecting $E(F^*(\omega^{\theta_0}))$.	
	The fact that resampling the configuration in $E(F^*(\omega^{\theta_0}))$
	 has no effect on the distribution when $\mathcal{A}_\interior^* \subseteq F^*(\omega^{\theta_0})$ 
	follows in similar fashion to the proof of Lemma~\ref{lemma:main:coupling}. 
Indeed, from the definition of $F^*(\omega^{\theta_0})$,
when $A_\interior^* \subseteq F^*(\omega^{\theta_0})$ there is a primal connection between $(c_\star,l)$ and $\partial_\south Q_\west$ in $E(Q_\west)\cup \partial E(Q_\west)$ together with a primal connection between $(d_\star, l)$ and $\partial_\south Q_\east$ in $E(Q_\east)\cup \partial E(Q_\east)$, so that $A_\interior^* \subseteq F^*(\omega^{\theta_0})$ implies the event $\Gamma^*$. Together with the fact that $\llb c_\star,d_\star \rrb$ is a disconnecting interval
	and the assumptions that $\xi$ is free on $\b_{\south}\mathcal{A}_{\interior}$
	and $(c_\star,l),(d_\star,l)\notin \partial_\parallel \mathcal R$, this
	ensures that the three induced boundary conditions on $E(F^*(\omega^{\theta_0}))$ coincide.
	
	Finally, part 3 of the lemma
	follows by an analogous argument to that in Claim~\ref{claim:coupling-proof}, only replacing the EDC property by the matching exponential decay of dual-connectivities when $p>p_c(q)$.
\end{proof}

\subsection{Proof of Theorem~\ref{thm:main-thin-rect}}\label{subsec:complete-proof}

In this section, we put together the results from Sections \ref{subsec:groups-of-rectangles}--\ref{subsec:block-dynamics} to prove Theorem~\ref{thm:main-thin-rect}. We first remind the reader of the following standard inequality concerning the spectral gaps of the FK and block dynamics.

\begin{theorem}[{\cite[Proposition~3.4]{Martinelli-notes}}]\label{thm:block-dynamics}
	Consider the FK-dynamics on a group of rectangles $\mathcal R$ with boundary condition $\zeta$.
	Let $\gap(\mathcal R^\zeta)$ and $\gap({\mathcal R}^\zeta;\mathscr B)$, respectively, be the spectral gaps of the FK-dynamics on $\mathcal R$ and of the block dynamics with blocks $\mathscr B = \{\mathcal B_1,\ldots,\mathcal B_k\}$ such that $\mathcal B_i \subset \mathcal R$ and $\bigcup_{i=1}^k  E(\mathcal B_i)=E(\mathcal R)$. For every $p,q$ there exists $\gamma = \gamma(p,q) \in (0,1)$ such that 
	\begin{equation*}
	\gap(\mathcal R^\zeta) \geq \gamma \cdot (\max_{e\in E(\mathcal R)} \#\{i : E(\mathcal B_i) \ni e\})^{-1}\cdot \gap (\mathcal R^\zeta;\mathscr{B}) \cdot \min_{\substack{i=1,\dots,k \\ \eta \in \Omega(\mathcal{B}_i^c)}} \gap(\mathcal B_i^{(\zeta,\eta)}) \,,
	\end{equation*}
	where $\Omega(\mathcal{B}_i^c)$ denotes the set of FK configurations on $E(\mathcal R) \setminus E(\mathcal B_i)$.
\end{theorem}
\noindent
The proposition in~\cite{Martinelli-notes} is written in the spin system setting, but the proof follows mutatis mutandis for the random-cluster model and its proof is thus omitted. 
Also, we note that this theorem holds in more generality for arbitrary graphs with arbitrary boundary conditions, but for clarity we choose to state it here for groups of rectangles.

The final ingredient is
the following spectral gap bound for the base case in our recursive proof. 
\begin{lemma}\label{lem:base-case}
	Consider a group of rectangles $\mathcal{R}_{0} \subset \L_{n,l}$
	with $W(\mathcal{R}_{0})\leq100m$. For every $q>1$ and $p\neq p_c(q)$, there exists $\kappa = \kappa(p,q)>0$
	such that for every boundary condition
	$\zeta$ on $\mathcal{R}_{0}$, 
	\begin{align*}
	\gap(\mathcal{R}_{0}^\zeta) & \geq \frac{1}{ l (\log l)^{2}\cdot q^{\kappa l}}\,.
	\end{align*}
\end{lemma}
\begin{proof}
	Note that $|\partial\mathcal{R}_{0}| = O(m+l)$. Hence,
	we can first modify the boundary conditions to be all free on all
	of $\partial\mathcal{R}_{0}$, incurring a cost of a $q^{O(l)}$ factor
	in the spectral gap by Lemma~\ref{lem:comparison-tmix}; recall that $m=O(\log l)$. Then we
	can use the fast mixing result of~\cite{BS}, for instance, to bound
	the mixing time on $\mathcal{R}_{0}$ with free boundary condition
	by $O(l (\log l)^{2})$. This translates into a lower bound for the spectral gap and the result follows.
\end{proof}

\begin{proof}[Proof of Theorem~\ref{thm:main-thin-rect}]
	Fix $q> 1$, $p\neq p_c(q)$ and $\L_{n,l}$ with a realizable boundary condition $\xi'$ that is \emph{free} on $\b_\east \L_{n,l}\cup \b_\south \L_{n,l}\cup \b_\west \L_{n,l}$. By Lemma~\ref{lem:comparison-tmix}, we may modify $\xi'$ to a boundary condition $\xi$ that is also free on all vertices a distance at most $m= C_\star \log l$ from $\b_\east \L_{n,l} \cup \b_\west \L_{n,l}$  at a cost of an  exponential in $m$ factor in the mixing time of the FK-dynamics. Let $\xi$ be the resulting realizable boundary condition. 
	
	We wish to prove, by induction, that for every $100m\leq s \leq n$, every group of rectangles $\mathcal R_s \subset \L_{n,l}$ that is compatible with $\xi$ and has $W(\mathcal R_s) = s$ satisfies
	\begin{align}\label{eq:induction-block-dynamics}
	\gap\left(\mathcal{R}_s^{(\xi,\omega_{\mathcal R_{s}^c})}\right) \geq \frac{1}{l (\log l)^2 q^{\kappa l} \cdot b^{\log s}}
	\end{align}
	for some $b = b(p,q)>0$ to be chosen, uniformly over all configurations $\omega_{\mathcal R_{s}^c}$ on $E^c(\mathcal{R}_{s})$. Eq.~\eqref{eq:induction-block-dynamics} concludes the proof since $\L_{n,l}$ is a group of rectangles with $W(\L_{n,l})=n$ and is compatible with~$\xi$. 
	
	The base case of this induction was shown in Lemma~\ref{lem:base-case}. Now suppose inductively that this holds for all $1\leq k \leq s-1$ for some $s\leq n$; we show that it also holds for $s$. Fix any $\mathcal R_s$ that is compatible with $\xi$, and any configuration $\omega_{\mathcal R_{s}^c}$. 
	Then, if we let $\mathcal R_{\interior}=\mathcal R_{\interior}(\mathcal R_s)$ and $\mathcal R_{\ext}= \mathcal R_{\ext}(\mathcal R_s)$ be the blocks given by Definition~\ref{def:subblocks} and $\mathscr B_s$ the block-dynamics with respect to these blocks, by Theorem~\ref{thm:block-dynamics}
	\begin{align*}
	\gap\left(\mathcal{R}_s^{(\xi,\omega_{\mathcal R_{s}^c})} \right) &\geq \frac{\gamma}{2} \cdot \gap \left(\mathcal{R}_s^{(\xi,\omega_{\mathcal R_{s}^c})} ;\mathscr B_s\right) \cdot \min_{i\in\{\interior,\ext\}} \min_{\omega_{\mathcal R_{i}^c}} ~\gap\left(\mathcal{R}_i^{(\xi,\omega_{\mathcal{R}_i^c})}\right) \\
	&\geq \frac{\gamma}{2K} \cdot \min_{i\in\{\interior,\ext\}} \min_{\omega_{\mathcal R_{i}^c}} \gap\left(\mathcal{R}_i^{(\xi,\omega_{\mathcal{R}_i^c})}\right),
	\end{align*}
	where the second inequality follows from Lemma~\ref{lem:block-dynamics-gap}.
	By Proposition~\ref{prop:subblock-properties},  $\max\{W(\mathcal R_\interior), W(\mathcal R_\ext)\} \leq \frac 45 s$, and we can apply the inductive hypothesis to bound the second term on the right-hand side above. Combined with Lemma~\ref{lem:base-case}, we see that the choice of $b= (2\gamma^{-1} K)^{\frac 1{\log (5/4)}}$ ensures that~\eqref{eq:induction-block-dynamics} holds also for $\mathcal R_s$. (Note that $2\gamma^{-1} K \ge 1$.)
	
	This establishes the result for the case when the boundary condition is free on $\b_\east \L_{n,l}\cup \b_\south \L_{n,l}\cup \b_\west \L_{n,l}$.
	As noted earlier (see Remark~\ref{rem:fk-dynamics-duality}), this implies by duality the same bound for the class of realizable boundary conditions $\xi$ that are wired on $\b_\east \L_{n,l}\cup \b_\south \L_{n,l}\cup \b_\west \L_{n,l}$ for all $p\neq p_c(q)$. 
	\end{proof}

\section{Polynomial mixing time for realizable boundary conditions}
\label{section:planar}

In this section we prove Theorem~\ref{thm:planar-mixing:intro}. 
This theorem is proved for $p < p_c(q)$ using the technology introduced in Section \ref{sec:general}; namely,
we construct a collection of subsets $\B$ for which we can establish LM and MSM; see Definitions~\ref{def:msm}--\ref{def:lm}.
To establish LM we crucially use Theorem~\ref{thm:main-thin-rect}. 
The results for $p > p_c(q)$ 
follow from the self-duality of the model and of realizable boundary conditions, as explained in  Section~\ref{subsection:prelim:dynamics}. 

For general realizable boundary conditions, proving LM for a collection of subsets $\B$ for which MSM holds 
is the main challenge. This is because,
for MSM to hold for a collection $\B$
for all realizable boundary conditions,
a subset in $\B$ needs to contain $\Omega(n)$ edges.
In particular, some element of $\B$ must include most (or all) edges near $\b\L_n$,
as otherwise it is straightforward to construct examples of realizable boundary conditions for which MSM
does not hold. 
Thus, a trivial (exponential in the perimeter) upper bound for the mixing time on those subsets with $\Omega(n)$ edges would be unhelpful and we use Theorem~\ref{thm:main-thin-rect}. 

\begin{figure}[t]
	\begin{subfigure}[b]{0.32\textwidth}
		
		\begin{center}
			\begin{tikzpicture}
			\draw[draw=black] (0,0) rectangle (3.5,3.5);
			\draw[draw=black,dotted] (0,0) rectangle (1,1);
			\draw[draw=black,dotted] (0,2.5) rectangle (1,3.5);
			\draw[draw=black,dotted] (2.5,2.5) rectangle (3.5,3.5);
			\draw[draw=black,dotted] (2.5,0) rectangle (3.5,1);
			
			\node[] (ne) at (0.5,0.5) {\tiny{$C_{\south \west}$}};	
			\node[] (ne) at (3,0.5) {\tiny{$C_{\south \east}$}};	
			\node[] (ne) at (0.5,3) {\tiny{$C_{\north \west}$}};	
			\node[] (ne) at (3,3) {\tiny{$C_{\north \east}$}};
			
			\draw[draw=gray!90,<->] 
			(0,1.1) -- (1,1.1);
			
			\draw[draw=gray!90,<->] 
			(1.1,0) -- (1.1,1);
			
			\node at (0.5,1.25) {{\tiny $5r$}};
			\node at (1.25,0.5) {{\tiny $5r$}};
			
			\end{tikzpicture}
			\caption{}
			
		\end{center}
	\end{subfigure}
	\begin{subfigure}[b]{0.32\textwidth}
		
		\begin{center}
			\begin{tikzpicture}
			\draw[draw=black] (0,0) rectangle (3.5,3.5);
			\draw[draw=black,dotted] (0,0.5) rectangle (0.4,3);		
			\draw[draw=black,dotted] (0.5,3.1) rectangle (3,3.5);
			\draw[draw=black,dotted] (3.1,0.5) rectangle (3.5,3);		
			\draw[draw=black,dotted] (0.5,0) rectangle (3,0.4);
			
			\node[] (ne) at (1.75,0.2) {\tiny{$R_\south$}};	
			\node[] (ne) at (1.75,3.3) {\tiny{$R_\north$}};	
			\node[] (ne) at (0.2,1.75) {\tiny{$R_\west$}};	
			\node[] (ne) at (3.3,1.75) {\tiny{$R_\east$}};
			
			\draw[draw=gray!90,<->] 
			(0.5,0.5) -- (0.5,3);
			
			\draw[draw=gray!90,<->] 
			(0,3.1) -- (0.4,3.1);
			
			\node at (0.9,1.75) {{\tiny $n\!-\!6r$}};
			\node at (0.2,3.25) {{\tiny $3r$}};
			\end{tikzpicture}
			\caption{}
			
		\end{center}
	\end{subfigure}
	\begin{subfigure}[b]{0.32\textwidth}
		
		\begin{center}
			\begin{tikzpicture}
			\draw[draw=black] (0,0) rectangle (3.5,3.5);
			\draw[draw=black,dotted] (0,0) rectangle (0.75,0.75);
			\draw[draw=black,dotted] (1,1) rectangle (2.25,2.25);
			
			\draw[color=black,fill=black] (1.525,1.625) circle (.03);
			\draw[color=black,fill=black] (1.725,1.625) circle (.03);
			\draw[color=black] (1.525,1.625) -- (1.725,1.625);
			\put(6.3,5.5){{\tiny $e$}} 
			\put(44.5,41.5){{\tiny $e$}} 
			
			\draw[color=black,fill=black] (0.15,0.15) circle (.03);
			\draw[color=black,fill=black] (0.15,0.35) circle (.03);
			\draw[color=black] (0.15,0.15) -- (0.15,0.35);

			\draw[draw=gray!90,<->] 
			(1,0.85) -- (2.25,0.85);
			
			\put(38,18){{\tiny $2r\!+\!1$}}			
			
			\end{tikzpicture}
			\caption{}
			
		\end{center}
	\end{subfigure}
	
	\caption{{{\rm (a)} The subsets $C_{\north \east}$, $C_{\north\west}$, $C_{\south\east}$, and $C_{\south\west}$.} {{\rm (b)} The subsets $R_\north$, $R_\east$, $R_\west$ and $R_\south$.} {{\rm (c)} $B(e,r)$ for two edges $e$ of $\L_n$.} }     
	\label{figure:planar:blocks}
\end{figure}

We now define the collection of blocks for which we can establish both LM and MSM.
Let $r \in \N$ and
let $C_{\north \east},C_{\north\west},C_{\south\east},C_{\south\west} \subset \L_n$ be the four square boxes of side length $5r$ 
with a corner that coincides with a corner of $\L_n$; 
see Figure \ref{figure:planar:blocks}(a). 
Let $R_\north \subset \L_n$ be the $(n-6r)\times 2r $ rectangle
at distance $3r$ from both $\b_\west\L_n$ and $\b_\east\L_n$ whose top boundary is contained in $\b_\north \L_n$ and let $R_\east,R_\west,R_\south$ be defined analogously; see Figure \ref{figure:planar:blocks}(b). Let $R = R_\north \cup R_\east \cup R_\west \cup R_\south$.
Now, for $e \in E(\L_n)$, let $B(e,r) \subset \L_n$ be the set of vertices in the minimal square box around $e$ such that $d(\{e\},\L_n\setminus B(e,r)) \ge r$. 
Note that if $d(\{e\},\b \L_n) > r$, then $B(e,r)$ is just a square box of side length $2r+1$ centered at $e$; otherwise $B(e,r)$ intersects $\b \L_n$; see Figure \ref{figure:planar:blocks}(c).
Finally, let
\begin{align}\label{eq:planar-B-def:sketch}
\mathcal{B}_r = \{C_{\north\east},C_{\north\west},C_{\south\east},C_{\south\west},R\} \cup \{B(e,r): e \in E(\L_n), d(\{e\},\b\L_n) > r\}.
\end{align}

We claim that LM holds for $\B_r$ with $r = \Theta(\log n)$
and $T = O(n^C)$ for some constant $C > 0$.

\begin{theorem}\label{thm:planar:lm}
	Let $q \ge 1$, $p < p_c(q)$ and
	$r = c_0 \log n$ with $c_0 > 0$ independent of $n$.
	There exists  a constant $C > 0$ such that LM holds for every realizable boundary condition $\xi$ and $\mathcal{B}_r$ with $T = O(n^C)$.
\end{theorem}

\noindent
The subsets $B(e,r)$ in $\mathcal B_r$
and  the corner boxes
$C_{\north \east}$, $C_{\north\west}$, $C_{\south\east}$ and $C_{\south\west}$
are small enough that crude bounds for their mixing times are sufficient.
As mentioned earlier, 
the main challenge for proving local mixing for $\B_r$ is to derive a mixing time bound for $R = R_\north \cup R_\east \cup R_\west \cup R_\south$ as it intersects the boundary of $\L_n$ and contains $\Omega(n)$ vertices. 
To establish such a bound
we rely on Theorem \ref{thm:main-thin-rect}. In particular, we relate the mixing time of the FK-dynamics on $R$ to that of the FK-dynamics on a single thin rectangle by concatenating the four rectangles constituting $R$, one after another, such that the union of their outer boundaries make up the northern boundary of the new rectangle.

The final ingredient of the proof is establishing MSM for the collection 
$ {\mathcal B}_r$.
We show that MSM holds for $ {\mathcal B_r}$ with $r=\Theta(\log n)$ for all realizable boundary conditions $\xi$
where the vertices in $\b\L_n$ at distance $5r$ from the corners of $\L_n$ are free in $\xi$.
This is sufficient since any realizable boundary condition 
can be turned into a realizable boundary condition with this property by simply removing all connections in $\xi$ involving  vertices near the corners of $\L_n$; this modification can change the mixing time of the FK-dynamics by a factor of at most $\exp(O(r))$; see Lemma \ref{lem:comparison-tmix}. Theorem \ref{thm:planar-mixing:intro} then follows 
from Theorems~\ref{thm:planar:lm},~\ref{thm:planar-msm:main} and~\ref{thm:general-mixing}. 

\begin{theorem}\label{thm:planar-msm:main}
	Let $q \ge 1$, $p < p_c(q)$ and
	$r = c_0 \log n$ with $c_0 > 0$ independent of $n$.
	Let $\xi$ be a realizable boundary condition
	with the property that every vertex $v \in \b \L_n$ at distance at most $5r$ from a corner of $\L_n$ is free in $\xi$.
	Then, for all sufficiently large $c_0 > 0$, MSM holds for $\xi$ and  $\mathcal{B}_r$ with $\delta < 1/(12|E(\L_n)|)$.
\end{theorem}

We are now ready to prove Theorem~\ref{thm:planar-mixing:intro} using the above.

\begin{proof}[Proof of Theorem~\ref{thm:planar-mixing:intro}]
	As mentioned earlier, by duality of the dynamics and self-duality of the class of realizable boundary conditions, it suffices to prove the theorem for $p<p_c(q)$. 
	Let $\mathcal{P}$ be the set of realizable boundary conditions of $\L_n = (\L_n,E(\L_n))$.
	For $\eta \in \mathcal{P}$, let 
	$(\eta_1,\eta_2,\dots,\eta_k)$ denote the partition of  $\b\L_n$ corresponding to $\eta$, and let $\eta(\ell)$ be the boundary condition 
	obtained as follows:
	for each $v \in \partial \L_n$, if  $v \in \eta_i$ and $v$ is at distance at most $\ell$ from a corner of $\L_n$, remove $v$ from $\eta_i$ and add it as a singleton to the partition.
	Let $\mathcal{P}_\ell$ be the set of all boundary conditions obtained in this manner. 
	
	Consider $\Lambda_n$ with arbitrary realizable boundary conditions $\xi\in \mathcal{P}$.
	By Lemma~\ref{lem:comparison-tmix}, we see that there exists $C>0$ such that for every $\xi \in \mathcal P$, we have 
	\begin{align*}
	\tmix (\Lambda_n^\xi)\leq Cq^{8C\ell} \cdot n^2 \cdot \tmix (\Lambda_n^{\xi(\ell)})\,.
	\end{align*}
	It therefore suffices to prove the mixing time estimate uniformly over all modified boundary conditions $\eta \in \mathcal P_{\ell}$ for $\ell=5r$  and $r=c_0(\log n)$ with $c_0$ taken to be sufficiently large, as $q^{8C\ell}$ would only be polynomial in $n$. 
	By Theorem~\ref{thm:planar-msm:main}, for $c_0$ large enough, uniformly over all such boundary conditions we have moderate spatial mixing with respect to $\mathcal B_r$ and  $\eta  \in \mathcal P_{\ell}$ with $\delta < 1/(12|E(\L_n)|)$. 
	Theorem \ref{thm:planar:lm} implies that
	LM holds for $\mathcal B_r$ and $\eta  \in \mathcal P_{\ell}$ with 
	$T = O(n^c)$ where $c > 0$ constant. 
	The result then follows from Theorem~\ref{thm:general-mixing}.
\end{proof}

\begin{remark}
	\label{remark:mixing-in-rect}
	We note that Theorem \ref{thm:planar-mixing:intro} also holds for the FK-dynamics on rectangles $\Lambda_{n,l} \subset \Z^2$ with $\ell \le n$, 
	provided these rectangles are not too thin. For example, when $l = \Omega((\log n)^2)$, our proofs would yield that the mixing time of the FK-dynamics is polynomial in $n$. 
\end{remark}

\subsection{Local mixing for realizable boundary conditions}
\label{subsec:planar:lm}

In this subsection, we prove Theorem~\ref{thm:planar:lm}. As mentioned earlier, this theorem may be viewed as a corollary of Theorem~\ref{thm:main-thin-rect}, which bounds the mixing time of the FK-dynamics on thin rectangles.

\begin{proof}[Proof of Theorem~\ref{thm:planar:lm}]
	Let $r=c_0 \log n$. We wish to show that each of the subsets in $\mathcal B_r$ has mixing time $O(n^c)$ under the boundary conditions $(1,\xi)$ and $(0,\xi)$. 
	We begin by bounding the mixing time on the square boxes $C_{\north \east}$, $C_{\north\west}$, $C_{\south\east}$, $C_{\south\west}$ and $B(e,r)$ of $\mathcal B_r$.
	Since these have side length $O(\log n)$, a crude estimate on the mixing time is sufficient. For instance,  by Lemma~\ref{lem:comparison-tmix}, at a cost of $\exp(O(r))= n^{O(1)}$ factor, we can compare the mixing time in these boxes with boundary condition 
	either $(1,\xi)$ or $(0,\xi)$
	to the mixing time on equally sized boxes with free boundary conditions. In this setting, an upper bound of $O((\log n)^2 \log \log n)$ is known~\cite{BS}, and thus we obtain an $n^{O(1)}$ bound for their mixing times.
	
	It remains to bound the mixing time of the FK-dynamics on the set $R= R_{\north} \cup R_\east \cup R_\west \cup R_\south$. For this, we use Theorem~\ref{thm:main-thin-rect}. 
	We argue that the mixing time of the FK-dynamics on $R$
	is roughly equal to that of the FK-dynamics on a $[4(n-6r)-3]\times 2r$ rectangle $Q$ with a suitably chosen boundary condition. 
	We proceed to construct the rectangle $Q$ and a boundary condition $\xi'$ whose vertices, edges and wirings are identified with those of $R$ and $(1,\xi)$.
	The case of $R$ and $(0,\xi)$ is handled later in similar fashion.
	
	We introduce some notation first. 
	For a rectangle $S$, let $S^\alpha$ denote the rectangle that results from a clockwise rotation of $S$ by an angle of amplitude $\alpha$. Also, if $S_1,\dots,S_k$ are rectangles of the same height, let
	$[S_1,\dots,S_k]$ denote the rectangle obtained by identifying the vertices in $\b_\east S_i$ with those in $\b_\west S_{i+1}$ for all $i=1,\dots,k-1$. 
	When identifying the vertices, the double edges are removed.
	We take
	$$
	Q = \left[R_{\west}^{\pi/2},R_{\north},R_{\east}^{-\pi/2},R_{\south}^{-\pi}\right].
	$$
	Observe that every vertex of $Q$, except those where the boundary overlaps occur, correspond to exactly one vertex in $R$; vertices in the overlaps correspond to exactly two vertices in $R$.
	Conversely, every vertex in $R$ corresponds to exactly one vertex of $Q$. 
	The edges of $Q$ and $R$ are identified using this correspondence between the vertices.
	Observe also that $\b_\north Q = \b R \cap \b \L_n$. We construct the boundary condition $\xi'$
	of $Q$ as follows. If $u,v \in \b R \cap \b \L_n$ are wired in $\xi$, the corresponding vertices are also wired in $\xi'$. The boundary condition $\xi'$ is also wired along $\b_\west Q \cup \b_\north Q \cup \b_\east Q$.
	
	\begin{claim}
		\label{claim:thin-rectangles-bridges}
		The boundary condition $\xi'$ of $Q$ is realizable. 
		In particular, $\xi'$ can be realized by an FK configuration in the half plane of $\Z^2$
		containing only vertices north of $\b_\north Q$, and a wiring of $\partial _\east Q\cup \partial_\south Q \cup \partial_\west Q$.
	\end{claim}
	
	Finally, to completely capture the effect of the boundary condition $(1,\xi)$ on $R$, 
	each of the three columns in $Q$ that corresponds to overlaps of columns from $R$, are externally wired.
	
	Now, by Theorem \ref{thm:main-thin-rect} and Lemma \ref{lem:comparison-tmix},
	we have
	$$
	\gap(Q^{\xi'}) = n^{-O(1)}\,.
	$$ 
	We claim next that the FK-dynamics on $R$ with boundary condition $(1,\xi)$
	has roughly the same gap as the FK-dynamics on $Q$ with boundary condition $\xi'$.
	To see this, 
	we add a double edge to each edge of $Q$ that corresponds to two edges in $R$.
	With this modification, there is now a one-to-one correspondence between the FK configurations in $R$ and $Q$. 
	Also, adding these edges has almost no effect on the mixing time of the FK-dynamics in $Q$, as their endpoints are wired, and so they only need to be updated once to mix.
	Moreover, by construction, the boundary condition $\xi'$ for $Q$ together with the wiring of the overlapping columns in $Q$ encode exactly the same connectivities as the boundary condition $(1,\xi)$ for $R$.
	Hence, for every pair of FK configurations on $Q$,
	the FK-dynamics has the same transition probability as FK-dynamics on $R$ 
	between the corresponding configurations. Consequently,  we can conclude 
	that
	$$
	\gap(R^{(1,\xi)}) = n^{-O(1)}\,.
	$$
	
	Finally, for the case of the FK-dynamics on $R$ with boundary condition $(0,\xi)$ we can simply wire
	$\b_\west R_\north$ to $\b_\north R_\west$,
	$\b_\south R_\west$ to $\b_\west R_\south$,
	$\b_\east R_\south$ to $\b_\south R_\east$ and
	$\b_\north R_\east$ to $\b_\east R_\north$, which only incur a penalty of $n^{O(1)}$
	by Lemma~\ref{lem:comparison-tmix} and proceed as in the previous case.
\end{proof}

\begin{proof}[Proof of Claim \ref{claim:thin-rectangles-bridges}]
	First note that $\b_\north Q$ corresponds to $\b R \cap \b \L_n$. Let $\omega$ be an FK configuration on $\Z^2 \setminus \L_n$ that 
	realizes $\xi$. A path from $u \in \b R \cap \b \L_n$ to $v \in \b R \cap \b \L_n$ in $\omega$ splits $\b R \cap \b \L_n$ into two parts $R_1$ and $R_2$, one containing all the vertices 
	from $u$ to $v$ in $\b R \cap \b \L_n$ clockwise and the other all the vertices from $u$ to $v$ in $\b R \cap \b \L_n$ counterclockwise.  
	The planarity of $\Z^2$ implies that any other boundary component  will be either completely contained in $R_1$ or $R_2$.
	From this property, it follows that if $v_1,v_2 \in \b_\north Q$ are wired in $\xi'$, then $\llb v_1,v_2\rrb$ is a disconnecting interval.
	This implies that the connectivities of $\xi'$ in $\b_\north Q$ can be realized by a configuration on the half plane of $\Z^2$ that contains all the vertices north of $\b_\north Q$.
	For example, every component $C = \{c_0,\dots,c_k\}$ of $\xi$, with $c_i$ to the left of $c_{i+1}$,  
	can be realized by 
	the gadget consisting $k$ paths 
	of length $h_C$ starting at $c_0,\dots,c_k$ and going north, together with one path 
	parallel to $\b_\north Q$ that joins the endpoints of all of these path. 
	Since $\llb c_i,c_{i+1}\rrb$ is a disconnecting interval for all $i$ and $C$,
	we can choose $h_C$ for each $C$ so that the resulting configuration is a valid configuration in the half plane.
\end{proof}

\subsection{Moderate spatial mixing for realizable boundary conditions}
\label{subsec:planar:msm}

In this section we prove Theorem \ref{thm:planar-msm:main}.
We reduce the moderate spatial mixing condition~\eqref{eq:ssm:def} to bounding the 
probability of certain connectivities in an FK configuration.
Specifically, if $e \in S \subset \L_{n}$, the configuration on $E^c(S)$
affects the state of $e$ when there are paths
from $e$
to the boundary of $S$; the probability of such paths is maximized when we assume an all wired configuration on $E^c(S)$. Recall that for $S \subset \L_{n}$, we let $S^c = \L_{n} \setminus S$, and we use  $E^c(S) = E(\L_{n})\setminus E(S)$.

\begin{lemma}\label{lemma:ssm:reduction}
	Consider the FK model on $\L_{n}$
	with arbitrary boundary condition $\xi$ on $\partial \L_{n}$. 
	For any $e \in E(\L_{n})$, any $S \subset \L_{n}$ such that $e \in E(S)$, and any pair 
	of configurations $\omega_1$, $\omega_2$ on $E^c(S)$:
	\begin{equation}\label{eq:bc:reduction}
	\left|\pi^\xi_{\L_{n}}(\,e=1\mid E^c(S)=\omega_1\,)-\pi^\xi_{\L_{n}}(\,e=1\mid E^c(S)=\omega_2\,)\right| ~\le~ \pi^\xi_{\L_{n}} \left(\, \{e\} \stackrel{\xi}\longleftrightarrow \b S\setminus \b\L_n \mid E^c(S)=1\, \,\right), \notag
	\end{equation}
	where  $\{e\} \stackrel{\xi}\longleftrightarrow  \b S$ denotes the event that there is a path from $e$ to $\b S$ 
	taking into account the connections induced by $\xi$. 
\end{lemma}

\noindent
In the proof of Theorem~\ref{thm:planar-msm:main} we use this lemma; its proof via machinery from~\cite{Alexander} will be straightforward.  

\begin{proof}[Proof of Theorem \ref{thm:planar-msm:main}]
	We need to show that as long as $c_0$ is large enough, for every $e \in E(\L_{n})$
	there exists $B_e \in \mathcal{B}_r$ 
	such that \eqref{eq:ssm:def} holds for some $\delta < 1/(12|E(\L_{n})|)$.
	For each $e \in E(\L_{n})$ the subset $B_e$ is chosen as follows:
	\begin{enumerate}
		\item If $d(\{e\},\b\L_n) > r$, then $B_e = B(e,r)$;
		\item Otherwise, if $e \in R$ and
		$
		d(\{e\},\partial R\setminus \partial \Lambda) \ge r,
		$
		then $B_e = R$;		
		\item Otherwise, $e \in C_i$ for some $i \in \{\north\east,\north\west,\south\west,\south\east\}$, and we take $B_e = C_i$.
	\end{enumerate}
	By Lemma \ref{lemma:ssm:reduction}, for every $e\in E(\L_{n})$,
	$$
	\left|\pi_{\L_{n}}^\xi(\,e=1\mid E^c(B_e)=1\,)-\pi^\xi_{\L_{n}}(\,e=1\mid E^c(B_e)=1\,)\right| ~\le~ \pi^\xi_{\L_{n}} \left(\, e \stackrel{\xi}\longleftrightarrow \b B_e \setminus \partial \Lambda_{n} \mid E^c(B_e)=1\, \,\right).
	$$
	In all three cases above, by construction $d(\{e\},\b B_e\setminus \partial \L_{n}) \ge r$.
	This together with
	the fact that all vertices of $\b\L_{n}$
	within distance $5r$ from the corners have no connections in $\xi$, implies that 
	for $e$ to be connected to $\b B_e\setminus \partial \L_{n}$ a path of open edges 
	reaching a distance at least $r$ is required in the configuration on $B_e$. 
	The EDC property (see \eqref{eq:EDC}) implies that for $c_0$ large enough
	\begin{equation*}
	\label{eq:ssm:planar}
	\pi^\xi_{\L_{n}} \left(\, \{e\} \stackrel{\eta}\longleftrightarrow \b B_e \mid E^c(B_e)=1\, \,\right) \le \frac{1}{12|E(\L_{n})|}\,,
	\end{equation*}
	and the result follows.
\end{proof}

\noindent
We conclude this section with the proof of Lemma \ref{lemma:ssm:reduction}. 

\begin{proof}[Proof of Lemma \ref{lemma:ssm:reduction}]
	Let $(1,\xi)$ be the boundary condition induced on $S$ by $\xi$ and the event $\{E^c(S)=1\}$.
	Similarly, let $\theta_1$ (resp., $\theta_2$) be the boundary condition induced on $S$ by configurations $\omega_1$ (resp., $\omega_2$) on $E^c(S)$ and $\xi$. For ease of notation set 
	$\pi^{\theta_1} = \pi^\xi_{\L_{n}}(\cdot\mid E^c(S)=\omega_1)$, 
	$\pi^{\theta_2} = \pi^\xi_{\L_{n}}(\cdot\mid E^c(S)= \omega_2)$ and 
	$\pi^{(1,\xi)} = \pi^{\xi}_{\L_{n}}(\cdot\,\mid E^c(S)=1)$. 
	For an FK configuration $\omega$ on $S$ let
	\begin{equation}\label{def:Gamma}
	\Gamma^{(1,\xi)}(\omega) := S \setminus  \bigcup\nolimits_{v \in \b S\setminus \b \L_{n}}\, C(v,\omega)\,,\notag
	\end{equation}
	where $C(v,\omega)$ is the set of vertices in the connected component of $v$ in $\omega$, taking into account the 
	connectivities induced by $(1,\xi)$.	
	In words, $\Gamma^{(1,\xi)}(\omega)$ is the set of vertices of $S$ not connected to $\b S\setminus \b\L_{n}$ in $\omega$ using possibly the boundary connections. 
	
	We claim that there exists a coupling $\P$ of the distributions $\pi^{\theta_1}$, $\pi^{\theta_2}$ and $\pi^{(1,\xi)}$ such that $\P(\omega_1,\omega_2,\omega) > 0$ only if $\omega_1 \leq \omega$ and  $\omega_2  \leq \omega$ on $E(S)$ and $\omega_1$, $\omega_2$ agree on all edges with both endpoints in $\Gamma^{(1,\xi)}(\omega)$.
	Given this coupling $\P$, we have
	\begin{align*}
	\label{eq:ssm:coupling}
	|\pi^\xi_{\L_{n}}(\,e=1\mid E^c(S)=\omega_1\,)-\pi^\xi_{\L_{n}}(\,e=1 \mid E^c(S)=\omega_2\,)| 
	& \le \P(\,\omega_1(e) \neq \omega_2(e)\,) \notag \\
	& \le \P\left(\,e\notin E(\Gamma^{(1,\xi)}(\omega) )  \,\right)\notag\\
	&= \pi^\xi_{\L_{n}} \left(\,e \stackrel{\xi} \longleftrightarrow \b S \setminus \b \L_{n} \mid E^c(S)=1\,\right), 
	\end{align*}
	as claimed.	
	The construction of the coupling $\P$ is standard and is thus ommitted; see, e.g., \cite{Alexander,BS} and the proof Lemma \ref{lemma:main:coupling} for similar constructions.
\end{proof}

\section{Near optimal mixing for typical boundary conditions}
\label{section:typical}
In this section we provide the proof of 
Theorem \ref{thm:typical-mixing:intro}, where we establish
a sharper $\tilde O(n^2)$ mixing time upper bound for the FK-dynamics on 
$\L_n = (\L_n,E(\L_n))$ for the class of boundary conditions we call
\emph{typical}.

\begin{definition}
	Let $\omega$ be a random-cluster configuration on $\Z^2$, and let
	 $\xi_\omega$ be the boundary condition on $\b\L_n$
	induced by the edges of~ $\omega$ in $E(\Z^2) \setminus E(\L_n)$.
	Suppose $\omega$ is sampled from $\pi_{\Z^2,p,q}$. 
	A set $\mathcal{C}$ of realizable boundary conditions for $\L_n$
	is called \emph{typical} (with respect to $(p,q)$) if $\xi_{\omega} \in \mathcal{C}$ with probability $1-o(1)$.
\end{definition}

Recall from Definition~\ref{def:a-localized} the classes of boundary conditions $\mathcal C_\alpha$ and $\mathcal C_\alpha^\star$, consisting of realizable boundary conditions whose distinct boundary components consist only of vertices at most distance $\alpha \log n$ apart in $\partial \L_n$. 
The following is a straightforward consequence of the EDC property \eqref{eq:EDC} when $p<p_c(q)$.

\begin{lemma}\label{lem:localized-typical}
	For every $q\geq 1$ and $p< p_c(q)$, the class of boundary conditions $\mathcal{C}_\alpha$ is typical with respect to $(p,q)$ for sufficiently large $\alpha > 0$. Similarly, for every $q\geq 1$ and $p>p_c(q)$, the class $\mathcal C_\alpha^\star$ is typical with respect to $(p,q)$ for sufficiently large $\alpha>0$. 
\end{lemma}

\begin{proof}
By planar duality (namely the duality of the sets of boundary conditions $\mathcal C_\alpha$ and $\mathcal C_\alpha^\star$, it suffices to prove the case $p<p_c(q)$). 
	For any $u,v \in \Z^2$, by the EDC property~\eqref{eq:EDC}, we have that for $q \ge 1$ and $p < p_c(q)$ there exists $c = c(p,q) > 0$ such that
	$
	\pi_{\Z^2}(u \leftrightarrow v) \le {\e}^{-c d(u,v)}\,.
	$
	Let $u, v \in \b\L_n$ and suppose $d(u,v) \ge \alpha \log n$. 
	Then, there exists some $C(p,q)>0$ such that for sufficiently large $\alpha > 0$,
	\begin{align*}
	\pi_{\Z^2\setminus \L_n}(u\stackrel{\Z^2\setminus\L_n}\longleftrightarrow v) = \pi_{\Z^2}(u\stackrel{\Z^2\setminus\L_n}\longleftrightarrow v) \le \pi_{\Z^2}(u \leftrightarrow v) \le C{\e}^{-c\alpha \log n} \le \frac{1}{n^3}\,,
	\end{align*}
	where recall that $u\stackrel{\Z^2\setminus\L_n}\longleftrightarrow v$
	denotes the event that there exists a path from $u$ to $v$ in $\Z^2\setminus\L_n$. 
	A union bound over all pairs of vertices in $\b\L_n$ implies that
	if $\omega$ is sampled from $\pi_{\Z^2}$
	and $\xi_{\omega}$ is the resulting boundary condition on $\b\L_n$,
	then $\xi_\omega \in \mathcal{C}_\alpha$ with probability $1-o(1)$ and thus $\mathcal{C}_\alpha$ is typical.
\end{proof}

\begin{remark}\label{rem:finite-typicality}
As mentioned in the introduction, one may also be interested in the following notion of typicality, which sometimes comes up in recursive mixing time upper bounds. Let $q\geq 1$ and $p<p_c(q)$ (resp., $p>p_c(q)$) and consider a random-cluster sample from $\pi_{R_{2n},p,q}^{\zeta}$, where $R_{2n}$ is the concentric box of side length $2n$ containing $\L_n$ with arbitrary boundary condition $\zeta$. One could easily show that the boundary condition induced by the configuration on $R_{2n}\setminus \L_{n}$ is in $\mathcal C_\alpha$ (resp., $\mathcal C_\alpha^\star$) with probability $1-o(1)$. This follows by coupling this measure to the infinite-volume measure using the fact that $\mathcal C_\alpha$ is a decreasing event, and finding a dual circuit in the annulus $R_{2n}\setminus \L_{n}$ (which exists with probability $1-O(e^{-\Omega(n)})$).  
\end{remark}

\noindent
We now show that when $p<p_c(q)$, the mixing time on $\L_n$
with boundary condition $\xi \in \mathcal{C}_\alpha$ satisfies
$$
\tmix(\L_n^\xi) = O\left(n^2 (\log n)^C\right)\,,
$$
where $C=C(p,q,\alpha) > 0$ is a constant independent of $n$ and $\xi$. In particular, we prove Theorem \ref{thm:typical-mixing:intro} from the introduction in the regime $p<p_c(q)$ and $\xi\in \mathcal C_\alpha$ and this translates to a matching bound at $p>p_c(q)$ and $\xi \in \mathcal C_\alpha^\star$ by duality.
To prove this theorem we again use the general framework from Theorem~\ref{thm:general-mixing}. Namely, we construct a collection of subsets of $\L_n$ for which we can establish MSM and LM; see Definitions~\ref{def:msm} and~\ref{def:lm}. The fact that $\xi\in \mathcal C_\alpha$ will allow us to prove MSM with respect to $\Theta((\log n)^2)\times \Theta((\log n)^2)$ rectangles along the boundary, and Theorem~\ref{thm:planar-mixing:intro} will provide the LM estimate on these rectangles.

Consider the collection $\B_r = \{B(e,r): e \in E(\L_n)\}$.
Recall that for $r \geq 0$ and $e \in E(\L_n)$, we set $B(e,r) \subset \L_n$ to be the set of vertices in the minimal square box around $e$ such that $d(\{e\},\L_n\setminus B(e,r)) \ge r$; see Figure \ref{figure:planar:blocks}(c).
We first show that MSM holds for $\B_r$ and $\xi \in \mathcal{C}_\alpha$
when $r = \Theta((\log n)^2)$ and $\delta < n^{-3}$.

\begin{lemma}
	\label{lemma:ssm:typical} 
	Let $q \ge 1$, $p < p_c(q)$, $\alpha > 0$, $\eta \in \mathcal{C}_\alpha$, $r = c_0(\log n)^2$
	and $\mathcal{B} = \{B(e,r): e \in E(\L_n)\}$.
	For large enough $c_0>0$, 
	MSM holds for $\eta$, $\mathcal{B}$ for some $\delta < n^{-3}$.
\end{lemma}

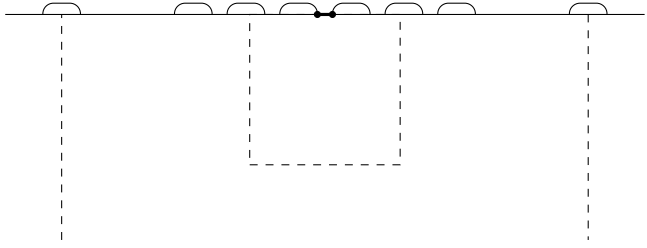
\begin{figure}[t]
	\begin{center}
		\begin{tikzpicture}
		\draw[draw=black,dashed] (0,0)--(0,3);
		\draw[draw=black,dashed] (7,3)--(7,0);
		\draw[draw=black] (-0.75,3)--(7.75,3);
		
		\draw[draw=black,dashed] (2.5,1) rectangle (4.5,3);
		
		\draw[color=black,very thick] (3.4,3) -- (3.6,3);
		\draw[color=black,fill=black] (3.4,3) circle (.04);
		\draw[color=black,fill=black] (3.6,3) circle (.04);
		
		\draw [rounded corners] 
		(3.4,3)--(3.4,3.15)--(2.9,3.15)--(2.9,3);
		
		\draw [rounded corners] 
		(2.7,3)--(2.7,3.15)--(2.2,3.15)--(2.2,3);
		
		\draw [rounded corners] 
		(2,3)--(2,3.15)--(1.5,3.15)--(1.5,3);
		
		\draw [rounded corners] 
		(0.25,3)--(0.25,3.15)--(-0.25,3.15)--(-0.25,3);
		
		\draw [rounded corners] 
		(3.6,3)--(3.6,3.15)--(4.1,3.15)--(4.1,3);
		
		\draw [rounded corners] 
		(4.3,3)--(4.3,3.15)--(4.8,3.15)--(4.8,3);
		
		\draw [rounded corners] 
		(5,3)--(5,3.15)--(5.5,3.15)--(5.5,3);
		
		\draw [rounded corners] 
		(6.75,3)--(6.75,3.15)--(7.25,3.15)--(7.25,3);

		\put(98,79.5){\tiny $e$} ;
		\put(-28,86){\tiny $\b\L$};
		\put(130.5,24){{\tiny $B(e,r')$}} 
		\put(202,0){{\tiny $B(e,r)$}} 
		\put(21,86){{\tiny $\cdots$}}
		\put(170,86){{\tiny $\cdots$}}
		
		\end{tikzpicture}
		\caption{If $r = \Theta ((\log n)^2)$ and $r'=\Theta (\log n)$, influence from outside of $B(e,r')$ may be easily propagated to $e$ through long boundary connections in $B(e,r')$; but to propagate influence from the exterior of $B(e,r)$, $\Omega(\log n)$ of them would have to be connected in $\L_n$.}     
		\label{figure:B(e,r)}
	\end{center}
	
\end{figure}

\noindent
For this lemma, it is crucial that $r = \Theta((\log n)^2)$, as MSM does not hold for typical boundary conditions 
for $\B_r$ when, for example, $r = \Theta(\log n)$.
This is because in a typical configuration $\omega$ on $\mathbb Z^2 \setminus \L_n$ it is likely that there exist pairs of vertices 
of $\b\L_n$ at distance $\gamma \log n$, for a suitably small constant $\gamma > 0$, that are connected in $\omega$. Thus, for some $e \in E(\L_n)$ close to $\b\L_n$, it is possible for the configuration outside of $B(e,r)$
to exert a strong influence on the state of $e$ when $r = \gamma' \log n$ with constant $\gamma' > 0$, even if $\gamma' \gg \gamma$; the presence of a constant number of open edges on (or near) $\b\L_n$ would propagate the influence from $\L_n\setminus B(e,r)$ to $e$.
Taking $r = \Omega((\log n)^2)$ avoids this issue, since, roughly speaking, $\Omega(\log n)$ open edges at specific points in $\b\L_n$ would now be required to propagate the influence from $\L_n\setminus B(e,r)$ to $e$;
see Figure \ref{figure:B(e,r)}(b). The proof of Lemma~\ref{lemma:ssm:typical} is provided in Section \ref{app-sec:msm-typical}.

The final ingredient in the proof of Theorem \ref{thm:typical-mixing:intro} is a LM estimate for $\B_r$  with $r = \Theta((\log n)^2)$. 
Such an estimate is readily provided by Theorem~\ref{thm:planar-mixing:intro}, with mixing time that is poly-logarithmic in $n$.

\begin{proof}[Proof of Theorem~\ref{thm:typical-mixing:intro}]
	Let $\alpha > 0$ be sufficiently large and let $\eta \in C_\alpha$. 
	By Lemma~\ref{lemma:ssm:typical},
	MSM holds for $\eta$ and $\B_r$ with $r = \Theta((\log n)^2)$ for some $\delta < n^{-3}$.
	Observe also that every $B(e,r) \in \mathcal B_r$ with boundary condition $(1,\eta)$ or $(0,\eta)$ 
	is a rectangle of side--length at most $r = O((\log n)^2)$ with a realizable boundary condition.
	Then by Theorem~\ref{thm:planar-mixing:intro} (see also Remark \ref{remark:mixing-in-rect}), for every $e\in E(\L_n)$ we have  
	\begin{align*}
	\max\{\tmix(B(e,r)^{0,\eta}),\tmix(B(e,r)^{1,\eta})\} \leq  (\log n)^{C}\,,
	\end{align*}
	for a suitable $C > 0$, yielding the desired LM estimate.
	The result then follows from Theorem~\ref{thm:general-mixing}.
\end{proof}

\subsection{Moderate spatial mixing for $\mathcal{C}_\alpha$}\label{app-sec:msm-typical}

We now prove Lemma \ref{lemma:ssm:typical}. The proof involves showing that if $\xi \in \mathcal C_\alpha$, when $p<p_c(q)$, the correlation between edges $e,e'\in E(\L_n)$ near the boundary decays exponentially in $d(e,e')/(\alpha\log n)$---whereas SSM would entail a decay rate that is exponential in just~$d(e,e')$.

\begin{proof}[Proof of Lemma \ref{lemma:ssm:typical}]
	Fix an edge $e\in E(\L_n)$ and for ease of notation let $B = B(e,r)\subset \L_n$ and $\pi^\eta = \pi^\eta_{\L_n,p,q}$.
	Let $(1,\eta)$ be the boundary condition induced on $B$ by $\eta$ and the event $\{E^c(B)=1\}$.
	
	Lemma \ref{lemma:ssm:reduction} implies that for every pair of configurations $\omega_1$, $\omega_2$ on $E^c(B)$, 
	\begin{equation}
		\label{eq:ssm:reduction}
	|\pi^\eta(\,e=1\mid E^c(B)=\omega_1\,)-\pi^\eta(\,e=1 \mid E^c(B)=\omega_2\,)| 
	 \le \pi^{(1,\eta)} \left(\, \{e\} \stackrel{\eta}\longleftrightarrow \b B \setminus \b\L_n \,\right), 
	\end{equation}
	where  $\{e\} \stackrel{\eta}\longleftrightarrow  \b B \setminus \b\L_n$ denotes the event that there is a path from $e$ to $\b B\setminus \b\L_n$ 
	taking into account the connections induced by $\eta$. 
	Thus, it is sufficient to bound the right-hand-side of \eqref{eq:ssm:reduction}.
	
	There are three cases corresponding to the location of $e$ in $\L_n$.
	First, if $d(\{e\},\b \L_n) > r$, then $B \cap \b \L_n = \emptyset$ and $(1,\eta)$ is just the wired boundary condition on $B$. In this case the right-hand-side of~\eqref{eq:ssm:reduction} is at most $n^{-3}$ by the EDC property; see \eqref{eq:EDC}.
	
	The second and third cases correspond to whether $B$ intersects one or two sides of $\partial \L_n$.
	For the second case, assume without loss of generality that 
	$B$ intersects $\bt\L_n$ but not $\bl\L_n$ or $\br\L_n$.
	That is,
	$d(\{e\},\b_\north \L_n) \le r$, but $\{e\}$ is at distance at least $r$ from
	$\bl\L_n$ and $\br\L_n$. 
	Let $\b_\west B$, $\b_\south B$, $\b_\east B$ be the west, south and east boundaries of $B$, all of  which are wired in $\omega$. By a union bound
	$$
	\pi^{(1,\eta)} \left( \{e\} \stackrel{\eta}\longleftrightarrow \b B \setminus \b \L_n   \right) \le   \pi^{(1,\eta)} \left( \{e\} \stackrel{\eta}{\longleftrightarrow}   \b_\west B  \right) +  \pi^{(1,\eta)} \left(\{e\} \stackrel{\eta}{\longleftrightarrow}  \b_\east B \right) +  \pi^{(1,\eta)} \left( \{e\} \stackrel{\eta}{\longleftrightarrow} \b_\south B \right) ,
	$$
	where  $\{e\} \stackrel{\eta}{\longleftrightarrow}  \b_\west B$ denotes the event that there is a path from $e$ to $\b_\west B$ in $B$, taking into account those connections inherited from $\eta$ (and ignoring the connections induced by the wired configuration on $E^c(B)$).
	Define $\{e\}\stackrel{\eta}{\longleftrightarrow}  \b_\east B$ and $\{e\} \stackrel{\eta}{\longleftrightarrow}  \b_\south B$ similarly.
	
	The event $\{e\} \stackrel{\eta}{\longleftrightarrow} \b_\south B$ implies that there exists a path of length  at least $r$, either from $\{e\}$ or from $\b\L_n$ to  $\b_\south B$. Therefore, the EDC property (\eqref{eq:EDC}) implies that for large enough $n$,
	$$
	\pi^{(1,\eta)} \left(\, \{e\} \stackrel{\eta}{\longleftrightarrow} \b_\south B  \,\right) \le \frac{1}{3n^3}\,.
	$$
	
	We bound next $\pi^{(1,\eta)} ( \{e\} \stackrel{\eta}{\longleftrightarrow} \b_\west B)$. 
	Let $\eta_0,\dots,\eta_d$ be the boundary components of $\eta$.
	Since $\eta$ is realizable, the planarity of $\Z^2$ implies that for every $i,j \in \{0,\dots,d\}$ there are only three possibilities: $L(\eta_i) \cap L(\eta_j) = \emptyset$, $L(\eta_i) \subset L(\eta_j)$ or $L(\eta_j) \subset L(\eta_i)$. 
	(Recall that $L(\eta_i) \subset \b\L_n$ is the path of minimum length that contains all the vertices in $\eta_i$.)
	Call $\eta_i$ a maximal boundary component if $\nexists j \in \{0,\dots,d\}$ such that $L(\eta_i) \subset L(\eta_j)$. The set of all maximal boundary components defines a partition for $\b\L_n$.
	Since also $\eta \in \mathcal{C}_\alpha$, we deduce that 
	there exists a sequence of edges
	$e_0=\{u_0,v_0\},e_1,\dots,e_k=\{u_k,v_k\}$ in $B \cap \b \L_n$ 
	such
	that 1) $(\gamma+\alpha) \log n \ge d(\{e_i\},\{e_{i+1}\}) \ge \gamma \log n$ for all $i=0,\dots,k-1$, where $\gamma$ is a large constant we choose later and $k \ge \frac {c_0}{2(\gamma+\alpha)} \log n$, and 2) the set $S_i = \llb v_{i},u_{i+1} \rrb \subset B \cap \b\L_n$ is a disconnecting interval.
	
	Let 
	$\mathcal{E}_i$ be the event that $S_i$ is connected to $S_{i+1}$ by a path of open edges in $B$. 
	Let $e_t$ be the closest edge in the sequence $e_0,e_1,\dots,e_k$  to $e$ and let
	$\mathcal{\hat{E}}_t= \cap_{i=0}^t \mathcal{E}_i$. Since $d(\{e\},\b_\west B)\geq r$, we also have $t\geq \frac{c_0}{8(\gamma+\alpha)} \log n$. Then,
	\begin{equation}
	\label{eq:ssm:main-conditioning}
	\pi^{(1,\eta)} \left(\, \{e\} \stackrel{\eta}{\longleftrightarrow}  \b_\west B\,\right) \le \pi^{(1,\eta)} \left(\, \{e\} \stackrel{\eta}{\longleftrightarrow}  \b_\west B \mid \mathcal{\hat{E}}^c_t \,\right) + \pi^{(1,\eta)} \left(\,\mathcal{\hat{E}}_t \,\right)\,. 
	\end{equation}
	If the event $\mathcal{\hat{E}}^c_t$ occurs, then there exists $i < t$ such that $S_i$ is not connected to $S_{i+1}$
	in $B$. This implies that there is a dual-path of length at least $\gamma \log n$ separating $S_i$ from $S_{i+1}$. Consequently, a path from $e$ to $\b_\west B$ would require
	two vertices at distance at least  $\gamma \log n$ to be connected by a path of open edges in $B$.
	By the EDC property and a union bound,  this has probability at most $1/(9n^3)$ for large enough $\gamma$. Thus,
	\begin{equation}
	\label{eq:ssm:no-rec-event}
	\pi^{(1,\eta)} \left(\, \{e\} \stackrel{\eta}{\longleftrightarrow}  \b_\west B \mid \mathcal{\hat{E}}^c_t \,\right) \le \frac{1}{9n^3}\,.
	\end{equation}
	
	We bound next $\pi^{(1,\eta)} (\mathcal{\hat{E}}_t )$. Let $ \{u_i,v_i\}$ denote the endpoints of the edge $e_i$, where $u_i$ is to the left of $v_i$ for all $i$. For $1 \le i \le t$, consider the rectangle $Q_i \subset B$ with corners at $v_{i-1}$, $u_{i+3}$ and the other two corners on $\b_\south B$. 
	Then,
	\begin{align}
	\label{eq:ssm:removing-events}
	\pi^{(1,\eta)} (\mathcal{\hat{E}}_t ) 
	&= \pi^{(1,\eta)} (\mathcal{E}_0,\dots ,\mathcal{E}_t) 
	\le \pi^{(1,\eta)} (\mathcal{E}_1,\mathcal{E}_5,\mathcal{E}_9\dots ,\mathcal{E}_\ell), 
	\end{align}
	where $t - 4 < \ell \le t$. Now,  let $\mathcal{E}'_i$ be the event that $S_i$ is connected to $S_{i+1}$ by a path completely contained in $Q_i$. We have
	\begin{align*}
	\pi^{(1,\eta)} (\mathcal{E}_1,\mathcal{E}_5,\mathcal{E}_9\dots ,\mathcal{E}_\ell) 
	&= \pi^{(1,\eta)} (\mathcal{E}'_1,\mathcal{E}_5,\mathcal{E}_9\dots ,\mathcal{E}_\ell) + \pi^{(1,\eta)} (\mathcal{E}_1\cap (\mathcal{E}'_1)^c,\mathcal{E}_5,\mathcal{E}_9\dots ,\mathcal{E}_\ell) \\
	&\le  \pi^{(1,\eta)} (\mathcal{E}'_1,\mathcal{E}_5,\mathcal{E}_9\dots ,\mathcal{E}_\ell) + \frac{1}{n^{4}},
	\end{align*}	
	where the last inequality follows from the fact that for the event $\mathcal{E}_1\cap (\mathcal{E}'_1)^c$ to occur there have to be 
	two vertices
	at distance at least $\gamma \log  n$
	connected by a path in $B$; by the EDC property this only occurs with probability at most $n^{-4}$
	for large $\gamma$. Iterating this procedure for $\mathcal{E}_5,\mathcal{E}_9,\dots$ we get
	\begin{align}
	\label{eq:ssm:rec-event}
	\pi^{(1,\eta)} (\mathcal{E}_1,\mathcal{E}_5,\mathcal{E}_9\dots ,\mathcal{E}_\ell) 
	&\le \pi^{(1,\eta)} (\mathcal{E}_1',\mathcal{E}_5',\mathcal{E}_9'\dots ,\mathcal{E}_\ell') + \frac{\ell}{n^{4}}.
	\end{align}
	Let $Q = \bigcup_{i=0}^{(\ell-1)/4} Q_{4i+1}$. Monotonicity implies that
	\begin{align*}
	\pi^{(1,\eta)} (\mathcal{E}_1',\mathcal{E}_5',\mathcal{E}_9'\dots ,\mathcal{E}_\ell')
	&\le \pi^1 (\mathcal{E}_1',\mathcal{E}_5',\mathcal{E}_9'\dots ,\mathcal{E}_\ell' \mid E^c(Q) = 1)  = \prod_{i = 0}^{\frac{\ell-1}{4}}  \pi^1 (\mathcal{E}_{4i+1}' \mid E^c(Q)=1),
	\end{align*}
	where for the last equality we use that the events $\mathcal{E}_1',\mathcal{E}_5',\dots ,\mathcal{E}_\ell'$ are independent under the wired boundary condition when also conditioning on $\{E^c(Q)=1\}$ .
	We claim that there exists a constant $\rho \in (0,1)$ (independent of $n$) such that for all $i=0,\dots,\frac{\ell-1}{4}$
	\begin{equation}
	\label{eq:local:connection}
	\pi^1 (\mathcal{E}_{4i+1}' \mid E^c(Q)=1) \le 1 - \rho\,.
	\end{equation}
	To see this, let $j=4i+1$ and note that for for large enough $D>0$, by the EDC property and a union bound imply that there is no connection between any pair of vertices $(u,v)$ with $u\in S_j$ and $v\in S_{j+1}$ and $d(u,v)\geq D$ with probability $\Omega(1)$. At the same time, by forcing an adjacent $O(D)$ edges to be closed at a cost of $e^{-O(D)}$, we see that with $\Omega(1)$ probability, in fact no other pairs $(u,v)$ with $d(u,v)\leq D$ are connected either.
	Thus, (\ref{eq:local:connection}) holds for a suitable $\rho \in (0,1)$ and so
	\begin{equation}
	\label{eq:ssm:final-event}
	\pi^{(1,\eta)} (\mathcal{E}_1',\mathcal{E}_5',\mathcal{E}_9'\dots ,\mathcal{E}_\ell') \le (1-\rho)^{\frac{\ell+3}{4}} \le (1-\rho)^{\frac{t-1}{4}} \le \frac{1}{9n^3}\,,
	\end{equation}
	where the last inequality holds for sufficiently large $c_0$ since $t \ge  \frac{c_0}{8(\gamma+\alpha)} \log n $.
	
	Putting (\ref{eq:ssm:final-event}), (\ref{eq:ssm:rec-event}), (\ref{eq:ssm:no-rec-event}), (\ref{eq:ssm:removing-events}) and (\ref{eq:ssm:main-conditioning}) together we get
	$$
	\pi^{(1,\eta)} \left(\, \{e\} \stackrel{\eta}{\longleftrightarrow} \b_\west B  \,\right) \le \frac{2}{9n^{3}} + \frac{\ell}{n^4} \le \frac{1}{3n^3}\,,
	$$
	since $\ell = O(\log n)$. Analogously, we get 
	$
	\pi^{(1,\eta)} (\{e\} \stackrel{\eta}{\longleftrightarrow} \b_\east B  ) \le \frac{1}{3n^{3}},
	$ 
	and thus
	$$
	\pi^{(1,\eta)} \left(\, \{e\} \stackrel{\eta}\longleftrightarrow\b B  \,\right) \le \frac{1}{n^{3}}\,.
	$$
	
	Finally for the third case, suppose without loss of generality that $B$ 
	intersects $\bt\L_n$ and $\bl\L_n$, but not $\partial_\east \L_n$ or $\partial_\south \L_n$.
	A union bound implies that
	\begin{equation}
	\label{eq:ssm:third-case}
	\pi^{(1,\eta)} \left(\, \{e\} \stackrel{\eta}\longleftrightarrow \b B \setminus \b \Lambda  \,\right) \le   \pi^{(1,\eta)} \left(\, \{e\} \stackrel{\eta}{\longleftrightarrow}  \b_\south B  \,\right) +  \pi^{(1,\eta)} \left(\, \{e\} \stackrel{\eta}{\longleftrightarrow}\b_\east B \,\right),
	\end{equation}
	and each term in the right-hand side of (\ref{eq:ssm:third-case}) can bounded in the same way as $\pi^{(1,\xi)} (\, \{e\} \stackrel{\eta}{\longleftrightarrow} \b_\west B  \,)$ in the second case; thus, the result follows.
\end{proof}

\section{Slow mixing under worst-case boundary conditions}
\label{section:lb:intro}

In this section we show that there are (non-realizable) boundary conditions 
for the graph $(\L_n,E(\L_n))$ 
for which 
the FK-dynamics 
requires exponentially many steps to converge to stationarity.
In particular, we prove Theorem \ref{thm:lb:intro} from the introduction.

Theorem \ref{thm:lb:intro} is a corollary of a more general theorem we establish. This general theorem
enables the transferring of mixing time lower bounds for the FK-dynamics on arbitrary graphs 
to mixing time lower bounds for the FK-dynamics on $\L_n$, for suitably chosen boundary conditions.
The high level idea is that any  graph $G$ with fewer than $\frac n4$ edges can be ``embedded'' into a subset $L$ of the
boundary $\bt\L_n$ of $\L_n$ as a boundary condition we shall denote $\xi(G)$.
When $p$ is sufficiently small, the effect of the configuration on $\L_n\setminus L$ becomes negligible, and so the mixing time of the FK-dynamics on $\L_n$ with boundary condition $\xi(G)$ is primarily dictated by its restriction to the embedded graph~$G$.

We show first how to embed any graph $G=(V_G,E_G)$  into  a subset $L\subset \partial_\north \L_n$. For $m\leq \lfloor n/4\rfloor$, let
$$
L= L(m) = \{\llb 4i, 4i+1 \rrb:i=0,...,m-1\}\times \{n\} \subset \partial_\north\L_n
$$ 
with edge set $E(L)$ consisting of all edges in $E(\L_n)$ connecting vertices in $L$. 
\begin{definition}
	\label{def:embedding}
	Let $G=(V_G,E_G)$ be a graph with $|E_G| = m$ for $m\leq \lfloor n/4\rfloor$
	and let $L$ be
	as above.  
	We say a function $\phi:L \rightarrow V_G$
	is an embedding of $G$ into $(L,E(L))$ if	
	for every $\{u,v\} \in E_G$
	there exists a unique pair $x \in \phi^{-1}(u) \subseteq L$ and $y \in \phi^{-1}(v)  \subseteq L$,
	where $\phi^{-1}(u)$ and $\phi^{-1}(v)$
	denotes	the pre-image sets for $u$ and $v$ respectively,
	such that $\{x,y\} \in E(\L_{n})$.
\end{definition}

\noindent
Notice that every graph $G$ on $m \le \lfloor n/4\rfloor$ edges can be embedded into $L$ by  identifying each edge in $E_G$ with an edge in $E(L)$.

\begin{fact}\label{fact:encoding-bc}
	For every graph $G=(V_G,E_G)$ with $m\leq \lfloor \frac n4\rfloor$ edges, there exists an embedding
	of $G$ into~$(L,E(L))$.
\end{fact}

\noindent
Now let
$\xi(G)$ be the boundary condition on $\b \L_n$ defined by the partition:
$$
\{\{v\}: v \in \b\L_n\setminus L\} \cup \{\phi^{-1}(v): v \in V_G\}\,.
$$
In words, $\xi(G)$ is the boundary condition that is free everywhere except in the vertices of $L$ and where all the vertices in $L$ that are mapped by $\phi$ to the same vertex of $G$ are wired in $\xi(G)$.
We are now ready to state our main comparison result from which Theorem~\ref{thm:lb:intro} follows straightforwardly.

\begin{theorem}\label{thm:lb:general}
	Let $G=(V_G,E_G)$ be a graph and suppose there exist
	$q > 2$ and 
	$p = \lambda {|E_G|}^{-\alpha}$ with $\lambda >0$, $\alpha > 1/3$ 
	such that 
	$
	\gap(G) \le \exp(-\Omega(|V_G|)).
	$
	Then, as long as $n \ge 4 |E_G| \ge \varepsilon n$ for some $\varepsilon > 0$, with the same choice of $p$ and $q$, we have
	$$
	\gap(\L_{n}^{\xi(G)}) \le  {\e}^{-\Omega(|V_G|)}\,.
	$$
\end{theorem}

\begin{proof}[Proof of Theorem~\ref{thm:lb:intro}]
	It was established in \cite{GLP} that
	for every $q > 2$ and every $\ell$ sufficiently large, there exists an interval
	$(\lambda_s(q), \lambda_S(q))$
	such that if $p' = \lambda'/\ell$ with $\lambda' \in (\lambda_s(q), \lambda_S(q))$, then the spectral gap at parameters $(p',q)$ satisfies
	$$
	\gap(K_\ell) \leq \exp(-\Omega(\ell))\,,
	$$ 
	where $K_\ell$ denotes the complete graph on $\ell$ vertices.
	Therefore, for a fixed $p = \lambda n^{-\alpha}$ there exists a choice of $\ell = \Theta(n^{\alpha})$ such that  at parameters $(p,q)$, $\gap(K_\ell) \leq \exp(-\Omega(\ell))$.
	Since the number of edges in $K_\ell$ is $\Theta(n^{2\alpha})$, the result follows from Theorem~\ref{thm:lb:general} and (\ref{eq:prelim:gap}).
\end{proof}

\subsection{Main comparison inequality: proof of Theorem~\ref{thm:lb:general}}

We now turn to the proof of Theorem~\ref{thm:lb:general}.
A standard tool for bounding spectral gaps is construction of bottleneck sets with small conductance; see Section~\ref{subsection:prelim:dynamics}.
  It will be easier to do so for the following \emph{modified heat-bath} (MHB) dynamics, allowing us to isolate moves on $E(L)$, where we have embedded $G$, from those in $E^c(L)=E(\L_n) \setminus E(L)$.

\begin{definition}
	\label{def:MHB}
	Given an FK configuration $X_{t}$, one step of the MHB chain is given by:
	\begin{enumerate}
		\item Pick $e \in E(\L_n)$ uniformly at random;
		\item If both endpoints of $e$ lie in $L$, then perform a heat-bath update on $e$. That is, replace the configuration in $e$ with a sample from $\pi_{\L_n,p,q}^\xi(\cdot \mid X_t(E(\L_n) \setminus \{e\}))$;
		\item Otherwise, replace the configuration in $E(\L_n) \setminus E(L)$ with a sample from $\pi_{\L_n,p,q}^\xi(\cdot \mid X_t(E(L)))$.
	\end{enumerate}
\end{definition} 

\noindent
The MHB chain is clearly reversible with respect to $\pi_{\L_n,p,q}^\xi$.

Let $\gap_{\MHB}(\L_n^{\xi})$ denote the spectral gap of the MHB dynamics on $\L_n$ with 
boundary condition $\xi$ and
parameters $p$ and $q$. 
The following comparison inequality allows us to focus on finding upper bounds for the spectral gap of the MHB dynamics; its proof is deferred to Section~\ref{subsec:proof-auxiliary}.

\begin{lemma}
	\label{lemma:gap-comparison:sketch}
	For all $p \in (0,1)$, $q > 0$, $n \in \N$ and boundary condition $\xi$ for $\L_n$, we have
	$$
	\gap (\L_n^\xi)\leq \gap_{\MHB}(\L_n^\xi) \,.
	$$
\end{lemma}

With this in hand, we are now ready to prove Theorem~\ref{thm:lb:general}.

\begin{proof}[Proof of Theorem~\ref{thm:lb:general}] 

Recall from~\eqref{eq:prelim:conductance} that since, by assumption, the FK-dynamics on $G$ has $\gap (G)\leq \exp(-\Omega(|V_G|))$, there must exist  $S_\star \subset \Omega_G$ (the set of FK configurations on $G$) with $\pi_G(S_\star)\leq \frac 12$ such that 
\begin{align}\label{eq:s-star-conductance:sketch}
\Phi(S_\star) = \frac {Q_G(S_\star, S_\star^c)}{\pi_G(S_\star)} \leq {\e}^{-\Omega(|V_G|)}\,. 
\end{align}
Here $Q_G$ is the edge measure \eqref{eq:prelim:conductance-def} of the FK-dynamics on $G$
and $\pi_G = \pi_{G,p,q}$ denotes the random-cluster measure on $G$.
We will construct from this set $S_\star$, a set $A_\star \subset \Omega$, such that :
	\begin{align}
	\label{eq:lb:conductance}
	\Phi(A_\star) = \frac{\QMHB(A_\star, A_\star^c)} {\pi^{\xi(G)}({A_\star})} \leq {\e}^{-\Omega(|V_G|)}\,, \qquad \mbox{and}\qquad \Phi(A_\star^c) =\frac{\QMHB(A_\star, A_\star^c)} {\pi^{\xi(G)}({A_\star^c})} \leq {\e}^{-\Omega(|V_G|)}\,,
	\end{align}
where $\QMHB$ denotes the edge measure \eqref{eq:prelim:conductance-def} of the MHB dynamics on $\L_n^{\xi(G)}$ and $\pi^{\xi(G)} = \pi^{\xi(G)}_{\L_n,p,q}$. This implies Theorem~\ref{thm:lb:general} by combining it with~\eqref{eq:prelim:gap} and~\eqref{eq:prelim:conductance}.

Let $\{\xi_1,\dots,\xi_k\}$ be the partition of $L$ induced by $\xi(G)$.
For an FK configuration $\omega$ on $E^c(L)$,
we say that $\xi_i \stackrel{\omega}\longleftrightarrow \xi_j$
if there is an open path in $\omega$ from a vertex in $\xi_i$ to a vertex in $\xi_j$.
Let $\mathcal{S}^{\xi(G)}(\omega)$ be the set
$$
\mathcal{S}^{\xi(G)}(\omega) = \{\xi_i \in \xi(G):\, \xi_i \stackrel{\omega}{\longleftrightarrow} \xi_j\text{  for some } j \neq i,~j \in \{1,\dots,k\}\}\,,
$$
i.e., those $\xi_i$ that are connected to some $\xi_j$ in $\omega$. For $M \ge 0$, let
\begin{equation*}
	\label{eq:def:r}
	\mathcal{R}^{\xi(G)}(M) = \{\omega \in \{0,1\}^{E^c(L)}: |\mathcal{S}^{\xi(G)}(\omega)| \leq  M \}\,.
\end{equation*}
In words, $\mathcal{R}^{\xi(G)}(M)$ is the set of FK configurations on 
$E(\L_n)\setminus E(L)$ that connect at most $M$ elements of the partition $\{\xi_1,\dots,\xi_k\}$ of the vertex set $L$.

Observe that any configuration $\theta$ on $E(L)$ corresponds to a configuration on $E_G$.
Namely, if $\theta(\{u,v\}) = 1$, then the edge $\{\phi(u),\phi(v)\}$ is open in the configuration on $G$, where 
$\phi$ is the embedding of $G$ into $L$.
With a slight abuse of notation we may use $\theta$ also
for the corresponding configuration on
$\Omega_G$.
With this convention, let
\begin{align}\label{eq:A-M:sketch}
	A_{M} = \{\omega \in \Omega: \omega(E(L)) \in S_\star \,,\, \omega(E^c(L)) \in \mathcal{R}^{\xi(G)}(M)\}\,.
\end{align}  

	We show that if $M=\delta |V_G|$ for some $\delta > 0$ sufficiently small, and $n$ is taken to be large enough, we can take $A_\star = A_M$. This will follow from the following two claims. 
		
	\begin{claim}   
		\label{claim:am-prob-facts:sketch} The following are true of $S_\star$ and $A_M$ defined above:  
		\begin{enumerate}
			\item[(i)] $\pi^{\xi(G)}(A_M) \ge q^{-M}({1-{\e}^{-\Omega (M)}}) \pi_G(S_\star)$\,;
			\item[(ii)] $\pi^{\xi(G)} (A_M^c) \ge {\e}^{-O(M)}$\,.
		\end{enumerate}
	\end{claim}	
	
	\begin{claim}
		\label{claim:am-edge-messaure:sketch} The modified heat-bath dynamics satisfies
		$$
		\QMHB(A_M,A_M^c) \le \pi^{\xi(G)}(A_M)  {\e}^{-\Omega(M \log n)} + \frac{q^{2M+1}}{p} Q_G(S_\star,S_\star^c)\,.
		$$
	\end{claim}
	
	Dividing the bound from Claim \ref{claim:am-edge-messaure:sketch} by $\pi^{\xi(G)}(A_M)$ and using the bounds from Claim \ref{claim:am-prob-facts:sketch}, we see that
	\begin{align*}
		\frac {\QMHB(A_M, A_M^c)}{\pi^{\xi(G)}(A_M)} 
		\le {\e}^{-\Omega(M \log n)} + \frac{2q^{3M+1}}{p} \frac{Q_G(S_\star,S_\star^c)}{ \pi_G(S_\star)} 
		\le  {\e}^{-\Omega(M \log n)}  + {\e}^{O(M)}{\e}^{-\Omega(|V_G|)}\,.
	\end{align*}
	for sufficiently large $M$, where the last inequality follows from (\ref{eq:s-star-conductance:sketch}) and the facts that $M = \delta |V_G|$ and $p \ge \lambda {|V_G|}^{-2\alpha}$. Similarly, we get
	\begin{align*}
		\frac {\QMHB(A_M, A_M^c)}{\pi^{\xi(G)}(A_M^c)} \le {\e}^{-\Omega(M \log n)}  + {\e}^{O(M)}{\e}^{-\Omega(|V_G|)}\,.
	\end{align*}
	Then, since $M = \delta |V_G|$, for some $\delta>0$ sufficiently small we obtain~\eqref{eq:lb:conductance}. 
\end{proof}

\subsection{Proof of  auxiliary facts}\label{subsec:proof-auxiliary}
In this section we provide the proofs of Lemma~\ref{lemma:gap-comparison:sketch}, and Claims~\ref{claim:am-prob-facts:sketch} and~\ref{claim:am-edge-messaure:sketch}.

\begin{proof}[Proof of Lemma~\ref{lemma:gap-comparison:sketch}]
	For any $B \subset E(\L_n)$, let $P_B$ be the transition matrix corresponding to a heat-bath update on the entire set $B$.
	For $e \in E(\L_n)$, we use $P_e$ for $P_{\{e\}}$.
	Let $\PHB$ and $\PMHB$ be the transition matrices for the FK-dynamics and the MHB dynamics on $\L_n^\xi$, respectively. Let $A = E(\L_n) \setminus E(L)$. Then, 
	$$
	\PMHB = \frac{1}{|E(\L_n)|}\bigg(\sum_{e \in E(L)} P_e + \sum_{e \in A} P_A\bigg).
	$$
	For ease of notation, set $\pi = \pi_{\L_n,p,q}^\xi$. Then,
	for any $f,g \in \R^{|\Omega|}$, where 
	$\Omega$ denotes the set of FK configurations on $\L_n$, 
	let 
	$$
	\inner{f}{g}{\pi} = \sum_{\omega \in \Omega} f(\omega)g(\omega) \pi(\omega)\,.
	$$
	If we endow $\R^{|\Omega|}$ with the inner product $\inner{\cdot}{\cdot}{\pi}$, we obtain a Hilbert space denoted $L_2(\pi) = (\R^{|\Omega|},\inner{\cdot}{\cdot}{\pi})$.
	Recall, that if a matrix $P$ is reversible with respect to $\pi$, it defines a \textit{self-adjoint} operator from $L_2(\pi)$ to $L_2(\pi)$ via matrix vector multiplication, and thus
	$\inner{f}{Pg}{\pi} = \inner{P^*f}{g}{\pi} = \inner{Pf}{g}{\pi}$.
	
	Now, the matrix $\PMHB$ is positive semidefinite, since it is an average of positive semidefinite matrices. Hence, it is a standard fact (see, e.g., \cite{LP})
	that if	
	$\inner{f}{\PMHB f}{\pi} \le \inner{f}{\PHB f}{\pi}$ 
	for all $f \in \R^{|\Omega|}$, then 
	$
	\gap_{\MHB}(\L_n^\xi) \ge \gap(\L_n^\xi).
	$
	To show this, note that
	for $e \in A$, we have $P_A = P_e P_A P_e$. Thus, for all $f \in \R^{|\Omega|}$,
	\begin{align*}
	\inner{f}{P_Af}{\pi} = \inner{f}{P_e P_A P_e f}{\pi} = \inner{P_e^*f}{P_A P_e f}{\pi} &= \inner{P_e f}{P_A P_e f}{\pi} \\
	&\le \inner{P_e f}{P_e f}{\pi} = \inner{f}{P_e^2f}{\pi}  = \inner{f}{P_ef}{\pi}\,,
	\end{align*}
	where we used that $P_e = P_e^*$, $P_e = P_e^2$ and that $\inner{f}{P f}{\pi} \le \inner{f}{f}{\pi}$  
	for every $f \in \R^{|\Omega|}$ and every matrix $P$ reversible with respect to $\pi$.
	Then,
	$$
	\inner{f}{\PMHB f}{\pi} \le \frac{1}{|E(\L_n)|} \sum_{e \in E(\L_n)} \inner{f}{P_e f}{\pi} = \inner{f}{\PHB f}{\pi},
	$$
	and the result follows.
\end{proof}

Recall the notation of Lemma~\ref{thm:lb:general}.
 In order to compare the marginal distribution in $L$ to $\pi_G$, we need to bound the number of connections in $E^c(L)$ between different boundary components of $\xi(G)$ restricted to $L$: this will show that typical FK configurations on $E^c(L)$ do not have much influence
on the connectivities amongst $L$. This bound follows from the fact that $p=O(n^{-\alpha})$ for some $\alpha > 1/3$ and the (approximate) independence of connections between $\xi_i$ and $\xi_j$.
For the reminder of this section we set $\xi=\xi(G)$ for ease of notation.
 Claims~\ref{claim:am-prob-facts:sketch}--\ref{claim:am-edge-messaure:sketch} will then be seen as consequences of the following lemma.

\begin{lemma}
	\label{lemma:lb:gadget-tail:sketch}
	Let $q\geq 1$ and let $p =  \lambda n^{-\alpha}$ for $\lambda>0$ and $\alpha > 1/3$.
	Let $\xi$ be any boundary condition on $\bt\L_n$
	and let $\eta$ be any FK configuration on $E(L)$. 
	For every $M\ge 1$,
	\begin{align*}
		\pi_{\Lambda_n}^{\xi} \big(\mathcal{R}^{\xi}(M) \mid \eta \big)  \ge 1- \exp\left[-\Omega \left(M \log [M n^{3\alpha -1}]\right)\right]\,.
	\end{align*}
\end{lemma} 

\begin{proof}
	Let $Y$ be the random variable for the number 
	of vertices of $L$ 
	connected to at least one other vertex of $L$ 
	in an FK configuration on $E^c(L)$ 
	sampled from the distribution $\pi^{\xi,\eta}(\cdot)  = \pi_{\Lambda_n,p,q}^{\xi} (\cdot \mid \eta )$. It is sufficient to show 
	that
	$$
	\pi^{\xi,\eta}(Y \ge M) \le \exp\big({-\Omega \left(M \log [M n^{3\alpha -1}]\right)}\big).
	$$
	By classical comparison inequalities (see e.g.,~\cite{Grimmett}), when $q \ge 1$, no matter the boundary conditions $\xi,\eta$, the random-cluster measure on $E^c(L)$ is stochastically dominated by the independent bond percolation distribution on $E^c(L)$ with the same parameter $p$, which we denote by $\nu = \nu_{\L_n,p}$. (Recall that $\nu$ is the distribution on $\Omega$ that results from adding every edge in $E(\L_n)$ independently with probability $p$.)
	Hence, if $X$ is defined as $Y$ but for $\nu$ on $E^c(L)$, we get
	$$
	\pi^{\xi,\eta}(Y \ge M) \le \nu(X \ge M). 
	$$
	
	Consider the subgraph $\hat{\L} = (\L_n,E^c(L))$.
	Let $Z$ be the total number of vertices in $L$ that are connected in a configuration sampled from $\nu$ to another vertex at distance $3$ in $\hat{\L}$. 
	Then, since the distance between any two vertices in $L$ in $\hat \Lambda$ is at least 3, we see that $X \le Z$ and
	$$
	\nu(X \ge M) \le \nu (Z \ge M). 
	$$
	Thus, it suffices to establish a tail bound for $Z$. Enumerate the vertices of $L$ as $v_1,\dots,v_{2m}$ and let $Z_j$ 
	be the indicator random variable for the event that $v_j$ is connected to another vertex at distance $3$ in $\hat{\L}$.
	For $r=0,1,2,3$, split up $Z = \sum_r \hat Z_r$, where 
	$$
	\hat{Z}_r = \sum_{i \ge 0} Z_{4i+r+1}\,.
	$$
	We claim that there is some suitable $c>0$ such that under $\nu$, each $\hat{Z}_r$ is stochastically dominated by the binomial random variable $S \sim \textrm{Bin}(\lceil n/4\rceil,cn^{-3\alpha})$. This is because the random variables $Z_{4i+r+1}$ are jointly dominated by independent 
	Bernoulli random variables, $\textrm{Ber}(cn^{-3\alpha})$, for a suitable $c > 0$,
	since for every $k$ the events $\{Z_k = 1\}$ and $\{Z_{k+4} = 1\}$
	depend on disjoint sets of edges. The fact that the success probability of each one is at most $cn^{-3\alpha}$ follows from the fact that there are at most 16 choices of three adjacent edges from a vertex in $\bt\L_n$, and $ p  = \lambda n^{-\alpha}$.
	Hence, by Chernoff--Hoeffding inequality, for every $\delta>0$,
	$$
	\nu(\hat{Z}_r \ge \E_{\nu}[S]+ \delta \lceil n/4\rceil) \le  \exp\left[ - \frac n4 D\Big( cn^{-3\alpha}+\delta \,\| \,cn^{-3\alpha}\Big)\right].
	$$
	where $D(a \|  b)$ is the relative entropy between the Bernoulli random variables $\mbox{Ber}(a)$ and $\mbox{Ber}(b)$:
	\begin{align*}
		D(a\|b) = a\log (a/b) +(1-a)\log ((1-a)/(1-b))\,.
	\end{align*}
	Since $ \E_{\nu}[S] = O(n^{1-3\alpha})$,  $\Var(S) = O(n^{1-3\alpha})$, $\alpha>\frac 13$ and $M \ge 1$, it follows that for every $M \geq 1$,
	$$
	\nu(\hat{Z}_r \ge M/4) \le {\e}^{-\Omega \left(M \log [M n^{3\alpha -1}]\right)}\,.
	$$
	We have $Z = \hat{Z}_1 + \hat{Z}_2 + \hat{Z}_3 + \hat{Z}_4$, and so a union bound implies the matching bound for $\nu(Z\geq M)$. 
\end{proof}

Now recall the definitions of the set $S_\star$ and $A_M$ from~\eqref{eq:A-M:sketch}.

\begin{proof}[Proof of  Claim \ref{claim:am-prob-facts:sketch}]
	For part (i), observe that if $\omega$ is sampled from $\pi^\xi$, then 
	\begin{align*}
		\pi^\xi(A_M) &  = \pi^\xi \big(\omega (L) \in S_\star \mid  \omega(E^c(L)) \in \mathcal{R}^\xi(M)\big)  \pi^\xi \big(\omega(E^c(L)) \in \mathcal{R}^\xi(M)\big)\,. 
	\end{align*}
	By Lemma \ref{lemma:lb:gadget-tail:sketch},
	$$
	\pi^{\xi}(\omega(E^c(L)) \in \mathcal{R}^\xi(M)) \ge 1-{\e}^{-\Omega (M)}\,.
	$$
	Moreover, since 
	$$
	\pi^\xi(\omega (L) \in S_\star \mid  \omega(E^c(L) = 0)) =  \pi_G(S_\star)\,,
	$$
	it follows from Lemma \ref{lemma:simple-rc-bound} that
	$$
	\pi^{\xi}(\omega (L) \in S_\star \mid  \omega(E^c(L)) \in \mathcal{R}^\xi(M)) \ge q^{-M} {\pi_G(S_\star)}\,,
	$$
	and thus, 
	$$\pi^\xi(A_M)  \geq q^{-M}({1-{\e}^{-\Omega (M)}}) \pi_G(S_\star).$$ 
	Similarly for part (ii), we have 
	\begin{align*}
		\pi^\xi(A_M^c) 
		&\ge  \pi^\xi \big(\omega (L) \not\in S_\star \mid  \omega(E^c(L)) \in \mathcal{R}^\xi(M)\big) 	\pi^\xi(\omega(E^c(L)) \in \mathcal{R}^\xi(M)) \\
		&\ge  q^{-M} (1-{\e}^{-\Omega(M)}) {\pi_G(S_\star^c)}
	\end{align*}
	which is at least $e^{-O(M)}$ since $\pi_G(S_\star) \leq \frac 12$.
\end{proof}

\begin{proof}[Proof of Claim \ref{claim:am-edge-messaure:sketch}]
	Let $\PMHB$ be the transition matrix for the MHB dynamics and for ease of notation set $B = E^c(L)$. We have%
	\begin{align}
		\QMHB(A_M,A_M^c) 
		& \leq \sum_{\omega \in A_M} \sum_{\substack{\omega' \in \Omega: \\ \omega'(B) \notin \mathcal{R}^\xi(M)}} \pi^\xi(\omega) \PMHB(\omega, \omega') + \sum_{\omega \in A_M} \sum_{\substack{\omega' \in \Omega: \\ \omega'(L) \notin S_\star}} \pi^\xi(\omega) \PMHB(\omega, \omega')\,. \label{eq:edge-measure-decomp}
	\end{align}
	For the first term in (\ref{eq:edge-measure-decomp}), observe by definition of MHB dynamics, for every $\omega \in A_M$
	$$
	\sum_{\substack{\omega' \in \Omega: \\ \omega'(B) \notin \mathcal{R}^\xi(M)}} \PMHB(\omega,\omega') \le \sum_{\substack{\omega' \in \Omega: \\ \omega'(B) \notin \mathcal{R}^\xi(M)}} \pi^\xi\big(\omega'(B) \mid \omega(L)\big) \le {\e}^{-\Omega(M \log n)},
	$$
	where the last inequality follows from Lemma~\ref{lemma:lb:gadget-tail:sketch}.
	Hence,
	$$
	\sum_{\omega \in A_M} \sum_{\substack{\omega' \in \Omega: \\ \omega'(B) \notin \mathcal{R}^\xi(M)}} \pi ^\xi(\omega) \PMHB(\omega, \omega') \le \pi^{\xi}(A_M)  {\e}^{-\Omega(M \log n)}.
	$$
	
	For the second term in (\ref{eq:edge-measure-decomp}), 
	observe that $\omega \neq \omega'$
	and that $\omega$ and $\omega'$ can differ in at most one edge $e$; otherwise  $\PMHB(\omega, \omega') = 0$. 
	Thus, setting 
	$$
	p^+(q) = \max\left\{p,\frac{q(1-p)}{q(1-p)+p}\right\},\quad p^-(q) = \min\left\{1-p,\frac{p}{q(1-p)+p}\right\},
	$$
	we obtain
	\begin{align*}
		\PMHB(\omega, \omega') &= \frac{1}{|E(\L_n)|} \pi\big(\omega'(e)\mid \omega(E(\L_n) \setminus \{e\})\big) \le \frac{p^+(q)}{|E(\L_n)|} \le  \frac{p^+(q)|E_G|}{p^-(q)|E(\L_n)|} P_G(\omega (L), \omega' (L))\,.
	\end{align*}
	Then, since $|E_G|\leq |E(\L_n)|$ and $p^+(q)/p^-(q) \le q/p$,
	\begin{align*}
		\sum_{\omega \in A_M} \sum_{\substack{\omega' \in \Omega: \\ \omega'(L) \notin S_\star}} \pi^\xi (\omega) \PMHB(\omega, \omega') 
		&\le 	\frac{q}{p}\sum_{\omega \in A_M} \sum_{\substack{\omega' \in \Omega: \\ \omega'(L) \notin S_\star}} \pi^\xi (\omega) P_G(\omega (L), \omega' (L)) \notag\\
		&\le \frac{q}{p}\sum_{\theta \in \mathcal{R}^\xi(M)} \pi^\xi(\theta) \sum_{\omega_1  \in S_\star}\sum_{\omega_2\not\in S_\star} \pi^\xi(\omega_1 \mid \theta) P_G(\omega_1, \omega_2) \\ 
		& \le \frac{q}{p} \pi^\xi(\mathcal{R}^\xi(M)) \sum_{\omega_1  \in S_\star}\sum_{\omega_2\not\in S_\star} \max_{\theta \in \mathcal{R}^\xi(M)} \pi^\xi(\omega_1 \mid \theta) P_G(\omega_1, \omega_2),
	\end{align*}
	where $\omega_1,\omega_2$ 
	are FK configurations on $E(L)$ 
	and $\theta$ is an FK configuration on $B$.
	Lemma~\ref{lemma:simple-rc-bound} implies 	$$\max_{\theta \in \mathcal{R}^\xi(M)} \pi^\xi(\omega_1 \mid \theta) \le q^{2M} \pi_G(\omega_1),$$
	and so
	\begin{align*}
		\sum_{\omega \in A_M} \sum_{\substack{\omega' \in \Omega: \\ \omega'(L) \notin S_\star}} \pi^\xi(\omega) \PMHB(\omega, \omega') 
		&\le \frac{q^{2M+1}}{p} \sum_{\omega_1  \in S_\star}\sum_{\omega_2\not\in S_\star} \pi_G(\omega_1) P_G(\omega_1, \omega_2) = \frac{q^{2M+1}}{p} Q_G(S_\star,S_\star^c).
	\end{align*}
	Combining these two bounds, we get
	\begin{equation*}
	\QMHB(A_M, A_M^c) \le \pi^{\xi}(A_M)  {\e}^{-\Omega(M \log n)} + \frac{q^{2M+1}}{p} Q_G(S_\star,S_\star^c)\,. \hfill \qedhere
	\end{equation*}
\end{proof}

\end{document}